\documentclass{article}
\usepackage{geometry}
\usepackage{enumerate, url, alltt} 
\usepackage{ae,aecompl}
\usepackage{amssymb}
\usepackage{amsmath}
\usepackage{amsthm}
\usepackage{mathtools}
\usepackage[utf8]{inputenc}
\usepackage[english]{babel}
\usepackage{fullpage}
\usepackage{appendix}
\usepackage[usenames, dvipsnames]{color}
\usepackage{emptypage} %empties all empty pages (therefore removes pagenumbers on empty pages)
\usepackage{xparse} %for multi optional arguments in \newcommand
\usepackage[perpage]{footmisc} %reset footnote counter on each page
\usepackage{todonotes}

\usepackage{accents}
\usepackage{nicefrac}
\usepackage{multirow}
\usepackage[numbers,sort]{natbib}

\usepackage{hyperref}
\numberwithin{equation}{section}
\usepackage[noblocks]{authblk}

\title{Error bounds for approximations with deep ReLU neural networks in $W^{s,p}$ norms}

\author{Ingo G{\"u}hring\thanks{Institut für Mathematik, Technische Universität Berlin, E-mail: \tt{guehring@math.tu-berlin.de}} \qquad  \qquad 
Gitta Kutyniok\thanks{Institut für Mathematik, Technische Universität Berlin, E-mail: \tt{kutyniok@math.tu-berlin.de}} \qquad  \qquad
Philipp Petersen\thanks{Mathematical Institute, University of Oxford, E-mail: \tt{Philipp.Petersen@maths.ox.ac.uk}}}

\newcommand{\R}{\mathbb{R}}

\newcommand{\N}{\mathbb{N}}

\newcommand{\Z}{\mathbb{Z}}

\newcommand{\bmat}[2]{\left[ \begin{array}{#1} #2 \end{array} \right]}

\newcommand{\id}{\mathrm{Id}}

\newcommand{\supp}{\mathrm{supp~}}
\newcommand{\ksub}{\subset\subset}
\let\emptyset\varnothing
\DeclareDocumentCommand{\Wkp}{ O{k} O{p} O{\Omega}}{{W^{#1,#2}(#3)}}
\DeclareDocumentCommand{\tWkp}{ O{k} O{p} O{\Omega}}{{\wtilde W^{#1,#2}(#3)}}
\DeclareDocumentCommand{\Wkpm}{ O{k} O{p} O{\Omega} O{m}}{{W^{#1,#2}(#3;\,\R^{#4})}}
\DeclareDocumentCommand{\Wkpd}{ O{k} O{p} O{\Omega} m}{W^{#1,#2}(#3;d#4)}
\DeclareDocumentCommand{\Lp}{ O{p} O{\Omega}}{L^{#1}(#2)}
\DeclareDocumentCommand{\Lpd}{ O{p} O{\Omega}}{L^{#1}(#2;dx)}
\newcommand{\Winf}{W^{1,\infty}(\Omega)}
\DeclareDocumentCommand{\Lip}{ O{\Omega}}{\Cns[0][1][#1]}

\newcommand{\Li}{{L^{\infty}}}

\newcommand{\Linf}{{L^{\infty}(\Omega)}}
\newcommand{\Linfc}{{L^{\infty}(\cube^d)}}
\newcommand{\WinfMd}{W^{1,\infty}(\intervalo{-M}{M}^2;dxdy)}
\newcommand{\WinfM}{W^{1,\infty}(\intervalo{-M}{M}^2)}

\newcommand{\Winfnc}{W^{n,\infty}(\cube^d)}
\newcommand{\cube}{\intervalo{0}{1}}

\newcommand{\Ctest}{C_c^\infty(\Omega)}

\newcommand{\Phid}{\Phi^\id}
\newcommand{\As}{A_{\text{sum}}}
\newcommand{\Phis}{\Phi_{\text{sum}}}
\newcommand{\Phist}{\Phi_{\text{st}}}
\newcommand{\PhiPs}{\Phi_{P,\eps}}
\newcommand{\epsin}{\intervalo{0}{\nicefrac{1}{2}}}
\newcommand{\netLarge}{\Ncal_{\text{large}}}
\newcommand{\netSmall}{\Ncal_{\text{small}}}

\renewcommand{\epsilon}{\varepsilon}
\newcommand{\eps}{\varepsilon}
\newcommand{\alplus}{\hspace{5mm}}
\newcommand{\normabsspace}{\;}
\renewcommand{\rho}{\varrho}

\newcommand{\MNd}{\{0,\ldots,N\}^d}

\newcommand{\expl}[1]{\text{\scriptsize{(#1)}}}

\newcommand{\Arch}{{\mathcal{A}}}
\newcommand{\Ncal}{\mathcal{N}}

\newcommand\numberthis{\addtocounter{equation}{1}\tag{\theequation}}

\newcommand{\Lcal}{\mathcal{L}}

\newcommand{\wtilde}{\widetilde}

\newcommand{\appsq}{\Phi^{\text{sq}}}
\newcommand{\appsqdelta}{\Phi^{\text{sq}}_\delta}
\newcommand{\interv}{\left(\nicefrac{k}{2^m},\nicefrac{(k+1)}{2^m}\right)}
\newcommand{\interval}[2]{\left[#1,#2\right]}
\newcommand{\intervalo}[2]{\left(#1,#2\right)}
\newcommand{\ddx}[1]{\frac{\partial}{\partial #1}}

\newcommand{\Fndone}{\mathcal{F}_{n,d,\infty,1}}
\newcommand{\Fnd}{\mathcal{F}_{n,d,\infty,B}}
\newcommand{\Fndp}{\mathcal{F}_{n,d,p,B}}
\newcommand{\fraccoef}{s}

\newcommand{\apmult}{\widetilde\times}

\newcommand{\act}[1]{R_\rho(#1)}
\newcommand{\actbig}[1]{R_\rho\big(#1\big)}
\newcommand{\comp}{\subset\subset}
\newcommand{\starmax}{r^\star_{\text{max}}}
\DeclareDocumentCommand{\icouple}{O{B_0} O{B_1} O{\theta} O{p}}{\left(#1,#2\right)_{#3,#4}}
\newcommand{\ipair}{B_{\theta,p}}
\newcommand{\bigO}{\mathcal{O}}
\definecolor{green}{rgb}{0.2,0.6,0.15}

\DeclareMathOperator{\vcdim}{VCdim}

\DeclareMathOperator{\ran}{ran}

\DeclareMathOperator{\diam}{diam}
\DeclareMathOperator{\co}{conv}

\DeclareMathOperator{\esssup}{ess\, sup}

\DeclarePairedDelimiter{\ceil}{\lceil}{\rceil}
\DeclarePairedDelimiter{\floor}{\lfloor}{\rfloor}
\DeclarePairedDelimiterX{\norm}[1]{\lVert}{\rVert}{#1}
\DeclarePairedDelimiterX{\pabs}[1]{\lvert}{\rvert}{#1}
\theoremstyle{definition}
\newtheorem{definition}{Definition}[section]

\theoremstyle{plain}
\newtheorem{remark}[definition]{Remark}
\newtheorem*{remark*}{Remark}
\newtheorem{theorem}[definition]{Theorem}
\newtheorem{definitionlemma}[definition]{Definition/Lemma}

\newtheorem{lemma}[definition]{Lemma}

\newtheorem{corollary}[definition]{Corollary}
\newtheorem{proposition}[definition]{Proposition}

\begin{document}

\maketitle

\begin{abstract}
We analyze approximation rates of deep ReLU neural networks for Sobolev-regular functions with respect to weaker Sobolev norms. First, we construct, based on a calculus of ReLU networks, artificial neural networks with ReLU activation functions that achieve certain approximation rates. Second, we establish lower bounds for the approximation by ReLU neural networks for classes of Sobolev-regular functions. 
Our results extend recent advances in the approximation theory of ReLU networks to the regime that is most relevant for applications in the numerical analysis of partial differential equations.
\end{abstract}

\pagestyle{plain}

\setcounter{tocdepth}{2}
\pagenumbering{arabic}

\section{Introduction}
Powered by modern, highly parallelized hardware and the immense amount of accessible data, deep neural networks substantially outperform both traditional modelling approaches, e.g.\ those based on differential equations, and classical machine learning methods in a wide range of applications. Prominent areas of application include image classification \cite{huang2017densely, simonyan2014very, krizhevsky2012imagenet}, speech recognition  \cite{hinton2012deep, dahl2012context, wu2016stimulated}, and natural language processing \cite{young2018recent}.

Despite this overwhelming success in applications, a comprehensive mathematical explanation of this success is has not been found. However, a deep theoretical understanding of these techniques is crucial to design more efficient architectures and essential in safety-critical applications, such as autonomous driving.

Many attempts at unraveling the extreme efficiency of deep neural networks have been made in the context of approximation theory. The universal approximation theorem \cite{cybenko1989approximation, hornik1991approximation} establishes that every continuous function on a compact domain can be uniformly approximated by neural networks. More refined results that also set into relation the size and the approximation fidelity of neural networks have been reported for smooth activation functions in, for example, \cite{Barron1994, Mhaskar:1996:NNO:1362203.1362213, ShaCC2015provableAppDNN, bolcskei2017optimal}. 

In most applications, the activation functions are not smooth but taken to be the piecewise linear ReLU function. For these networks, approximation bounds for classes of smooth functions have been established in \cite{yarotsky2017error} and for piecewise smooth functions in \cite{petersen2017optimal}. Connections to sparse grids and the associated approximation rates were established in \cite{montanelli2019} and connections to linear finite element approximation were reported in \cite{he2018relu}.
It was also discovered that the depth of neural networks, i.e., the number of layers, crucially influences the approximation capabilities of these networks in the sense that deeper networks are more efficient approximators \cite{pmlr-v70-safran17a, mhaskar2017when, petersen2017optimal, cohen2016expressive}.

One of the applications of deep learning where a strong knowledge of the approximation capabilities of neural networks directly translates into quantifiable theoretical results appears when solving partial differential equations using deep learning techniques. Some notable advances in this direction have been made in \cite{lagaris1998artificial, weinan2018deep, han2018solving, sirignano2018dgm, weinan2017deep, elbrachter2018dnn, beck2018solving}. In this regard, not only the approximation fidelity with respect to standard Lebesgue spaces is of high interest, but also that with respect to Sobolev-type norms. First results in this direction were reported in \cite{schwab2018deep, opschoor2019deep}.

In this work, we give a comprehensive analysis of the approximation rates of deep neural networks of Sobolev-regular functions with respect to (fractional) Sobolev norms.

In the remaining part of this introduction, we briefly describe neural network architectures relevant to this work. Then, we present the setting that neural networks are typically used in. Afterwards, we give a practical and theoretical motivation for studying approximations of functions and their derivatives with neural networks. Finally, we state our results.

\subsection{Neural networks}
There exists a wide variety of neural network architectures, each adapted to specific tasks. One of the most commonly-used architectures is a \emph{feedforward} architecture also known as \emph{multi-layer perceptron} which implements a function as a sequence of affine-linear transformations followed by a componentwise application of a non-linear function, called \emph{activation function}. The length $L$ of the sequence is referred to as the number of \emph{layers} of the network and the input of a layer consists of the output of the preceding layer. We denote by $N_0\in\N$ the dimension of the input space and by $N_l\in\N$ the dimension of the output of the $l$-th layer for~$l=1,\ldots,L$. Thus, if $\rho:\R\to\R$ is the activation function and\[
	T_l:\R^{N_{l-1}}\to\R^{N_l}	
	\] is the affine-linear transformation of the $l$-th layer, then the computation in that layer can be described as
	\[
		f_l:\R^{N_{l-1}}\to\R^{N_l},\qquad x\mapsto \rho(T_l(x)) 
	\]
	for $l=1,\ldots,L-1$, and
	\[
		f_L:\R^{N_{L-1}}\to\R^{N_L},\qquad x\mapsto T_L(x).
	\]
	Note that $\rho$ acts componentwise in all but the last layer and in the last layer no activation function is applied.
%Thus, if $\rho:\R\to\R$ is the activation function, $N_0\in\N$ the dimension of the input space, $N_l\in\N$ and\[
%	T_l:\R^{N_{l-1}}\to\R^{N_l}	
%	\] the affine-linear transformation of the $l$-th layer, then the computation in that layer can be described as
%	\[
%	f_l:\R^{N_{l-1}}\to\R^{N_l},\qquad x\mapsto \rho(T_l(x)),	
%		\]
%		where $\rho$ is applied componentwise and $l=1,\ldots,L-1$. 
The parameters defining the affine-linear transformations $T_l$ are referred to as \emph{weights} and $\sum_{k=0}^L N_l$ is called the number of \emph{neurons} of the network, since each output coordinate of $f_l$ can be seen as a small computational unit, similar to neurons in the brain. If a network has $3$ or more layers, then it is usually called \emph{deep} and a $2$-layer network is called \emph{shallow}. The \emph{complexity of a neural network} is typically measured in the number of layers, weights, and neurons (see~\cite{anthony2009neural}). The term \emph{deep learning} refers to the subset of machine learning methods associated with deep neural networks. 

Recently, more general feedforward architectures which allow connections between non-neighboring layers, so-called \emph{skip connections}, have been shown to yield state-of-the-art results in object recognition tasks \cite{huang2017densely, he2016deep}. Here, the input of a layer consists of the output of \emph{all} preceding layers.  Note, that the case where no skip connections are allowed, is a special case of this more general architecture.

One of today's most widely-used activation functions is the \emph{Rectified Linear Unit (ReLU)} (see~\cite{lecun2015deep}), defined as
\[
	\rho:\R\to\R,\qquad x\mapsto\max\{0,x\}.
	\]
The popularity of the ReLU can be explained by a number of factors. It is cheap to compute, promotes sparsity in data representation \cite{bengio2013representation}, alleviates the problem of vanishing-gradients \cite{bengio2011deepsparse}, and thus yields better optimization properties. Moreover, the ReLU is the only commonly-used activation function such that the associated function spaces are not necessary non-closed \cite{petersen2018topological}. In this work, we will mostly focus on ReLU networks with a feedforward architecture allowing skip connections.

\subsection{Supervised learning with neural networks}
Typically, neural networks are applied in \emph{supervised learning} problems. The starting point is a dataset of input-output pairs $(x_i,f(x_i))_{i=1}^m$, called samples, where $f$ is in most cases an unknown function\footnote{We will later see a case where $f$ is in fact known.} with values only given at sample points $x_i$. As an example, $x_i$ can be thought of as an image and $f(x_i)$ as a vector of scores, where each score is the probability of a certain category, e.g. "dog" or "cat", being associated with $x_i$. During training, one then seeks to \emph{learn} the function $f$ by adapting the weights $w$ of a neural network $\Ncal$ such that the \emph{empirical loss}
\begin{equation}\label{eq:intro_empirical_loss}
	\frac{1}{m}\sum_{i=1}^m l\left(\Ncal(x_i|w),f(x_i)\right)	
\end{equation}
is minimized for some loss function $l$. Ultimately, one is interested in how well the learning algorithm performs on formerly unseen data points  $x$. The associated error is called the \emph{generalization error}. 

\subsection{Motivation: approximating functions and derivatives}\label{subsec:intro_motivation}
From a learning theory point of view, the task of estimating the generalization error decomposes into a statistical problem depending on the samples and an approximation theoretical problem, independent of the samples. An accessible introduction to learning theory from the perspective of approximation theory can be found in \cite{cucker2007learning}. The aim of this work is to study the simultaneous approximation of a function and its derivative with a neural network. There are multiple scenarios in which this is possible and useful. 

\begin{itemize}
	\item  In \cite{NIPS2017_7015}, the $j$-th order derivatives of $f$ are incorporated in the empirical loss function~\eqref{eq:intro_empirical_loss}, resulting in an empirical loss
\[
	\frac{1}{m}\sum_{i=1}^m\left[l\left(\Ncal(x_i|w),f(x_i)\right)+\sum_{j=1}^k l_j\left(D^j_x\Ncal(x_i|w),D^j_x f(x_i)\right)\right],
	\]
		which encourages the network to encode information about the derivatives of $f$ in its weights. The authors of \cite{NIPS2017_7015} call this method \emph{Sobolev training} and reported reduced generalization errors and better data-efficiency in a \emph{network compression} task (see \cite{hinton2015distilling}) and in application to synthetic gradients (see \cite{jaderberg2017decoupled}). In case of network compression, the approximated function $f$ is a function realized by a possibly very large neural network $\netLarge(\cdot|w)$, that has been trained for some supervised learning task and is learnt by a smaller network $\netSmall$. In contrast to usual supervised learning settings, the approximated function $f(\cdot) =\netLarge(\cdot|w)$ is known and the derivatives can be computed.
	\item Motivated by the performance of deep learning-based solutions in classical machine learning tasks and, in particular, by their ability to overcome the curse of dimension, neural networks are now also applied for the approximative solution of \emph{partial differential equations (PDEs)} (see \cite{lagaris1998artificial, weinan2018deep, han2018solving, sirignano2018dgm}). 

		In \cite{sirignano2018dgm} the authors present their \emph{deep Galerkin method} for approximating solutions of high-di\-men\-sion\-al quasilinear parabolic PDEs. For this, a functional $J(f)$ encoding the differential operator, boundary conditions, and initial conditions is introduced. A neural network $\Ncal_{\text{PDE}}$ with weights $w$ is then trained to minimize the functional $J(\Ncal_{\text{PDE}}(w))$. This is done by a discretization and randomly sampling spatial points.
\end{itemize}
The theoretical foundation for approximating a function and higher-order derivatives with a neural network was already given in a less known version of the universal approximation theorem by Hornik in \cite[Theorem~3]{hornik1991approximation}. In particular, it was shown that if the activation function $\rho$ is $k$-times continuously differentiable, non-constant, and bounded, then any $k$-times continuously differentiable function $f$ and its derivatives up to order $k$ can be uniformly approximated by a shallow neural network on compact sets. Note though that the conditions on the activation function are very restrictive and that, for example, the ReLU is not included in the above result. However, in \cite{NIPS2017_7015}, it was shown that the theorem also holds for shallow ReLU networks if $k=1$. Theorem~3 in \cite{hornik1991approximation} was also used in \cite{sirignano2018dgm} to show the existence of a shallow network approximating solutions of the PDEs considered in this paper.

An important aspect, that is untouched by the previous approximation results is how the complexity of a network and, in particular, its depth relates to its approximation properties.

\subsection{Our contribution}
In Theorem~1 in~\cite{yarotsky2017error}, Yarotsky showed upper complexity bounds for approximations in $L^\infty$ norm of functions from the Sobolev space $\Winfnc$ with neural networks for continuous piecewise linear activation functions with a finite number of breakpoints. Precisely, it is shown there, that for any $\eps>0$ there exists a neural network $\Ncal_{\eps}$ with at most $c\cdot\log_2(\nicefrac{1}{\eps})$ layers and at most $c \cdot\eps^{-d/n}\log_2(\nicefrac{1}{\eps})$ weights and neurons such that for any $f\in\Winfnc$ with $\norm{f}_{\Winfnc}\leq 1$ there is a choice of weights $w_f$ with  
\[
	\norm{\Ncal_{\eps}(\cdot|w_f)-f(\cdot)}_{\Linfc}\leq \eps,
	\]
where $c$ is a constant depending on $d$ and $n$. 

Furthermore, under the assumption of a possibly discontinuous weight selection the number of weights needed by a neural network $\Ncal_\eps$ to be able to realize an $\eps$-approximation in $L^\infty$ norm for any $f\in\Winfnc$ with $\norm{f}_{\Winfnc}\leq 1$ is shown to be lower bounded by $c'\cdot\eps^{d/(2n)}$ in \cite[Theorem~4~a)]{yarotsky2017error}. The constant $c'$ depends on $d$ and $n$.

We show for the same set of activation functions that the approximation can also be done with respect to higher-order Sobolev norms with arbitrary $1\leq p\leq\infty$ and that there is a trade-off between the regularity used in the approximation norm and the regularity used in the complexity bounds. Specifically, we show that for any approximation accuracy $\eps>0$ and regularity $0\leq s\leq 1$, there is a ReLU neural network $\Ncal_{\eps}$ with at most $c\cdot\log_2(\eps^{-n/(n-s)})$ layers and $c\cdot\eps^{-d/(n-s)}\cdot\log_2(\eps^{-n/(n-s)})$ weights and neurons such that and any $f\in\Wkp[n][p][\cube^d]$ with $\norm{f}_{\Wkp[n][p][\cube^d]}\leq B$ there is a choice of weights $w_f$ with
	\[
		\norm{\Ncal_{\eps}(\cdot | w_f) - f(\cdot)}_{\Wkp[s][p][\cube^d]}\leq \eps,
		\]
		where $c$ is a constant depending on $d,n,p,s$, and $B$. In the boundary case $s=0$ and $p=\infty$ our results corresponds to the theorem shown by Yarotsky and for $s=1$ and $p=\infty$ the function $\Ncal_{\eps}(\cdot|w_f)$ and its weak gradient uniformly approximate $f$ and the weak gradient of $f$, respectively. For non-integer $s$ the function $\Ncal_{\eps}(\cdot | w_f)$ approximates the function $f$ and its fractional derivative of order $s$ approximates the fractional derivative of order $s$ of $f$. The case $0\leq s\leq 1$ and $p=\infty$ was already shown by one of the authors, I.\ G{\"u}hring, in his Master thesis~\cite{masterthesis}.

Moreover, analogously to \cite[Theorem~4~a)]{yarotsky2017error} we establish lower bounds where the same regularity-complexity trade-off can be observed. We show that if a neural network $\Ncal_\eps$ is able to realize an $\eps$-approximation in $W^{1,\infty}$ norm for any $f\in\Winfnc$ with $\norm{f}_{\Winfnc}\leq 1$, then the number of weights of $\Ncal_\eps$ is lower bounded by $c' \cdot \eps^{\nicefrac{-d}{2(n-1)}}$. Here, $c'$ is a constant depending on $d$ and $n$.

\subsection{Outline}
As a preparation, we start by introducing notation and some definitions in~Section~\ref{sec:prelim}. In~Section~\ref{sec:network}, we rigorously define a neural network and its architecture in mathematical terms and develop a network calculus. In~Section~\ref{sec:sobolev}, we briefly introduce (fractional) Sobolev spaces. In~Section~\ref{sec:main}, we present our results which will be discussed in~Section~\ref{sec:discussion}. To allow a concise presentation of our results most proofs are given in the appendix: Appendix~\ref{app:interpolation} and~\ref{app:sobolev} contain the necessary preparation for the proof of the results from~Section~\ref{subsec:upper_bounds} which is given in~Appendix~\ref{app:upper_bounds}. The results from~Section~\ref{subsec:lower_bounds} are proven in~Appendix~\ref{app:lower_bounds}.

\subsection{Notation}\label{sec:prelim}
We set $\N:=\{1,2,\ldots\}$ and $\N_0:= \N\cup \{0\}$. For $k\in\N_0$ the set $\N_{\geq k}:=\{k,k+1,\ldots\}$ consists of all natural numbers larger than or equal to $k$. For a set $A$ we denote its cardinality by $\lvert A\rvert\in\N\cup\{\infty\}$. If $x\in\R$, then we write $\ceil{x}:=\min\{k\in\Z:k\geq x\}$ where $\Z$ is the set of integers.

If $d\in\N$ and $\norm{\cdot}$ is a norm on $\R^d$, then we denote for $x\in\R^d$ and $r>0$ by $B_{r,\norm{\cdot}}(x)$ the open ball around $x$ in $\R^d$ with radius $r$, where the distance is measured in $\norm{\cdot}$. By $\pabs{x}$ we denote the euclidean norm of $x$ and by $\norm{x}_{\ell^\infty}$ the maximum norm.

We endow $\R^d$ with the standard topology and for $A\subset \R^d$ we denote by $\overline{A}$ the closure of $A$ and by $\partial A$ the boundary of $A$. For the convex hull of $A$ we write $\co A$. The diameter of a non-empty set $A\subset\R^d$ is always taken with respect to the euclidean distance, i.e. $\diam A:=\diam_{\pabs{\cdot}}A:=\sup_{x,y\in A}\pabs{x-y}$.
%	The \emph{characteristic function} of the set $A$ is denoted by 
%	\[
%		\chi_A:\R^d\to\R\quad\text{with}\quad \chi_A(x)=\begin{cases}
%			1,&\text{if }x\in A\\
%			0,&\text{if }x\notin A.
%		\end{cases}
%		\]
If $A,B\subset\R^d$, then we write $A\ksub B$ if $\overline{A}$ is compact in $B$.

For $d_1,d_2\in\N$ and a matrix $A\in\R^{d_1,d_2}$ the number of nonzero entries of $A$ is counted by $\norm{\cdot}_{\ell^0}$, i.e.\
	\[
		\norm{A}_{\ell^0}:=\pabs*{\{(i,j): A_{i,j}\neq 0\}}.
	\]
If $d_1,d_2,d_3\in\N$ and $A\in\R^{d_1,d_2},B\in\R^{d_1,d_3}$, then we use the common block matrix notation and write for the horizontal concatenation of $A$ and $B$ 
\[
\bmat{c c}{A & B} \in\R^{d_1,d_2+d_3}\quad\text{or}\quad \bmat{c|c}{A & B}\in\R^{d_1,d_2+d_3}
%\bmat{c|c}{A & B}\in\R^{d_1,d_2+d_3} \text{ or } \bmat{c}{A\\[1em] B}\in\R^{d_1+d_3,d_2},	
\]
where the second notation is used if a stronger delineation between different blocks is necessary. A similar notation is used for the vertical concatenation of $A\in\R^{d_1,d_2}$ and $B\in\R^{d_3,d_2}$.

Of course, the same notation also applies to (block) vectors.

For a function $f:X\to\R$ we denote by $\supp f$ the support of $f$. If $f:X\to Y$ and $g:Y\to Z$ are two functions, then we write $g\circ f:X\to Z$ for their composition.

We use the usual \emph{multiindex} notation, i.e.\ for $\alpha\in\N_0^d$ we write $\pabs{\alpha}:=\alpha_1+\ldots + \alpha_d$ and $\alpha!:=\alpha_1!\cdot\ldots\cdot \alpha_d!$. Moreover, if $x\in\R^d$, then we have
	\[
		x^\alpha:=\prod_{i=1}^d x_i^{\alpha_i}.
		\]

Let $\Omega\subset\R^d$ be open. For $n\in\N_0 \cup \{\infty\}$, we denote by $C^n(\Omega)$ the set of $n$ times continuously differentiable functions on $\Omega$. The space of \emph{test functions} is 
\[
	\Ctest:=\left\{f\in C^\infty(\Omega)\;\vert\; \supp f\ksub\Omega \right\}.
	\]
 For $f\in C^n(\Omega)$ and $\alpha\in\N_0^d$ with $\pabs{\alpha}\leq n$ we write
\[
	D^\alpha f:=\frac{\partial^{\pabs{\alpha}}f}{\partial x_1^{\alpha_1} \partial x_2^{\alpha_2}\cdots \partial x_d^{\alpha_d}}.
\]

We say that a function $f:\Omega\to\R^m$ for $\Omega\subset\R^d$ is \emph{Lipschitz continuous} if there is a constant $L>0$ such that 
\[
	\pabs{f(x)-f(y)}\leq L \pabs{x-y}
	\]
	for all $x,y\in\Omega$.  If we want to emphasize the constant $L$, then we say that $f$ is \emph{$L$-Lipschitz}. 

	If $X$ is a linear space and $\norm{\cdot}_1$ and $\norm{\cdot}_2$ are two norms on $X$, then we say that $\norm{\cdot}_1$ and $\norm{\cdot}_2$ are \emph{equivalent} if there exist constants $C_1,C_2>0$ such that 
	\[
		C_1\norm{x}_1	\leq \norm{x}_2\leq C_2 \norm{x}_1\quad\text{for all }x\in X.
		\] 
		For two normed linear spaces $X,Y$ we denote by $\Lcal(X,Y)$ the set of bounded linear operators mapping $X$ to $Y$ and for $T\in\Lcal(X,Y)$ the induced operator norm of $T$ is denoted by 
	\[
		\norm{T}_{\Lcal(X,Y)}:=\sup\left\{\norm{Tx}_Y:x\in X,\;\norm{x}_X\leq 1\right\}.
		\]

		Let $a>0$, then we say for two functions $f:(0,a)\to [0,\infty)$ and $g:(0,a)\to[0,\infty)$ that $f(\eps)$ is in $\bigO(g(\eps))$ if there exists $0<\delta<a$ and $C>0$ such that $f(\eps)\leq C g(\eps)$ for all $\eps\in (0,\delta)$.

\section{Neural networks}\label{sec:network}
In this section, we introduce the notion of neural networks used in this paper. As in \cite{yarotsky2017error}, we will consider a general type of feedforward architecture that also allows for connections of neurons in non-neighboring layers. It can be seen, though, that any function realized by such a network can also be realized by a network with a more restrictive feedforward architecture where only neurons from neighboring layers can be connected (Lemma \ref{lemma:neural_nets_dense_to_standard}). As in \cite{petersen2017optimal}, we draw a distinction between the neural network and the function that the network realizes. This gives us the possibility to develop a network calculus in the spirit of~\cite[Chapter~2]{petersen2017optimal}. As in that paper we will introduce the notion of network concatenation and parallelization.

The following definition is similar to \cite[Definition 2.1]{petersen2017optimal}, where the difference is that we also allow connections between non-neighboring layers.
\begin{definition}\label{def:nn}
	Let $d,L\in\N$. A \emph{neural network $\Phi$ with input dimension $d$ and $L$ layers} is a sequence of matrix-vector tuples
	\[
		\Phi=((A_1,b_1),(A_2,b_2),\dots,(A_L,b_L)),
		\]
		where $N_0=d$ and $N_1,\ldots,N_L\in\N$, and where each $A_l$ is an $N_l\times \sum_{k=0}^{l-1} N_k$ matrix, and $b_l\in\R^{N_l}$.

	If $\Phi$ is a neural network as above, and if $\rho:\R\to\R$ is arbitrary, then we define the associated \emph{realization of $\Phi$ with activation function $\rho$} as the map $\act{\Phi}:\R^d\to\R^{N_L}$ such that 
		\[
			\act{\Phi}(x)=x_L,
			\]
			where $x_L$ results from the following scheme:
			\begin{align*}
				&x_0:=x,\\
				&x_l:=\rho\left(A_l\bmat{c|c|c}{x_0^T & \ldots &  x_{l-1}^T}^T+b_l\right), \quad\text{for }l=1,\ldots L-1,\\
				&x_L:=A_L \bmat{c|c|c}{x_0^T & \ldots &  x_{L-1}^T}^T+b_L,
			\end{align*}
			where $\rho$ acts componentwise, i.e., $\rho(y)=[\rho(y^1),\ldots,\rho(y^m)]$ for $y=[y^1,\ldots,y^m]\in\R^m$. We sometimes write $A_l$ in block-matrix form as
			\[
				A_l=\bmat{c|c|c}{
					A_{l,x_0} & \ldots & A_{l,x_{l-1}}
					},
				\]
				where $A_{l,x_k}$ is an $N_l \times N_k$ matrix for $k=0,\ldots,l-1$ and $l=1,\ldots,L$. Then
			\begin{align*}
				&x_l=\rho\left(A_{l,x_0} x_0+\ldots+ A_{l,x_{l-1}} x_{l-1}+b_l\right), \quad\text{for }l=1,\ldots L-1,\\
				&x_L=A_{L,x_0} x_0+\ldots+ A_{L,x_{L-1}} x_{L-1}+b_L.
			\end{align*}

We call $N(\Phi):=d+\sum_{j=1}^L N_j$ the \emph{number of neurons} of the network $\Phi$, $L=L(\Phi)$ the \emph{number of layers}, and finally $M(\Phi):=\sum_{j=1}^L (\norm{A_j}_{\ell^0}+\norm{b_j}_{\ell^0})$ denotes the number of nonzero entries of all $A_l,b_l$ which we call the \emph{number of weights} of $\Phi$. Moreover, we refer to $N_L$ as the \emph{dimension of the output layer} of $\Phi$.
\end{definition}

We will now define the class of neural networks where only connections between neighboring layers are allowed. Networks without skip connections are a special case of the networks of Definition~\ref{def:nn}. Since they are more frequently used in the literature, we coin such networks \emph{standard neural networks}.
\begin{definition}\label{def:nn_standard_nn}
	If $\Phi=((A_1,b_1),(A_2,b_2),\dots,(A_L,b_L))$ is a neural network as above and we have for 
			\[
				A_l=\bmat{c|c|c}{
					A_{l,x_0} & \ldots & A_{l,x_{l-1}}
					},
				\] that $A_{l,x_i}=0$ for $l=1,\ldots,L$ and $i=0,\ldots,l-2$, then
				we call $\Phi$ a \emph{standard neural network}. The computation scheme then reduces to the following:
			\begin{align*}
				&x_0:=x,\\
				&x_l:=\rho\left(A_{l,x_{l-1}}x_{l-1}+b_l\right), \quad\text{for }l=1,\ldots L-1,\\
				&x_L:=A_{L,x_{L-1}}  x_{L-1}+b_L.
			\end{align*}
\end{definition}

In practice, before training a neural network, i.e.\ adjusting the weights of the network, one has to decide which network architecture to use. The following definition will clarify the notion of a network architecture.

\begin{definition}
	Let $d,L\in\N$. A \emph{neural network architecture $\Arch$ with input dimension $d$ and $L$ layers} is a sequence of matrix-vector tuples
	\[
		\Arch=((A_1,b_1),(A_2,b_2),\dots,(A_L,b_L)),
		\]
		where $N_0=d$ and $N_1,\ldots,N_L\in\N$, and where each $A_l$ is an $N_l\times \sum_{k=0}^{l-1} N_k$ matrix, and $b_l$ a vector of length $N_l$ with elements in $\{0,1\}$. So a neural network architecture is simply a neural network with binary weights.
		
We call $N(\Arch):=d+\sum_{j=1}^L N_j$ the \emph{number of neurons} of the architecture $\Arch$, $L=L(\Arch)$ the \emph{number of layers}, and finally $M(\Arch):=\sum_{j=1}^L (\norm{A_j}_{\ell^0}+\norm{b_j}_{\ell^0})$ denotes the number of nonzero entries of all $A_l,b_l$ which we call the \emph{number of weights} of $\Arch$. Moreover, we refer to $N_L$ as the \emph{dimension of the output layer} of $\Arch$.

We say that a \emph{neural network $\Phi=((A'_1,b'_1),(A'_2,b'_2),\dots,(A'_L,b'_L))$ with input dimension $d$ and $L$ layers has architecture $\Arch$} if 
	\begin{enumerate}[(i)]
		\item $N_l(\Phi)=N_l$ for all $l=1,\ldots,L$ and
		\item $[A'_l]_{i,j}\neq 0$ implies $[A_l]_{i,j}\neq 0$ for all $l=1,\ldots,L$ with $i=1,\ldots,N_l$ and $j=1,\ldots,\sum_{k=0}^{l-1}N_k$.
	\end{enumerate}

	In the spirit of Definition~\ref{def:nn_standard_nn} we say that $\Arch$ is a \emph{standard neural network architecture} if all weights connecting neurons from non-neighboring layers are zero.

	By fixing an ordering of the nonzero entries in $A_l$ for $l=1,\ldots,L$ we can uniquely identify a vector $w\in\R^{M(\Arch)}$ with a (standard) neural network that has (standard) architecture $\Arch$. We denote this (standard) neural network by $\Arch(w)$.
\end{definition}
Note that in general an architecture of a neural network does not need to be unique.

Moreover, we define one of the in practice commonly used (see \cite{lecun2015deep}) activation functions. 

\begin{definition}\label{def:nn_ReLU}
	The function
	\[
		\rho:\R\to\R,\quad x\mapsto\max(0,x),
		\]
		is called \emph{ReLU} (Rectified Linear Unit) \emph{activation function}.
\end{definition}
% ################ Proposition 1 yarotsky ########
%\begin{lemma}
%	 For any continuous piecewise linear activation function $\hat\rho$ with $M$ breakpoints, where $1\leq M<\infty$ there exist absolute constants $C_1=C_1(M)>0$ and $C_2=C_2(M)>0$ with the following properties:
%
%	For any neural network $\Phi$ with $d$-dimensional input, with $L=L(\Phi)$ layers, $N=N(\Phi)$ neurons and $M=M(\Phi)$ nonzero weights there is a neural network $\Phi'$ with $d$-dimensional input , with $L$ layers, at most $C_1 N$ neurons and at most $C_2 M$ nonzero weights such that
%	\[
%		R_{\hat\rho}(\Phi)(x)=\act{\Phi'}(x)\quad\text{for all }x\in\R^d,
%		\]
%		 where $\rho$ is the ReLU activation function.
%\end{lemma}
%
It can be shown, see~\cite[Proposition 1]{yarotsky2017error}, that in terms of approximation and upper complexity bounds for networks with continuous and piecewise linear activation functions it suffices to focus on the ReLU.

\subsection{Network concatenation}
In this subsection we introduce the notion of concatenation of two networks which allows us to realize compositions of network realizations. In detail, if $\Phi^1$ and $\Phi^2$ are two networks, then we are interested in a network $\Phi^1\bullet\Phi^2$ such that $\act{\Phi^1\bullet\Phi^2}(x)=\act{\Phi^1}\circ\act{\Phi^2}(x)$ for all $x\in\R^d$. We will start by constructing a network $\Phi^1\bullet\Phi^2$ for general activation functions $\rho$. In the case where the activation function is the ReLU we use the fact that it is possible to find a network which realizes the identity function and construct a so called \emph{sparse concatenation} which gives us control over the complexity of the resulting network.

\begin{definitionlemma}
	Let $L_1,L_2\in\N$ and let 
	\[
		\Phi^1=((A^1_1,b^1_1),\dots,(A^1_L,b^1_L)), \quad\Phi^2=((A^2_1,b^2_1),\dots,(A^2_L,b^2_L))
		\]
		be two neural networks such that the input layer of $\Phi^1$ has the same dimension as the output layer of $\Phi^2$. Then, $\Phi^1\bullet\Phi^2$ denotes the following $L_1+L_2-1$ layer network:
	\[
		\Phi^1\bullet\Phi^2:=\left((A^2_1,b^2_1),\dots,(A^2_{L_2-1},b^2_{L_2-1}),(\tilde{A}_1^1,\tilde{b}_1^1),\ldots,(\tilde{A}_{L_1}^1,\tilde{b}_{L_1}^1)\right),
		\]
		where 
		\[
			\tilde{A}_l^1:=\bmat{c|c|c|c}{A_{l,x_0}^1 A_{L_2}^2 & A_{l,x_1}^1 & \dots & A_{l,x_{l-1}}^1}\quad\text{and}\quad\tilde{b}_l^1:=A_{l,x_0}^1 b_{L_2}^2+b_l^1,
			\]
		for $l=1,\ldots,L_1$. We call $\Phi^1\bullet\Phi^2$ the \emph{concatenation of $\Phi^1$ and $\Phi^2$}.

Then, $\act{\Phi^1\bullet\Phi^2}=\act{\Phi^1}\circ\act{\Phi^2}$.
\end{definitionlemma}

\begin{proof}
	This can be verified with a direct computation.
\end{proof}

Next, we adapt \cite[Lemma 2.3]{petersen2017optimal} to the case where skip connections are allowed and define a network which realizes an identity function if the activation function is the ReLU.
\begin{lemma}\label{lemma:neual_Id}
	Let $\rho:\R\to\R$ be the ReLU and $d\in\N$. We define
	\[
		\Phid:=((A_1,b_1),(A_2,b_2))	
		\]
		with
		\[
			A_1:=
			\bmat{c}{
				\id_{\R^d}\\[1em]
				-\id_{\R^d}
			},\quad
			b_1:=0,\quad
			A_2:=\bmat{c|c|c}{0_{\R^{d,d}} & \id_{\R^d} & -\id_{\R^d}},\quad
			b_2:=0.
			\]
			Then $\act{\Phid}=\id_{\R^d}$.
\end{lemma}

Using the identity network $\Phid$ from the previous lemma we can now (as in \cite[Definition~2.5]{petersen2017optimal}) define an alternative network concatenation which allows us to control the complexity of the resulting network for the case of the ReLU activation function.
\begin{definition}
	Let $\rho:\R\to\R$ be the ReLU, let $L_1,L_2\in\N$, and let 
	\[
		\Phi^1=((A^1_1,b^1_1),\dots,(A^1_{L_1},b^1_{L_1})),\quad \Phi^2=((A^2_1,b^2_1),\dots,(A^2_{L_2},b^2_{L_2}))
		\]be two neural networks such that the input dimension of $\Phi^1$ is equal to the output dimension of $\Phi^2$. Let $\Phid$ be as in Lemma \ref{lemma:neual_Id}.
	
	Then the \emph{sparse concatenation of $\Phi^1$ and $\Phi^2$} is defined as
	\[
		\Phi^1\odot\Phi^2:=\Phi^1\bullet\Phid\bullet\Phi^2.
		\]
\end{definition}

\begin{remark}\label{remark:neural_sparse_concat}
		We have
		\[
			\Phi^1\odot\Phi^2=\left((A_1^2,b_1^2),\ldots,(A_{L_2-1}^2,b_{L_2-1}^2),\left(
			\bmat{c}{
				A_{L_2}^2\\[1em]
				-A_{L_2}^2
			},
			\bmat{c}{
				b_{L_2}^2\\[1em]
				-b_{L_2}^2
			}\right),
			(\hat{A}_1^1,b_1^1),\ldots (\hat{A}_{L_1}^1,b_{L_1}^1)
			\right),
			\]
			where 
			\[
				\hat{A}_l^1:=\bmat{c|c|c|c|c|c}{0_{N_l^1,\sum_{k=0}^{L_2-1}N_k^2} & A_{l,x_0}^1 & -A_{l,x_0}^1 & A_{l,x_1}^1 & \ldots & A_{l,x_{l-1}}^1},
				\]
				for $l=1,\ldots,L_1$. Furthermore, it holds that $L(\Phi^1\odot\Phi^2)=L_1+L_2$, and $M(\Phi^1\odot\Phi^2)\leq 2M(\Phi^1)+2M(\Phi^2)$ and $N(\Phi^1\odot\Phi^2)\leq 2N(\Phi^1)+2N(\Phi^2)$.
\end{remark}

\subsection{Network parallelization}
In some cases it is useful to combine several networks into one larger network. In particular, if $\Phi^1$ and $\Phi^2$ are two networks, then we are looking for a network $P(\Phi^1,\Phi^2)$ such that $\act{P(\Phi^1,\Phi^2)}(x)=(\act{\Phi^1}(x),\act{\Phi^2}(x))$ for all $x\in\R^d$.

\begin{definitionlemma}\label{deflemma:parallelization}
	Let $L_1,L_2\in\N$ and let 
	\[
	\Phi^1=((A^1_1,b^1_1),\dots,(A^1_{L_1},b^1_{L_1})), \quad\Phi^2=((A^2_1,b^2_1),\dots,(A^2_{L_2},b^2_{L_2}))
	\]
	be two neural networks with $d$-dimensional input. If $L_1\leq L_2$, then we define
	\[
		P(\Phi^1,\Phi^2):=\left(\left(\tilde A^1_1,\tilde b^1_1\right),\dots,\left(\tilde A^1_{L_2},\tilde b^1_{L_2}\right)\right),
		\]
		where
		\[
			\tilde A_l:=
			\left[
			\begin{array}{c|cc|c|cc}
				A_{l,x_0}^1  & A_{l,x_1}^1 & 0 & \multirow{2}{*}[-.5em]{$\ldots$} & A_{l,x_{l-1}}^1 & 0\\[1em]
				A_{l,x_0}^2 & 0 & A_{l,x_1}^2 &  & 0& A_{l,x_{l-1}}^2
			\end{array}
			\right],\quad			
			\tilde b_l:=
			\bmat{c}{
				 b_l^1\\[1em]
				 b_l^2
			}\quad\text{for }1\leq l< L_1,
			\]
			and
			\[
			\tilde A_l:=
			\bmat{c|cc|c|cc|c|c|c}{
				A_{l,x_0}^2 & 0 & A_{l,x_1}^2 & \ldots & 0 & A_{l,x_{L_1-1}}^2 & A_{l,x_{L_1}}^2 & \ldots & A_{l,x_{l-1}}^2
			},\quad			
			\tilde b_l:=b_l^2
				\]
				for $L_1\leq l< L_2$ and
		\[
			\tilde A_{L_2}:=
			\left[
				\begin{array}{c|cc|c|cc|c|c|c}
					A_{L_1,x_0}^1 & A_{L_1,x_1}^1 & 0 & \multirow{2}{*}[-.5em]{$\ldots$} & A_{L_1,x_{L_1-1}}^1 & 0 & 0 &\multirow{2}{*}[-.5em]{$\ldots$} & 0\\[1em]
					A_{L_2,x_0}^2 & 0 & A_{L_2,x_1}^2 && 0 & A_{L_2,x_{L_1-1}}^2 & A_{L_2,x_{L_1}} && A_{L_2,x_{L_2-1}}\\
				\end{array}
			\right],
			\]
			\[
				\tilde b_{L_2}:=
			\bmat{c}{
				b_{L_1}^1\\[1em]
				 b_{L_2}^2
			}.
			\]

		A similar construction is used for the case $L_1>L_2$. Then $P(\Phi^1,\Phi^2)$ is a neural network with $d$-dimensional input and $\max\{L_1,L_2\}$ layers, which we call the \emph{parallelization of $\Phi^1$ and $\Phi^2$}. We have $M(P(\Phi^1,\Phi^2))=M(\Phi^1)+M(\Phi^2)$, $N(P(\Phi^1,\Phi^2))=N(\Phi^1)+N(\Phi^2)-d$ and $\act{P(\Phi^1,\Phi^2)}(x)=(\act{\Phi^1}(x),\act{\Phi^2}(x))$. 
\end{definitionlemma}

\begin{remark}
The previous Definition~\ref{deflemma:parallelization} can be generalized to a parallelization of an arbitrary number of neural networks. Let $n\in\N$ and, for $i=1,\ldots,n$, let $\Phi^i$ be a neural network with $d$-dimensional input and $L_i$ layers. The network
	\[
		P(\Phi^i:i=1,\ldots,n):=P(\Phi^1,\ldots,\Phi^n):=P(\Phi^1,P(\Phi^2,P(\ldots,P(\Phi^{n-1},\Phi^n)\ldots)))
		\]
		with $d$-dimensional input and $L=\max\{L_1,\ldots,L_n\}$ layers is called the \emph{parallelization of $\Phi^1,\Phi^2,\ldots,\Phi^n$}. We have 
		\[
			M(P(\Phi^1,\ldots,\Phi^n))=\sum_{i=1}^n M(\Phi^i) \quad\text{and}\quad N(P(\Phi^1,\ldots,\Phi^n))=\sum_{i=1}^n N(\Phi^i)-(n-1)d,
			\]
			and $\act{P(\Phi^1,\ldots,\Phi^n)}(x)=(\act{\Phi^1}(x),\ldots,\act{\Phi^n}(x))$. 
\end{remark}

\subsection{Standard neural networks}
In this subsection, we will show that if a function is a realization of a neural network, then it can also be realized by a standard neural network with similar complexity. The proof of this result is heavily based on the idea of the construction of an identity function with ReLU networks (Lemma \ref{lemma:neual_Id}).
\begin{lemma}\label{lemma:neural_nets_dense_to_standard}
	Let $\rho$ be the ReLU. Then there exist absolute constants $C_1,C_2>0$ such that for any neural network $\Phi$ with $d$-dimensional input and $L=L(\Phi)$ layers, $N=N(\Phi)$ neurons and $M=M(\Phi)$ nonzero weights, there is a standard neural network $\Phist$ with $d$-dimensional input, and $L$ layers, at most  $C_1 L N$ neurons and $C_2\cdot(LN+M)$ nonzero weights such that 
	\[
		\act{\Phi}(x)=\act{\Phist}(x)
		\]
		for all $x\in\R^d$.
\end{lemma}

% ###################### A less efficient construction #################
%\begin{proof}
%	We start by defining the matrices
%	\[
%		A^{\text{st}}_{l,x_{l-1}}:=	\left[
%			\begin{array}{ cc|c|ccc }
%				\multicolumn{5}{c}{$\id_{\R^{2\sum_{k=0}^{l-2}N_k}}$} & 0\\[1em]
%				\multicolumn{5}{c}{$0$} & \id_{\R^{N_{l-1}}}\\[1em]
%				\multicolumn{5}{c}{$0$} & -\id_{\R^{N_{l-1}}}\\[1em]
%				A_{l,x_{0}} & -A_{l,x_{0}} & \ldots & A_{l,x_{l-2}} & -A_{l,x_{l-2}} & A_{l,x_{l-1}}
%			\end{array}
%			\right]
%			\]
%			and
%			\[
%				b^{\text{st}}_l:=\left[
%					\begin{array}{c}
%						0\\[1em]
%						0\\[1em]
%						0\\[1em]
%						b_l
%					\end{array}
%					\right]
%				\]
%				for $l=1,\ldots,L-1$ and TODO L
%		Now we set \[ 
%				A^{\text{st}}_l=\bmat{c|c|c|c}{
%					0 & \ldots & 0 & A^{\text{st}}_{l,x_{l-1}}
%					},
%					\]
%					and \[
%						\Phist:=((A^{\text{st}}_1,b^{\text{st}}_1),\dots,(A^{\text{st}}_L,b^{\text{st}}_L))
%						\],
%						then $\Phist$ is a standard feedforward neural network and one can verify that ... using the computation scheme from Lemma \ref{}.
%				
%\end{proof}

\begin{proof}
	If $L=1$, then there is nothing to show. For $L>1$ we start by defining the matrices
	\[
		A^{\text{st}}_{1,x_0}:=\bmat{c}{\id_{\R^{N_0}}\\[1em]
						 -\id_{\R^{N_0}}\\[1em]
						 A_{1,x_0}},\quad
						 b^{\text{st}}_0:=\bmat{c}{0\\[1em]
						 0\\[1em]
						 b_0}
		\]
		and
	\[
		A^{\text{st}}_{l,x_{l-1}}:=	\left[
			\begin{array}{ ccccc }
				\multicolumn{2}{c}{\id_{\R^{2N_0}}} & \multicolumn{3}{c}{0}\\[1em]
				\multicolumn{2}{c}{0} & \multicolumn{3}{c}{\id_{\R^{\sum_{k=1}^{l-1}N_k}}}\\[1em]
				A_{l,x_0} & -A_{l,x_0} & A_{l,x_1} & \ldots & A_{l,x_{l-1}}
			\end{array}
			\right],\quad
				b^{\text{st}}_l:=\left[
					\begin{array}{c}
						0\\[1em]
						0\\[1em]
						b_l
					\end{array}
					\right]
				\]
				for $l=1,\ldots,L-1$ and 
				\[
					A^{\text{st}}_{L,x_{L-1}}:=\bmat{ cc|c|c|c}{
						A_{L,x_0} & -A_{L,x_0} & A_{L,x_1} & \ldots & A_{L,x_{L-1}}
						},\quad b_L^{\text{st}}:=b_L.
					\]
		Now we set \[ 
				A^{\text{st}}_l:=\bmat{c|c|c|c}{
					0 & \ldots & 0 & A^{\text{st}}_{l,x_{l-1}}
					},
					\]
					and \[
						\Phist:=((A^{\text{st}}_1,b^{\text{st}}_1),\dots,(A^{\text{st}}_L,b^{\text{st}}_L)).
						\]
						Then $\Phist$ is a standard neural network with $d$-dimensional input and $L$ layers and one can verify that $\act{\Phi}(x)=\act{\Phist}(x)$ for all $x\in\R^d$. For the number of neurons in the $l$-th layer $N^{\text{st}}_l$ we have $N^{\text{st}}_l=2N_0+\sum_{k=1}^l N_l\leq 2 N$. Hence, $N(\Phi_{\text{st}})\leq 2 L N$. To estimate the number of weights, note that $\norm{A^{\text{st}}_l}_{\ell^0}\leq 2 N + 2\norm{A_l}_{\ell^0}$ and $\norm{b^{\text{st}}_l}_{\ell^0}=\norm{b_l}_{\ell^0}$, from which we conclude that $N(\Phi_{\text{st}})\leq 2 (LN+M)$. 
				
\end{proof}

In the next remark we collect some properties of the class of functions that can be realized by a neural network.
\begin{remark}\label{remark:nn_lipschitz}
	If $\Phi$ is a neural network with $d$-dimensional input and $m$-dimensional output and $\rho$ is the ReLU, then $\act{\Phi}$ is a Lipschitz-continuous, piecewise affine-linear function. This follows easily from Definition \ref{def:nn_standard_nn} and Lemma \ref{lemma:neural_nets_dense_to_standard} which show that $\act{\Phi}$ can be expressed as the composition of Lipschitz-continuous, piecewise affine-linear transformations.

	In \cite[Theorem 2.1]{arora2018understanding} also the converse was shown, i.e.\ every piecewise affine-linear function can be realized by a neural network with the ReLU as activation function.
\end{remark}

\section{Sobolev spaces}\label{sec:sobolev}
Spaces of functions that admit generalized derivatives fulfilling suitable integrability properties are a crucial concept in modern theory of partial differential equations (cf.\ e.g.\ \cite{roubivcek2013nonlinear,evans1998partial,adams1975sobolev}). In order to study properties of PDEs using functional analytic tools, a differential equation is reformulated via a differential operator mapping one function space to another. For a wide range of differential equations the appropriate spaces in this formulation are Sobolev spaces. A historical background of the development of Sobolev spaces in the context of PDEs can be found in \cite[Lecture 1]{tartar2007introduction}. For a detailed treatment of the broad theory of Sobolev spaces we refer the reader to \cite{adams1975sobolev,evans1998partial, brezis2010functional}.

For this entire section, let $d\in\N$ and, if not stated otherwise, $\Omega\subset\R^d$ denote an open subset of $\R^d$. For $p \in [1,\infty]$, we denote by $L^p(\Omega)$ the standard Lebesgue spaces on $\Omega$. The following definition can be found in \cite[Chapter 4]{brezis2010functional}.

\begin{definition}[Sobolev space]\label{def:sobolev_space}
	Let $n\in\N_0$ and $1\leq p\leq\infty$. Then we define
		\[
			\Wkp[n]:=\left\{f\in\Lp: D^\alpha f\in\Lp \text{ for all } \alpha\in\N_0^d\text{ with }\pabs{\alpha}\leq n\right\}.
		\]
			Furthermore, for $f\in\Wkp[n]$ and $1\leq p<\infty$, we define the norm
			\[
				\norm{f}_{\Wkp[n]}:=\left(\sum_{0\leq \pabs{\alpha}\leq n}\norm{D^\alpha f}_{\Lp}^p\right)^{1/p}
				\]
				and 
			\[
				\norm{f}_{\Wkp[n][\infty]}:=\max_{0\leq \pabs{\alpha}\leq n}\norm{D^\alpha f}_{\Linf}.
				\]

\end{definition}

We write $\nabla f :=(D^1 f,\ldots, D^d f)$. Moreover, the space $\Wkp[n]$ equipped with the norm $\norm{\cdot}_{\Wkp[n]}$ is a Banach space. For $f\in\Wkp[n]$ we shall use the notation 
\[
	\norm{f}_{\Wkp[n]}=\norm{f}_{W^{n,p}}=\norm{f(x)}_{\Wkpd[n]{x}}
	\]
	if the domain is clear from the context, or if we want to emphasize the variable $x$ that the function $f$ depends on, respectively. If $n=0$, then the Sobolev space is just a Lebesgue space, i.e.\ $\Wkp[0]=\Lp$.

Many results about function spaces defined on a domain $\Omega$ require $\Omega$ to fulfill certain regularity conditions. Different geometrical conditions and the resulting properties have been intensively studied and can for example be found in \cite{adams1975sobolev}. For our purposes it is enough to focus on the condition introduced in the next definition which can be found in \cite[Appendix C.1]{evans1998partial}.

\begin{definition}[Lipschitz-domain]
	We say that a bounded and open set $\Omega\subset \R^d$ is a \emph{Lipschitz-domain} if for each $x_0\in\partial\Omega$ there exists $r>0$ and a Lipschitz continuous function $g:\R^{d-1}\to\R$ such that 
	\[
		\Omega\cap B_{r,\pabs{\cdot}}(x_0)=\left\{x\in B_{r,\pabs{\cdot}}(x): x_d>g(x_1,\ldots,x_{d-1})\right\},
	\]
	after possibly relabeling and reorienting the coordinate axes. 
\end{definition}
In the sequel, we will only work with convex domains. Moreover, every open, bounded, and convex domain is a Lipschitz domain \cite[Corollary 1.2.2.3]{grisvard1985elliptic}.

\subsection{Fractional Sobolev spaces}\label{subsec:fractional_sobololev}
Fractional Sobolev spaces play an important role in the analysis of partial differential equations. 
%In particular, they can be used to describe the image of trace operators, i.e.\ operators that realize a they are the appropriate tool to measure the regularity of Sobolev functions on the boundary $\partial \Omega$. Finer scale of regularity
In particular, they characterize the regularity of functions from Sobolev spaces defined on a domain $\Omega$ restricted to the boundary $\partial \Omega$ of the domain where the restriction is realized by a so-called trace operator (see e.g.\ \cite{brezis2010functional}). Moreover, a detailed description of various areas of further application is listed in \cite{bucur2016nonlocal}. For a more in-depth analysis of these spaces the interested reader is referred to \cite{lunardi2018interpolation} and \cite{tartar2007introduction}. 

Next, we define fractional-order spaces in terms of an intrinsic norm.
\begin{definition}[Sobolev-Slobodeckij spaces]\label{def:slobodeckij}
	We set for $0<s<1$ and $1\leq p\leq \infty$
		\[
			%\tWkp[s]:=\left\{f\in L^p(\Omega): \frac{\pabs{f(x)-f(y)}}{\pabs{x-y}^{s+d/p}}\in L^p(\Omega\times\Omega)\right\} 
			\Wkp[s]:=\left\{f\in L^p(\Omega): \norm{f}_{\Wkp[s]}<\infty\right\} 
			\]
			with
			\[
				\norm{f}_{\Wkp[s]}:=\left(\norm{f}_{\Lp}^p+\int_\Omega\int_\Omega\left(\frac{\pabs{f(x)-f(y)}}{\pabs{x-y}^{s+d/p}}\right)^p dx dy\right)^{1/p}\quad\text{for }1\leq p<\infty
			\]
			and 
			\[
				\norm{f}_{\Wkp[s][\infty]}:=\max\left\{\norm{f}_{\Lp[\infty]}, \esssup_{x,y\in\Omega}\frac{\pabs{f(x)-f(y)}}{\pabs{x-y}^s}\right\}.
				\]

\end{definition}
The space $\Wkp[s]$ endowed with the norm $\norm{\cdot}_{\Wkp[s]}$ is a Banach space called \emph{Sobolev-Slobodeckij space}.\footnote{For $p=\infty$ the space coincides with the space of bounded $s$-H{\"o}lder continuous functions. Precisely, each equivalence class contains a bounded $s$-H{\"o}lder continuous representative.}

\section{Approximations with deep ReLU neural networks in Sobolev type norms}\label{sec:main}
In this section, we derive upper and lower complexity bounds for approximations of functions from certain Sobolev spaces with deep ReLU networks. Our result is a generalization of~\cite[Theorem~1]{yarotsky2017error} (upper bounds) and~\cite[Theorem~4~a)]{yarotsky2017error} (lower bounds) to the case where the approximation error is measured in a Sobolev-type norm.

From now on we assume that $\rho:\R\to\R$ is the ReLU activation function (cf.\ Definition~\ref{def:nn_ReLU}). We are interested in approximating functions in subsets of the Sobolev space $\Wkp[n][p][\cube^d]$ with realizations of neural networks. For this we define the set:
\[
	\Fndp :=\left\{f\in \Wkp[n][p][\cube^d]:\norm{f}_{\Wkp[n][p][\cube^d]}\leq B\right\}.
	\]

\subsection{Upper complexity bounds}\label{subsec:upper_bounds}
In \cite[Theorem 1]{yarotsky2017error} it is shown that for $B=1$ and an arbitrary function $f\in\Fnd$ a ReLU network can be constructed that realizes an approximation with $L^\infty$-error at most $\eps$ using $\bigO(\log_2(\nicefrac{1}{\eps}))$ layers and $\bigO(\eps^{-d/n}\log_2(\nicefrac{1}{\eps}))$ nonzero weights and neurons.

We will show that the approximation can also be performed with respect to a continuous scale of higher order Sobolev-type norms for arbitrary $1\leq p\leq\infty$ and additionally, that there is a trade-off between the regularity used in the norm in which the approximation error is measured and the regularity used in the bounds. In particular, we will show the following theorem:

\begin{theorem}\label{thm:main}
	Let $d\in \N$, $n\in\N_{\geq 2}$, $1\leq p\leq \infty$, $B>0$, and $0\leq s \leq 1$. Then, there exists a constant $c=c(d,n,p,B,s)>0$ with the following properties:
	
	For any $\epsilon \in \epsin$, there is a neural network architecture $\Arch_\eps=\Arch_\eps(d,n,p,B,s,\eps)$ with $d$-dimensional input and one-dimensional output such that for any $f\in\Fndp$ there is a neural network $\Phi_\eps^f$ that has architecture $\Arch_\eps$ such that
	\[
		\norm{\act{\Phi_\eps^f} - f}_{\Wkp[s][p][\cube^d]}\leq \eps
		\]
		and
	\begin{enumerate}[(i)]
		\item $L(\Arch_\eps)\leq c\cdot\log_2(\eps^{-n/(n-s)})$;
		\item $M(\Arch_\eps)\leq c\cdot\eps^{-d/(n-s)}\cdot\log_2(\eps^{-n/(n-s)})$;
		\item $N(\Arch_\eps)\leq c\cdot\eps^{-d/(n-s)}\cdot\log_2(\eps^{-n/(n-s)})$.
	\end{enumerate}
\end{theorem}

Clearly, if $s=0$ and $p=\infty$, then Theorem~\ref{thm:main} coincides with the theorem shown by Yarotsky in~\cite{yarotsky2017error}. Note that if $\Phi$ is a neural network and $\rho$ is the ReLU, then the restriction of $\act{\Phi}$ to $\cube^d$ is bounded and Lipschitz continuous (cf.\ Remark \ref{remark:nn_lipschitz}) and hence in view of \cite[Chapter 1.3]{brenner2007mathematical} an element of the Sobolev space $\Wkp[1][\infty][\cube^d]$. Therefore, the expressions in the previous theorem are well-defined.

\begin{remark*}
	Theorem~\ref{thm:main} holds for all activation functions that are in some sense similar to the ReLU. In particular, it follows directly from~\cite[Proposition~1]{yarotsky2017error} that the statement holds for any continuous piecewise linear activation function $\hat\rho$ with $M$ breakpoints, where $1\leq M<\infty$.
\end{remark*}

As a consequence of Theorem~\ref{thm:main} and Lemma~\ref{lemma:neural_nets_dense_to_standard}, we can easily derive a similar statement for standard neural networks.
\begin{corollary}
	Let $d\in \N$, $n\in\N_{\geq 2}$, $1\leq p\leq \infty$, $B>0$, and $0\leq s \leq 1$. Then, there exists a constant $c=c(d,n,p,B,s)>0$ with the following properties:
	
	For any $\epsilon \in \epsin$, there is a standard neural network architecture $\Arch_{\text{st},\eps}=\Arch_{\text{st},\eps}(d,n,p,B,s,\eps)$ with $d$-dimensional input and one-dimensional output such that for any $f\in\Fndp$ there is a standard neural network $\Phi_{\text{st},\eps}$ that has architecture $\Arch_{\text{st},\eps}$ such that
	\[
		\norm{\act{\Phi_{\text{st},\eps}^f} - f}_{\Wkp[s][p][\cube^d]}\leq \eps
		\]
		and
	\begin{enumerate}[(i)]
		\item $L(\Arch_{\text{st},\eps})\leq c\cdot\log_2(\eps^{-n/(n-s)})$;
		\item $M(\Arch_{\text{st},\eps})\leq c\cdot\eps^{-d/(n-s)}\cdot\log_2^2(\eps^{-n/(n-s)})$;
		\item $N(\Arch_{\text{st},\eps})\leq c\cdot\eps^{-d/(n-s)}\cdot\log_2^2(\eps^{-n/(n-s)})$.
	\end{enumerate}
\end{corollary}

\subsection{Lower complexity bounds}\label{subsec:lower_bounds}
In this subsection, we show that the same regularity-complexity trade-off that can be observed for upper complexity bounds can be shown for lower bounds. In~Theorem~4~a) in~\cite{yarotsky2017error} Yarotsky proves that a network architecture capable of approximating any function $f\in \Fndone$ up to a $L^\infty$-error $\eps$ has at least $c\cdot \eps^{-d/(2n)}$ weights. Here we consider approximations in $W^{1,\infty}$ norm. The following theorem combines the result from~\cite{yarotsky2017error} with our result.
\begin{theorem}\label{thm:lower_bound}
	Let $d\in\N$, $n\in\N_{\geq 2}$, $B>0$ and $k\in\{0,1\}$. Then, there is a constant $c=c(d,n,B,k)>0$ with the following property:

	If $\eps\in\epsin$ and $\Arch_\eps=\Arch_\eps(d,n,B,k,\eps)$ is an architecture such that for any $f\in\Fnd$ there is a neural network $\Phi_\eps^f$ that has architecture $\Arch_\eps$ and 
	\[
		\norm*{\act{\Phi_\eps^f}-f}_{\Wkp[k][\infty][\cube^d]}\leq \eps,
		\]
		then $\Arch_\eps$ has at least $M(\Arch_\eps)\geq c\eps^{\nicefrac{-d}{2(n-k)}}$ weights.
\end{theorem}

	In case $k=0$, Theorem~\ref{thm:lower_bound} coincides with~\cite[Theorem~4~a)]{yarotsky2017error}.
\begin{remark*}
	Note that the case of standard neural network architectures is included in the theorem above and, thus, the same lower bounds hold true for this case.

	Moreover, the proof can easily be adapted to piecewise polynomial activation functions (not necessarily continuous) with a finite number of breakpoints.
\end{remark*}

\section{Discussion and future work}\label{sec:discussion}
The motivation behind deriving theoretical results for deep neural networks is to improve our understanding of their empirical success in a broad area of applications. In this section, we briefly describe the practical relevance of our result, but mostly focus on aspects where our theoretical framework differs from empirical observations or fails to incorporate computational limitations met in implementations. Furthermore, we will discuss possible extensions of our results.

\begin{itemize}
	\item \emph{Practical relevance.} Our results give upper and lower bounds for the complexity of neural network architectures used in \emph{Sobolev training} (see Section~\ref{subsec:intro_motivation} and~\cite{NIPS2017_7015}). Furthermore, we anticipate our results to lead to complexity bounds in the theoretical foundation of deep learning-based approaches for the numerical solution of PDEs. We leave an in-depth analysis of this topic for future research.
	\item \emph{Curse of dimension.} One of the main reasons for the current interest in deep neural networks is their outstanding performance in high-dimensional problems. As an example, consider the popular ILSVR\footnote{Typically abbreviated with \emph{ILSRC} which stands for \emph{ImageNet Large Scale Visual Recognition Challenge}.} challenge  (see~\cite{russakovsky2015imagenet}), an image recognition task on the ImageNet database (see~\cite{deng2009imagenet}) containing variable-resolution images which are typically downsampled before training to yield a roughly $(240\times240\times 3)$-dimensional input space (cf.\ e.g.\ \cite{krizhevsky2012imagenet,zeiler2014visualizing, simonyan2014very}).
		
Our asymptotic upper complexity bounds for the number of weights and neurons are in $\bigO(\eps^{-d/(n-s)})$, and thus depend strongly on the dimension $d$ of the input space. Moreover, the constants in the upper bounds for the number of layers, weights, and neurons also increase exponentially with increasing $d$. %Additionally, our lower complexity bounds show that in the setting considered in this paper one cannot expect to circumvent this dependency.
Even if there is still a gap between our upper and lower bounds for the weights of a network architecture, Theorem~\ref{thm:lower_bound} shows that one cannot expect to circumvent the curse of dimension in the setting considered in this paper.

Several ways have been proposed so far to tackle this common problem. One idea is to think of the data as residing on or near a manifold $\mathcal{M}$ embedded in $\R^d$ with dimension~$d_{\mathcal{M}}\ll d$ (see~\cite{bengio2013representation}). Considering, for example, image classification problems this idea seems to be rather intuitive since most of the elements in $\R^{240\times 240 \times 3}$ are not perceived as images. A similar approach is to narrow down the approximated function space by incorporating invariances (see \cite{mallat2012group}). If, for example, the approximated function~$f$ maps images to labels, then it makes sense to assume that $f$ is translation and rotation invariant. The additional structure of the function space can then be exploited in the approximation (see~\cite[Section~5]{petersen2017optimal}). In particular, the last approach could also be suitable to tackle the curse of dimension in our result.
%manifold reference \cite{andras2017high}

%In \cite{yarotsky2017error} D.\ Yarotsky showed that the in Theorem \ref{thm:main} for $s=0$ constructed unbounded depth approximations for functions in $\Wkp[n][\infty][\cube^d]$ (with $n>2$) are asymptotically more efficient in the number of weights and neurons than approximations with fixed length $L$ if 
	\item \emph{The power of depth.} For $s=0$, the approximations obtained in Theorem~\ref{thm:main} and \cite[Theorem 1]{yarotsky2017error} coincide. In this case, Yarotsky showed in \cite{yarotsky2017error} that the constructed unbounded depth approximations for functions in $\Wkp[n][\infty][\cube^d]$ (with $n>2$) are asymptotically more efficient in the number of weights and neurons than approximations with fixed length $L$ if 
	\[
		\frac{d}{n}<\frac{1}{2(L-1)}.	
		\]
		As a consequence, to be more efficient than a shallow network, i.e.\ a network with depth~$L=2$, one needs $n>2d$ regularity. 
%In view of a typically very large dimension $d$ this does not completely explain the success of deep neural networks over shallow ones. This problem is in parts also addressed by possible approaches to tackle the curse of dimension presented in the paragraph above. 
Even if this result does not completely explain the success of deep networks over shallow ones, since $d$ is typically very large, it would be interesting to obtain similar results for higher-order Sobolev norms.
%There are no results yet about the effect of depth on approximations in higher-order Sobolev-type norms ($0<s\leq 1$). This remains an interesting question for further research.
		
%This gets more pronounced with increasing reg.

%under some assumptions on the regularity $n$ a deep neural network implements an approximation more efficiently than a shallow one. 
%\item \emph{The power of depth.} 
%	\begin{enumerate}
%		\item \emph{The case $s=0$.} In \cite{yarotsky2017error} Yarotsky showed that the in Theorem \ref{thm:main} for $s=0$ constructed unbounded depth approximations for functions in $\Wkp[n][\infty][\cube^d]$ (with $n>2$) are asymptotically more efficient in the number of weights and neurons than approximations with fixed length $L$ if 
%	\[
%		\frac{d}{n}<\frac{1}{2(L-1)}.	
%		\]
%		As a consequence, to be more efficient than a shallow network with depth $L=2$ one needs $n>2d$ regularity. In view of a typically very large dimension $d$ (see paragraph above) we do not think that this completely explains the success of deep neural networks over shallow ones. This problem is in parts also addressed by the possible approaches to tackle the curse of dimension. 
%		\item\emph{The case $0<s\leq 1$.}
%			There are no results yet about the effect of deep networks for approximations in higher-order Sobolev-type norms ($0<s\leq 1$). This remains an interesting question for further research.
%	\end{enumerate}
		
	\item \emph{Unbounded complexity of weights.} When training a neural network on a computer, the weights have to be stored in memory. In practice, storing weights up to an arbitrary precision or even unbounded weights (in absolute value) is infeasible. In the construction of the neural network $\act{\Phi^f_\eps}$ that approximates a function $f$ up to an $W^{s,\infty}$-error $\eps$ and realizes the upper complexity bound from Theorem~\ref{thm:main} we used weights $W_\eps$ with $\pabs{W_\eps}\to \infty$ as $\eps\downarrow 0$ (see construction of $\Phi_{m,l}$ in the proof of Lemma~\ref{lemma:partition_of_unity}~(\ref{item:pou_network})).
		
In \cite{petersen2017optimal, petersen2017optimal, GrohsPEB2019}, neural networks are restricted to possess quantized weights (see~\cite[Definition~2.9]{petersen2017optimal}) which controls the complexity of the weights depending on the approximation error $\eps$. It would be interesting to extend Theorem~\ref{thm:main} to networks with quantized weights and to analyze how coarse a quantization is possible to maintain the associated upper bounds.
\end{itemize}

\section*{Acknowledgements}
I.G.\ would like to thank Mones Raslan for fruitful discussions on the topic and acknowledges support from the Research Training Group "Differential Equation- and Data-driven Models in Life Sciences and Fluid Dynamics: An Interdisciplinary Research Training Group (DAEDALUS)" (GRK 2433) funded by the German Research Foundation (DFG). G. K. acknowledges partial support by the Bundesministerium f\"ur Bildung und Forschung (BMBF) through the Berliner Zentrum for Machine Learning (BZML), Project AP4, by the
Deutsche Forschungsgemeinschaft (DFG) through grants CRC 1114 "Scaling Cascades in Complex Systems",
Project B07, CRC/TR 109 "Discretization in Geometry and Dynamics', Projects C02 and C03,
RTG DAEDALUS (RTG 2433), Projects P1 and P3, RTG BIOQIC (RTG 2260), Projects P4 and P9,
and SPP 1798 "Compressed Sensing in Information Processing", Coordination Project and Project Massive
MIMO-I/II, by the Berlin Mathematics Research Center MATH+ , Projects EF1-1 and EF1-4, and
by the Einstein Foundation Berlin. P.P.\ is supported by a DFG Research Fellowship ”Shearlet-based energy functionals for anisotropic phase-field methods”.

\appendix

\section{Interpolation spaces}\label{app:interpolation}
In this section, we recall some essential notations and results from interpolation theory in Banach spaces. In particular, we present the $K$-method of real interpolation.

%Interpolation techniques allow to transfer results obtained for two Banach spaces $B_0$ and $B_1$ to spaces that in some sense are in between $B_0$ and $B_1$, so-called \emph{interpolation spaces}.
For two Banach spaces $B_0$ and $B_1$ interpolation techniques allow the definition and study of a family of spaces that in some sense bridge the gap between $B_0$ and $B_1$. This rather abstract concept will be applied to Sobolev spaces in Section \ref{subsec:fractional_sobololev} to derive spaces of fractional regularity. Interpolation theory is a valuable tool in harmonic analysis and the study of partial differential equations. A small summary about the history of the development of the theory of interpolation spaces can be found in \cite[Lecture 21]{tartar2007introduction}.
Interested readers are recommended to read the monographs \cite{lunardi2018interpolation,lofstrm1976interpolation,tartar2007introduction} and \cite{triebel1978interpolation} for a detailed treatment in the context of partial differential equations. 

The basic notions of functional analysis are assumed to be known and we refer to \cite{yosida1995functional} for an introduction.

If $B_0,B_1$ are two Banach spaces such that $B_1$ is continuously embedded in $B_0$, then we write $B_1\subset B_0$ and call $(B_0,B_1)$ an \emph{interpolation couple}. One can also consider more general pairs of Banach spaces but this setting is sufficiently general for our purposes.
\begin{definition}
	Let $(B_0,B_1)$ be an interpolation couple. For every $u\in B_1$ and $t>0$ we define
	\[
		K(t,u,B_0,B_1):=\inf_{v\in B_1} \left(\norm{u-v}_{B_0} + t\norm{v}_{B_1}\right).
		\]
\end{definition}
The first term measures how well $u$ can be approximated by elements from the smaller space $B_1$ in $\norm{\cdot}_{B_0}$ and the second term is a penalty term weighted by $t$. 
%TODO: when this example is used one has to introduce characteristic functions, and C^1[0,1] (currently only C^{1,0}[0,1])
%Consider e.g.\ the interpolation couple $(B(0,1),C^{1,0}[0,1])$ and the function $u:(0,1)\to\R$ with $u(x)=\chi_{(0,1/2)}$ which has a jump point at $x=1/2$. For $\eps>0$ let now $v_\eps\in C^1[0,1]$ such that $\norm{u-v_\eps}_{B(0,1)}\leq \eps$, then clearly we have for the penalty term $t \norm{v_\eps}_{C^1[0,1]}\to \infty$ as $\eps\to 0$.

Now we can define interpolation spaces.
\begin{definition}
	Let $(B_0,B_1)$ be an interpolation couple. Let $0<\theta<1$ and $1\leq p\leq\infty$. We set
	\[
		\norm{u}_{\icouple}:=\begin{dcases}
			\left( \int_0^\infty t^{-\theta p}K(t,u,B_0,B_1)^p \frac{dt}{t}\right)^{1/p},&\text{for}\quad 1\leq p<\infty\\[0.5em]
			\sup_{0<t<\infty} t^{-\theta}K(t,u,B_0,B_1),& \text{for}\quad p=\infty.
		\end{dcases}
	\]
	Moreover, we define the set
	\[
		\icouple:=\ipair:=\left\{u\in B_0:\norm{u}_{\icouple}<\infty\right\}.
	\]
\end{definition}
	The norm $u\mapsto \norm{u}_{\icouple}$ turns $\icouple$ into a Banach space (cf.\ \cite[Proposition~1.5]{lunardi2018interpolation}) which is called a \emph{real interpolation space}. 

The next lemma shows that in the sense of a continuous embedding the spaces~$B_{\theta,p}$ form a family of nested Banach spaces.
	\begin{lemma}\label{lemma:interpolation_relations}
			Let $1\leq p\leq \infty$, then the following holds:
			\begin{enumerate}[(i)]
				\item If $(B_0,B_1)$ is an interpolation couple and $0<\theta_1<\theta_2<1$, then
			\[
					B_1\subset B_{\theta_2,p}\subset B_{\theta_1,p} \subset B_0;
					\]
				\item\label{item:interpolation_identity} if $B$ is a Banach space and $0<\theta<1$, then $\icouple[B][B]=B$.
			\end{enumerate}
	\end{lemma}
	\begin{proof}
		For (i) see \cite[p.\ 2]{lunardi2018interpolation} and for (ii) \cite[Proposition 1.4]{lunardi2018interpolation}.	
	\end{proof}

	One of the most important theorems in interpolation theory shows how the norm of an operator defined on interpolation couples relates to the operator norm with respect to the corresponding interpolation spaces. The theorem can be found in \cite[Theorem 1.6]{lunardi2018interpolation}.
\begin{theorem}\label{thm:interpolation_operator}
	Let $(A_0,A_1)$ and $(B_0,B_1)$ be two interpolation couples. If $T\in\Lcal(A_0,B_0)\cap\Lcal(A_1,B_1)$, then $T\in\Lcal(A_{\theta,p},B_{\theta,p})$ for every $0<\theta<1$ and $1\leq p\leq\infty$. Furthermore,
\begin{equation}
	\norm{T}_{\Lcal(A_{\theta,p}, B_{\theta,p})}\leq\norm{T}_{\Lcal(A_0,B_0)}^{1-\theta}\cdot \norm{T}_{\Lcal(A_1,B_1)}^\theta.
\end{equation}
\end{theorem}

As a corollary from the previous theorem the following useful result can be obtained (see \cite[Corollary 1.7]{lunardi2018interpolation}).
\begin{corollary}\label{cor:interpol_norm}
	Let $(B_0,B_1)$ be an interpolation couple. Moreover, let $0<\theta<1$ and $1\leq p\leq\infty$, then there is a constant $c=c(\theta,p)>0$ such that for all $u\in B_1$ we have  
	\[
		\norm{u}_{B_{\theta,p}}\leq c \norm{u}_{B_0}^{1-\theta}\norm{u}_{B_1}^{\theta}.
		\]
		For the case $p=\infty$ we have $c(\theta,\infty)=1$.
\end{corollary}

%\subsection{notes}
%\begin{itemize}
%	\item \cite[p. 117]{brezis2010functional} "useful tool in harmonic analysis and PDE" and a lot of references to quote
%	\item introduction: Genzel(Section 5.2)
%	\item \cite[Viele Sachen zum Einleitung schreiben vermutlich]{lofstrm1976interpolation}
%	\item \cite[1.1 Introduction, S.15]{triebel1978interpolation}
%	\item M.G. verweist auf gleiches Buch von Triebel aber 1983 version S.35
%	\item \cite[Chapter 14]{brenner2007mathematical}
%	\item \cite[wo genau?]{tartar2007introduction}
%	\item \cite[Introduction]{lunardi2018interpolation}
%	\item Yserentant script
%\end{itemize}

\section{Sobolev spaces}\label{app:sobolev}
In this section, we give a brief introduction to Sobolev spaces and present related concepts which are important for the proof of Theorem~\ref{thm:main}. In detail, we show estimates for compositions and products of functions from certain Sobolev spaces in Subsection~\ref{subsec:composition} and~\ref{subsec:product}, respectively. In Subsection \ref{subsec:taylor}, we introduce a generalization of Taylor polynomials which is one of the key tools in the proof of Theorem~\ref{thm:main}. In the last subsection, we apply Banach space interpolation theory to (integer) Sobolev spaces to define fractional Sobolev spaces. Again, let for this entire section, $d\in\N$ and, if not stated otherwise, $\Omega\subset\R^d$ denote an open subset of $\R^d$. 

As a generalization of Definition~\ref{def:sobolev_space}, we now introduce Sobolev spaces of vector-valued functions (see \cite[Chapter 1.2.3]{roubivcek2013nonlinear}).
\begin{definition}
	We set for $n\in\N_0$, $m\in\N$ and $1\leq p\leq \infty$
	\[
		\Wkpm[n]:=\left\{(f_1,\ldots,f_m): f_i\in\Wkp[n]\right\}
		\]
		with 
		\[
			\norm{f}_{\Wkpm[n]}:=\left(\sum_{i=1}^m \norm{f_i}_{\Wkp[n]}^p\right)^{1/p}\quad\text{for }1\leq p<\infty
			\]
		and
		\[
			\norm{f}_{\Wkpm[n][\infty][\Omega][m]}:=\max_{i=1,\ldots,m}\norm{f_i}_{\Wkp[n][\infty]}.
			\]
\end{definition}
These spaces are again Banach spaces. To simplify the exposition, we introduce notation for a family of semi-norms.

\begin{definition}[Sobolev semi-norm]\label{def:sob_seminorm}
	For $n,k\in\N_0$ with $k\leq n$, $m\in\N$ and $1\leq p\leq\infty$ we define for $f\in\Wkpm[n]$ the \emph{Sobolev semi-norm}
			\[
				\pabs{f}_{\Wkpm[k]}:=\Bigg(\sum_{i=1,\ldots,m,\,\pabs{\alpha}=k}\norm{D^\alpha f_i}_{\Lp}^p\Bigg)^{1/p}\quad\text{for }1\leq p<\infty
				\]
				and
			\[
				\pabs{f}_{\Wkpm[k][\infty][\Omega][m]}:=\max_{\substack{i=1,\ldots,m,\\\pabs{\alpha}=k}}\norm{D^\alpha f_i}_{\Linf}.
				\]
\end{definition}
For $m=1$ we only write $\pabs{\cdot}_{\Wkp[k]}$. Note that $\pabs{\cdot}_{\Wkp[0]}$ coincides with $\norm{\cdot}_{\Lp}$.

% for intro to geometrical properties see \cite[IV, gem prop of domains]{adams1975sobolev} for nice intro.
%similar definition in \cite[Definition 2.9]{sofonea18variational}\cite[Definition 1.4.4]{brenner2007mathematical}

%A fundamental class of results in the theory of function spaces are extension theorems. Properties of function spaces defined on the whole $\R^d$ can be transferred to spaces defined on a domain $\Omega\subset \R^d$, if $\Omega$ admits an extension. 

Rademacher's Theorem states that a Lipschitz continuous function $f:\Omega\to\R$ is differentiable (in the classical sense) almost everywhere in $\Omega$ (see e.g.\ \cite[Theorem 5.8.6]{evans1998partial}). The following theorem yields an alternative characterization for the space $W^{1,\infty}$ using Lipschitz continuous functions and connects the notion of the classical and the weak derivative. 
%The following theorem identifies the Sobolev space $W^{1,\infty}$ and the set of bounded Lipschitz functions. 
\begin{theorem}\label{theorem:Lipschitz}
	Let $\Omega\subset \R^d$ be open and convex. If $f\in\Winf$, then $f$ is $\norm{\normabsspace\pabs{\nabla f(x)}\normabsspace}_{\Lpd[\infty]}$-Lipschitz. Conversely, if $f$ is bounded and $L$-Lipschitz on $\Omega$, then $f \in \Winf$, the weak gradient of $f$ agrees almost everywhere with the classical gradient, and 
	\[
		\norm{\normabsspace\pabs{\nabla f(x)}\normabsspace}_{\Lpd[\infty]} \leq L.
		\]
\end{theorem}

A proof can be found in \cite[Theorem 4.1]{heinonen2005lectures} together with \cite[Remark 4.2]{heinonen2005lectures}.

To use the bounds for the gradient obtained in the previous theorem in the Sobolev semi-norm setting from Definition \ref{def:sob_seminorm}, the following observation turns out to be useful. If $f\in \Winf$, then 
	\begin{equation}\label{eq:gradient_equivalence}
		\pabs{f}_{\Winf}\leq\norm{\normabsspace\pabs{\nabla f}\normabsspace}_{\Linf} \leq \sqrt d \pabs{f}_{\Winf}.
	\end{equation}

%\begin{proof}
%	Step 1: estimate $\norm{x}_{l^2}\leq \sqrt n\norm{x}_{l^\infty}$\\
%	Step 2: interchange $ess sup$ and $\max$ e.g.\ with this strategy: Since $g_i$ is continuous (and also $\max g_i$) we have that ess sup = sup.
%	then interchange.
%	https://math.stackexchange.com/questions/618101/essential-supremum-with-the-continuous-function	
%\end{proof}
Moreover, it is not hard to see that if $\Omega\subset \R^d$ open and bounded, then $\Wkp[n][\infty]\subset \Wkp[n]$ for $1\leq p<\infty$, and $n\in\N_0$.

\subsection{Composition estimate}\label{subsec:composition}
In this subsection, we derive an estimate for the semi-norm $\pabs{\cdot}_{W^{1,\infty}}$ of the composition $g \circ f$ of two functions $f\in \Wkpm[1][\infty][\Omega][m]$ and $g\in \Wkp[1][\infty][\R^m]$. Since in general this is not a well-defined notion we need to clarify what $g \circ f$ denotes in this case.
To demonstrate the problem let $f$ and $g$ be as above and recall that $f$ and $g$ are equivalence classes of functions. Let now $\hat f$ be a representative for the equivalence class $f$ and $\hat g_1, \hat g_2$ be two representatives for $g$. Then $\hat g_1=\hat g_2$ almost everywhere but in general we do not have $\hat g_1\circ \hat f =\hat g_2 \circ \hat f$ almost everywhere. To circumvent this problem we denote by $g\circ f$ the equivalence class of $\hat g \circ \hat f$ where $\hat g$ and $\hat f$ are Lipschitz continuous representatives of $g$ and $f$, respectively.

The next lemma recalls a well known fact about the composition of Lipschitz functions and is not hard to see.
\begin{lemma}\label{lemma:lipschitz_comp}
	Let $d,m\in\N$ and $\Omega_1\subset \R^d$, $\Omega_2\subset\R^m$. Let $f:\Omega_1\to\Omega_2$ be a $L_1$-Lipschitz function and $g:\Omega_2\to\R$ be a $L_2$-Lipschitz function for some $L_1,L_2>0$. Then the composition $g\circ f:\Omega_1\to\R$ is a $(L_1\cdot L_2)$-Lipschitz function.
\end{lemma}

The following corollary establishes a chain rule estimate for $W^{1,\infty}$.
\begin{corollary}\label{cor:composition_norm}
	Let $d,m\in\N$ and $\Omega_1\subset \R^d$, $\Omega_2\subset\R^m$ both be open, bounded, and convex. Then, there is a constant $C=C(d,m)>0$ with the following property:
	
	If $f\in \Wkpm[1][\infty][\Omega_1][m]$ and $g\in \Wkp[1][\infty][\Omega_2]$ are Lipschitz continuous representatives such that $\ran f\subset \Omega_2$, then for the composition $g\circ f$ it holds that $g\circ f\in \Wkp[1][\infty][\Omega_1]$ and we have 
	\[
		\pabs{g\circ f}_{\Wkp[1][\infty][\Omega_1]}\leq C \pabs{g}_{\Wkp[1][\infty][\Omega_2]}\pabs{f}_{\Wkpm[1][\infty][\Omega_1][m]}
		\]
		and
		\[
			\norm{g\circ f}_{\Wkp[1][\infty][\Omega_1]} \leq C \max\left\{\norm{g}_{\Lp[\infty][\Omega_2]},\pabs{g}_{\Wkp[1][\infty][\Omega_2]}\pabs{f}_{\Wkpm[1][\infty][\Omega_1][m]}\right\}.
			\]
\end{corollary}

\begin{proof}
	We set for $j=1,\ldots,m$
	\[
		L_j:=\norm{\normabsspace\pabs{\nabla f_j}\normabsspace}_{\Lp[\infty][\Omega_1]}\quad\text{and}\quad L_f:=(L_1,\ldots,L_m).
		\]
		It then follows from Theorem \ref{theorem:Lipschitz} that $f_j$ is $L_j$-Lipschitz and hence, $f$ is $\pabs{L_f}$-Lipschitz. In the same manner we define 
		\[
			L_g:=\norm{\normabsspace\pabs{\nabla g}\normabsspace}_{\Lp[\infty][\Omega_2]}
			\]
			and have that $g$ is $L_g$-Lipschitz. With Lemma \ref{lemma:lipschitz_comp} we conclude that also the composition $g\circ f$ is $(\pabs{L_f}\cdot L_g)$-Lipschitz. Since $g$ is bounded and hence also $g\circ f$ applying Theorem \ref{theorem:Lipschitz} again yields that $g\circ f\in \Wkp[1][\infty][\Omega_1]$ and furthermore 
			\begin{equation}\label{eq:sob_gradient_chain}
				\norm{\normabsspace\pabs{\nabla (g\circ f)}\normabsspace}_{\Lp[\infty][\Omega_1]}\leq \pabs{L_f}\cdot L_g.
			\end{equation}
		We get
	\begin{align*}
		\pabs{g\circ f}_{\Wkp[1][\infty][\Omega_1]}&\leq\norm{\normabsspace\pabs{\nabla (g\circ f)}\normabsspace}_{\Lp[\infty][\Omega_1]}\\
		&\leq\pabs{L_f}\cdot L_g\\
		&\leq \sqrt m \norm{L_f}_{\ell^\infty}\cdot \sqrt m \pabs{g}_{\Wkp[1][\infty][\Omega_2]}\\
		&\leq \sqrt d  m \pabs{f}_{\Wkpm[1][\infty][\Omega_1][m]}\cdot \pabs{g}_{\Wkp[1][\infty][\Omega_2]},
	\end{align*}
	where we used Equation \eqref{eq:gradient_equivalence} in the first step and the previous estimate \eqref{eq:sob_gradient_chain} in the second. The third step is due to estimating the $\ell^2$ norm with the $\ell^\infty$ norm on $\R^m$, and applying \eqref{eq:gradient_equivalence} to the definition of $L_g$. The last step uses Equation \eqref{eq:gradient_equivalence} again. This shows the first estimate. The second estimate is now an easy consequence of the first.
\end{proof}

%\begin{lemma}\label{lemma:weakdiff}
%Gegeben sei eine Zerlegung $a=x_0 < x_1 < \ldots < x_N = b$ des Intervalls $\left[a,b\right]$ und eine Funktion $f \in C\left(\left[a,b\right]\right)$ mit 
%\[
%	f\restriction\interval{x_i}{x_{i+1}} \in C^1\left(\interval{x_i}{x_{i+1}}\right), \quad i=0,\ldots,N-1 .
%\]
%	Dann ist $f$ schwach differenzierbar in $\intervalo{a}{b}$ und die schwache Ableitung sind die st\"uckweise definierten Ableitungen.
%\end{lemma}

\subsection{Product estimate}\label{subsec:product}
In the following lemma, we show an estimate for the semi-norm of a product of weakly differentiable functions.
\begin{lemma}\label{lemma:product_rule_bound_p}
	Let $f\in \Wkp[1][\infty][\Omega]$ and $g\in \Wkp[1][p][\Omega]$ with $1\leq p\leq \infty$, then $fg\in\Wkp[1][p][\Omega]$ and there exists a constant $C=C(d,p)>0$ such that
	\[
		\pabs{fg}_{\Wkp[1][p]}\leq C \left(\pabs{f}_{\Wkp[1][\infty]} \norm{g}_{\Lp} + \norm{f}_{\Linf} \pabs{g}_{\Wkp[1][p]}\right).
		\]
		For $p=\infty$, we have $C=1$.
\end{lemma}

\begin{proof}
	Clearly, $fg, fD^i g+gD^i f\in L^1_{\text{loc}}(\Omega)$ so that the product formula in \cite[Chapter 7.3]{gilbarg1998elliptic} yields that for the weak derivatives of $fg$ it holds 
	\[
		D^i(fg)= (D^i f)g+ f (D^i g)
		\]
		for $i=1,2,\ldots,d$. Thus, we have
		\begin{align*}
			\pabs{fg}_{\Wkp[1]}&=\left( \sum_{i=1}^d \norm*{(D^i f)g+ f (D^i g)}_{\Lp}^p\right)^{1/p}\\
			&\leq \sum_{i=1}^d \norm*{(D^i f)g+ f (D^i g)}_{\Lp}\\
			&\leq \sum_{i=1}^d \norm*{(D^i f)g}_{\Lp}+ \norm{f (D^i g)}_{\Lp}.
		\end{align*}
		Note that we have 
		\[
			\norm*{(D^i f)g}_{\Lp}\leq \norm{D^i f}_{\Lp[\infty]} \norm{g}_{\Lp}\leq \pabs{ f}_{\Wkp[1][\infty]} \norm{g}_{\Lp}
			\]
			and
			\[
				\norm{f (D^i g)}_{\Lp}\leq \norm{f}_{\Lp[\infty]}\norm{D^i g}_{\Lp}.
				\]
				Finally, we get
		\begin{align*}
			\pabs{fg}_{\Wkp[1]}&\leq d\pabs{ f}_{\Wkp[1][\infty]} \norm{g}_{\Lp}+\norm{f}_{\Lp[\infty]}\sum_{i=1}^d \norm{D^i g}_{\Lp}\\
			&\leq C \left(\pabs{f}_{\Wkp[1][\infty]} \norm{g}_{\Lp} + \norm{f}_{\Linf} \pabs{g}_{\Wkp[1][p]}\right),
		\end{align*}
		where $C=C(d,p)>0$ is a constant. Furthermore, it is clear that also $fg\in\Lp$.
\end{proof}

\subsection{Averaged Taylor polynomial}\label{subsec:taylor}
In this subsection, we develop a polynomial approximation in the spirit of Taylor polynomials but appropriate for Sobolev spaces. A polynomial approximation $P_f$ of a function $f\in\Fndp$ is the first step towards an approximation of $f$ realized by a neural network in the proof of Theorem~\ref{thm:main}.

A reference for this entire subsection is \cite[Chapter 4.1]{brenner2007mathematical}.

%The following definition introduces the tool used for the "averaging".
%\begin{definition}[cut-off function]
%	Let $x\in\R^d$ and $r>0$. We call a function $\phi\in C_c^\infty(\R^d)$ such that
%	\begin{enumerate}[(i)]
%		\item $\supp \phi = \overline{B_{r,\pabs{\cdot}}(x)}$
%		\item $\int_{\R^n}\phi(x)dx=1$
%	\end{enumerate}
%	a \emph{cut-off function supported on $B$}
%\end{definition}

%Definition from \cite[Definition 4.1.3]{brenner2007mathematical}.
\begin{definition}[averaged Taylor polynomial]\label{def:taylor}
	Let $n\in\N$, $1\leq p\leq \infty$ and $f\in\Wkp[n-1]$, and let $x_0\in\Omega$, $r>0$ such that for the ball $B:=B_{r,\pabs{\cdot}}(x_0)$ it holds that $B \ksub \Omega$. The corresponding \emph{Taylor polynomial of order $n$ of $f$ averaged over $B$} is defined for $x\in\Omega$ as 
	\begin{equation}\label{eq:sob_taylor}
		Q^n f(x):=\int_B T^n_y f(x)\phi(y)dy,
	\end{equation}
		where
	\begin{equation}\label{eq:sob_taulor_sub}
		T^n_y f(x) := \sum_{\pabs{\alpha}\leq n-1}\frac{1}{\alpha!}D^\alpha f(y)(x-y)^\alpha
	\end{equation}
	and $\phi$ is an arbitrary cut-off function supported in $\overline{B}$, i.e.\ 
	\[
		\phi\in C_c^\infty(\R^d) \quad\text{with}\quad\phi(x)\geq 0\text{ for all }x\in\R^d, \quad\supp \phi = \overline{B}\quad \text{and} \quad\int_{\R^n}\phi(x)dx=1.
		\]
\end{definition}

A cut-off function as used in the previous definition always exists. A possible choice is
		\begin{equation*}
			\psi(x)=\begin{cases}
				e^{-\left(1-(\pabs{x-x_0}/r)^2\right)^{-1}}, & \text{if }\pabs{x-x_0}<r\\
					0,&\text{else}
			\end{cases}
		\end{equation*}
		normalized by $\int_{\R^d}\psi(x)dx$. Next, we derive some properties of the averaged Taylor polynomial.

\begin{remark}\label{remark:sob_taylor_lin}
		From the linearity of the weak derivative we can easily conclude that the averaged Taylor polynomial is linear in $f$.
\end{remark}

Recall that the averaged Taylor polynomial is defined via an integral and some cut-off function (cf.~\eqref{eq:sob_taylor}) that perform an averaging of a polynomial expression (cf.~\eqref{eq:sob_taulor_sub}). Additionally, the following lemma shows that the averaged Taylor polynomial of order $n$ is indeed a polynomial of degree less than $n$ in $x$.
\begin{lemma}\label{prop:taylor_is_polynom}
	Let $n\in\N$, $1\leq p\leq \infty$ and $f\in\Wkp[n-1]$, and let $x_0\in\Omega$, $r>0,R\geq 1$ such that for the ball $B:=B_{r,\pabs{\cdot}}(x_0)$ it holds that $B \ksub \Omega$ and $B\subset B_{R,\norm{\cdot}_{\ell^\infty}}(0)$. Then the Taylor polynomial of order $n$ of $f$ averaged over $B$ can be written as
	\[
		Q^n f(x)=\sum_{\pabs{\alpha}\leq n-1} c_{\alpha}x^\alpha
		\]
		for $x\in\Omega$. 
		
		Moreover, there exists a constant $c=c(n,d,R)>0$ such that the coefficients $c_\alpha$ are bounded with $\pabs{c_\alpha}\leq c r^{-d/p}\norm{f}_{\Wkp[n-1][p][\Omega]}$ for all $\alpha$ with $\pabs{\alpha}\leq n-1$.
		%One could also take $\pabs{c_\alpha}\leq c r^{-d/p}\norm{f}_{\Wkp[n-1][p][B]}$
\end{lemma}

\begin{proof}
	The first part of the proof of this lemma follows closely the chain of arguments in \cite[Equations~(4.1.5) - (4.1.8)]{brenner2007mathematical}. We write for $\alpha\in\N_0^d$ 
	\[
		(x-y)^\alpha=\prod_{i=1}^d (x_i-y_i)^{\alpha_i} = \sum_{\substack{\gamma,\beta \in \N_0^d,\vspace{0.2em}\\\gamma+\beta=\alpha}} a_{(\gamma,\beta)} x^\gamma y^\beta,
		\]
		where $a_{(\gamma,\beta)}\in\R$ are suitable constants with 
		\begin{equation}\label{eq:sob_multinomial}
			\pabs*{a_{(\gamma,\beta)}}\leq \left(\begin{array}{c}
				\gamma+\beta\\
				\gamma
			\end{array}
			\right)=\frac{(\gamma+\beta)!}{\gamma\,!\;\beta\,!}
		\end{equation}
		in multi-index notation. Then, combining Equation \eqref{eq:sob_taylor} and \eqref{eq:sob_taulor_sub} yields
	\begin{align*}
		Q^n f(x) &= \sum_{\pabs{\alpha}\leq n-1}\sum_{\gamma +\beta=\alpha} \frac{1}{\alpha!}a_{(\gamma,\beta)}x^\gamma\int_B D^\alpha f(y)y^\beta\phi(y)dy\\
		&=\sum_{\pabs{\gamma}\leq n-1}x^\gamma \underbrace{\sum_{\pabs{\gamma+\beta}\leq n-1}\frac{1}{(\gamma+\beta)!}a_{(\gamma,\beta)}\int_B D^{\gamma+\beta} f(y)y^\beta\phi(y)dy}_{=:c_\gamma}.\numberthis\label{eq:c_gamma}
	\end{align*}
	For the second part, note that 
	\begin{align*}
		\pabs*{\int_B D^{\gamma+\beta} f(y)y^\beta\phi(y)dy}&\leq\int_B \pabs*{D^{\gamma+\beta} f(y)}\pabs*{y^\beta}\pabs{\phi(y)}dy\\
		&\leq R^{\pabs{\beta}} \norm{f}_{\Wkp[n-1][p][B]} \norm{\phi}_{\Lp[q][B]}\numberthis\label{eq:sob_taylor_coeff},
	\end{align*}
	where we used $B\subset B_{R,\norm{\cdot}_{\ell^\infty}}(0)$ and the H{\"o}lder's inequality with $1/q=1-1/p$. Next, since $\phi\in L^1(B)\cap L^\infty(B)$ and $\norm{\phi}_{L^1}=1$ we have (see \cite[Chapter~X.4 Exercise~4]{amann2008analysis})
\[
	\norm{\phi}_{L^q}\leq \norm{\phi}_{L^1}^{1/q}\norm{\phi}_{L^\infty}^{1-1/q}=\norm{\phi}_{L^\infty}^{1/p}.
\]
	Combining the last estimate with Equation~\eqref{eq:sob_taylor_coeff} yields
	\begin{align*}
		\pabs*{\int_B D^{\gamma+\beta} f(y)y^\beta\phi(y)dy}&\leq R^{n-1} \norm{f}_{\Wkp[n-1][p]}\norm{\phi}_{\Lp[\infty][B]}^{1/p}\\
		&\leq c R^{n-1} \norm{f}_{\Wkp[n-1][p]}r^{-d/p}\numberthis\label{eq:sob_int},
	\end{align*}
	 where the second step follows from $\norm{\phi}_{L^\infty}\leq c r^{-d}$ for some constant $c=c(d)>0$ (see~\cite[Section~4.1]{brenner2007mathematical}).
	To estimate the absolute value of the coefficients $c_\gamma$ (defined in Equation~\ref{eq:c_gamma}), we have
	\begin{align*}
		\pabs{c_\gamma}&\leq \sum_{\pabs{\gamma+\beta}\leq n-1}\frac{1}{(\gamma+\beta)!}\pabs*{a_{(\gamma,\beta)}}\pabs*{\int_B D^{\gamma+\beta} f(y)y^\beta\phi(y)dy}\\
		&\leq  c R^{n-1} \norm{f}_{\Wkp[n-1][p][\Omega]} r^{-d/p}\sum_{\pabs{\gamma+\beta}\leq n-1}\frac{1}{\gamma!\beta!}=c'\norm{f}_{\Wkp[n-1][p][\Omega]} r^{-d/p}.
	\end{align*}
	Here, the second step used Equation \eqref{eq:sob_int} together with Equation \eqref{eq:sob_multinomial}, and $c'=c'(n,d,R)>0$ is a constant.
\end{proof}

The next step is to derive approximation properties of the averaged Taylor polynomial. To this end, recall that for the (standard) Taylor expansion of some function $f$ defined on a domain $\Omega$ in $x_0$ to yield an approximation at some point $x_0 + h$ the whole path $x_0 +th$ for $0\leq t\leq 1$ has to be contained in $\Omega$ (see \cite[Theorem~6.8.10]{marsden1974elementary}). In case of the averaged Taylor polynomial the expansion point $x_0$ is replaced by a ball $B$ and we require that the path between each $x_0\in B$ and each $x\in\Omega$ is contained in $\Omega$. This geometrical condition is made precise in the following definition.
%Definition from \cite[Definition 4.2.2]{brenner2007mathematical}.
\begin{definition}
	Let $\Omega,B\subset \R^d$. Then $\Omega$ is called \emph{star-shaped with respect to $B$} if,
	\[
		\overline{\co} \left(\{x\}\cup B\right)\subset\Omega\quad\text{for all }x\in\Omega.
		\]
\end{definition}

% Inspiration for this passage is the passage after \cite[(10.5.1) Definition]{brenner2007mathematical} where it is said that a mesh is non-degenerate iff the chunkiness parameters of the triangles are uniformly bounded.
The next definition introduces a geometric notion which becomes important when given a family of subdivisions $\mathcal{T}^h$, $0<h\leq 1$ of a domain $\Omega$ which becomes finer for smaller $h$. One typically needs to control the nondegeneracy of $(\mathcal{T}^h)_h$ which can be done e.g.\ with a uniformly bounded chunkiness parameter.
%Definition from \cite[Definition 4.2.16]{brenner2007mathematical}.
\begin{definition}
	Let $\Omega\subset\R^d$ be bounded. We define the set
	\[
		\mathcal{R}:=\left\{r>0\,: \begin{array}{l}
			\text{there exists }x_0\in\Omega \text{ such that }\Omega\text{ is}\\[0.3em]
				\text{star-shaped with respect to }B_{r,\pabs{\cdot}}(x_0)
			\end{array}
			\right\}. 
		\]
		If $\mathcal{R}\neq\emptyset$, then we define
	\[
		\starmax:=\sup\mathcal{R}\quad\text{and call}\quad\gamma:=\frac{\diam(\Omega)}{\starmax}
		\]
	the \emph{chunkiness parameter} of $\Omega$.
\end{definition}

To emphasize the dependence on the set $\Omega$, we will occasionally write $\starmax(\Omega)$ and $\gamma(\Omega)$.

Finally, the next lemma shows approximation properties of the averaged Taylor polynomial.
\begin{lemma}[Bramble-Hilbert]\label{lemma:bramble_hilbert}
	Let $\Omega\subset \R^d$ be open and bounded, $x_0\in\Omega$ and $r>0$ such that $\Omega$ is star-shaped with respect to $B:=B_{r,\pabs{\cdot}}(x_0)$, and $r>(1/2)\starmax$. Moreover, let $n\in\N$, $1\leq p\leq\infty$ and denote by $\gamma$ the chunkiness parameter of $\Omega$. Then there exists a constant $C=C(n,d,\gamma)>0$ such that for all $f\in\Wkp[n]$
	\begin{equation*}
		\pabs{f-Q^n f}_{\Wkp[k]} \leq C h^{n-k}\pabs{f}_{\Wkp[n]} \quad\text{for } k=0,1,\ldots,n,
	\end{equation*}
	where $Q^n f$ denotes the Taylor polynomial of order $n$ of $f$ averaged over $B$ and $h=\diam(\Omega)$.
\end{lemma}

A proof can be found in \cite[Lemma 4.3.8]{brenner2007mathematical}.

\subsection{Fractional Sobolev spaces}
In this subsection, we derive a generalization of Sobolev spaces characterized by fractional-order derivatives using two different approaches. First, we interpolate integer-valued Sobolev spaces, and secondly, we define an intrinsic norm. Both approaches are then shown to be equivalent under some regularity condition on the domain $\Omega$.

We start by defining fractional-order Sobolev spaces $W^{s,p}$ for $0<s<1$ via Banach space interpolation (see~Section~\ref{app:interpolation}).
\begin{definition}\label{def:fractional_interpol}
	Let $0<s<1$ and $1\leq p\leq \infty$, then we set
	\[
		\tWkp[s][p]:=\icouple[\Lp][\Wkp[1]][s].
		\]
\end{definition}

The next theorem shows that, given suitable regularity conditions of the domain $\Omega$, Definition~\ref{def:fractional_interpol} and Definition~\ref{def:slobodeckij} yield the same spaces with equivalent norms. The theorem can be found in~\cite[Theorem 14.2.3]{brenner2007mathematical} for $1\leq p<\infty$. The case $p=\infty$ is not included in~\cite[Theorem 14.2.3]{brenner2007mathematical} but can be shown with the same technique.
\begin{theorem}\label{thm:interpol_eqiv_hoelder}
	Let $\Omega\subset\R^d$ be a Lipschitz-domain. Then we have for $0<s<1$ and $1\leq p\leq \infty$ that
	\[
	\Wkp[s][p]=\wtilde W^{s,p}(\Omega)
	\]
		with equivalence of the respective norms.
\end{theorem}

This motivates the following remark.
\begin{remark*}
	We make use of the norm equivalence by only using the $W^{s,p}$ norm in the following section and transferring results obtained by interpolation theory for the $\wtilde W^{s,p}$ norm to the $W^{s,p}$ norm without mentioning it again.
\end{remark*}

\section{Upper bounds for approximations}\label{app:upper_bounds}
The strategy of the proof of Theorem~\ref{thm:main} follows the idea of the proof of Theorem~1 in~\cite{yarotsky2017error}: A polynomial approximation $P_f$ of a function $f\in\Fndp$ is approximated by a function realized by a neural network~$\Phi_{P_f}$. As in~\cite{yarotsky2017error}, we start by constructing a neural network that approximates the square function on the interval $\cube$ up to an approximation error at most $\eps$. In our result the error is measured in the $W^{1,\infty}$ norm (Proposition~\ref{prop:appro_square}). This result is then used to obtain an approximation of a multiplication operator (Proposition~\ref{prop:approximate_multiplication}). Next, we derive a partition of unity (Lemma~\ref{lemma:partition_of_unity}) as a product of univariate functions that can be realized by neural networks. Using this construction we then build a localized Taylor approximation for a function $f\in\Fndp$ in the $L^p$ and $W^{1,p}$ norm. An interpolation argument shows that same construction can be used for an approximation in the $W^{s,p}$ norm, where $0\leq s\leq 1$ (Lemma~\ref{lemma:polynomial_approximation}). In the next step, we lay the foundation for approximating products of the partition of unity and monomials, i.e.\ localized monomials, with ReLU networks (Lemma~\ref{lemma:network_multiplikation}) in the $W^{k,\infty}$ norm for $k=0,1$. Via an embedding this result this result is then used to approximate the localized polynomials in $W^{k,p}$ with ReLU networks (Lemma~\ref{lemma:network_polynomial_approximation}).

All Sobolev spaces encountered in the following proofs are defined on open, bounded, and convex domains which are, in particular, Lipschitz-domains (see \cite[Corollary 1.2.2.3]{grisvard1985elliptic}).

The results presented in the following two propositions can be found in a similar way in~\cite[Proposition~3.1]{schwab2018deep}. However, since our results contain some minor extensions we decided to give the proof here for the sake of completeness. In detail, Proposition~\ref{prop:approximate_multiplication}~(ii) and~(iii) are not shown in~\cite[Proposition~3.1]{schwab2018deep}.

In \cite[Proposition 2]{yarotsky2017error}, a network is constructed that approximates the square function $x\to x^2$ on the interval $\cube$ in the $L^\infty$ norm. Interestingly, the same construction can be used when measuring the approximation error in the $W^{1,\infty}$ norm. In particular, the depth and the number of weights and neurons of the network do not grow asymptotically faster to satisfy the approximation accuracy with respect to this stronger norm. 
\begin{proposition}\label{prop:appro_square}
	There exist constants $c_1, c_2, c_3,c_4>0$, such that for all $\epsilon\in\epsin$ there is a neural network $\appsq_\eps$ with at most $c_1\cdot \log_2(\nicefrac{1}{\eps})$ nonzero weights, at most $c_2\cdot \log_2(\nicefrac{1}{\eps})$ layers, at most $c_3\cdot\log_2(\nicefrac{1}{\eps})$ neurons, and with one-dimensional input and output such that 
	\begin{equation}\label{eq:appsq_bound}
		\norm{R_\rho(\appsq_\eps)(x)-x^2}_{\Wkpd[1][\infty][\cube]{x}}\leq\eps
	\end{equation}
	and $\act{\appsq_\eps}(0)=0$. Furthermore, it holds that
	\begin{equation}\label{eq:bound_derivative}
		\pabs{R_\rho(\appsq_\eps)}_{\Wkp[1][\infty][\cube]}\leq c_4.
	\end{equation}
\end{proposition}

\begin{proof}
	In the proof of \cite[Proposition 2]{yarotsky2017error} it is shown that there exist constants $c_1$, $c_2$, $c_3>0$, such that for each $m\in\N$ there is a neural network $\Phi_m$ with at most $c_1\cdot m$ nonzero weights, at most $c_2\cdot m$ layers, at most $c_3\cdot m$ neurons the realization of which is a piecewise linear interpolation of $x\mapsto x^2$ on $\intervalo{0}{1}$. In detail, it is shown there, that the network $\Phi_m$ satisfies for all $k\in\{0,\ldots,2^m-1\}$ and all $x\in\interval{\frac{k}{2^m}}{\frac{k+1}{2^m}}$
\begin{equation}\label{eq:Phi_m}
	R_\rho(\Phi_m)(x) =\left(\frac{\left(k+1\right)^2}{2^m} - \frac{k^2}{2^m}\right)\left(x-\frac{k}{2^m}\right) + \left(\frac{k}{2^m}\right)^2.
\end{equation}
		Thus, $R_\rho(\Phi_m)$ is a piecewise linear interpolant of $f$ with $2^m+1$ uniformly distributed breakpoints $\frac{k}{2^m}, k=0,\ldots,2^m$. In particular, $\act{\Phi_m}(0)=0$. Furthermore, it is shown in the proof of~\cite[Proposition~2]{yarotsky2017error} that \begin{equation}\label{eq:square_bound}
			\norm{R_\rho(\Phi_m)(x)-x^2}_{\Lpd[\infty][\cube]}\leq 2^{-2-2m}.
	\end{equation}
	We will now show that the approximation error of the derivative can be bounded in a similar way. In particular, we show the estimate
	\begin{equation}\label{eq:square_derivative_bound}
		\pabs{R_\rho(\Phi_m)-x^2}_{\Wkpd[1][\infty][\cube]{x}}\leq 2^{-m}.
	\end{equation}
	From Equation~\eqref{eq:Phi_m} we get for all $k=0,\ldots,2^m-1$
\begin{align*}
	|R_\rho(\Phi_m)(x) - x^2|_{\Wkpd[1][\infty][\interv]{x}}&= \left\lVert\frac{\left(k+1\right)^2}{2^m} - \frac{k^2}{2^m} - 2x\right\rVert_{\Lpd[\infty][\interv]}\\
	&=\left\lVert \frac{2k+1}{2^m} - 2x\right\rVert_{\Lpd[\infty][\interv]}\\
	&=\max\left\{\left\vert\frac{2k+1}{2^m}-2\frac{k}{2^m}\right\vert, \left\vert\frac{2k+1}{2^m}-2\frac{k+1}{2^m}\right\vert\right\} =2^{-m}.
\end{align*}
	Combining Equation \eqref{eq:square_bound} and \eqref{eq:square_derivative_bound} yields
\[
	\|R_\rho(\Phi_m)(x) - x^2\|_{\Wkpd[1][\infty][\cube]{x}}\leq\max\left\{2^{-2m-2},2^{-m}\right\}=2^{-m}.
\]

Clearly, the weak derivative of $\Phi_m$ is a piecewise constant function, which assumes its maximum on the last piece. Hence,
	\begin{equation}\label{eq:bound_derivative_apmult}
		\pabs{R_\rho(\Phi_m)}_{\Wkp[1][\infty][\cube]}\leq\frac{(2^m)^2-(2^m-1)^2}{2^m}=2-\frac{1}{2^m}\leq 2.
	\end{equation}

	Let now $\eps\in\epsin$ and choose $m=\lceil\log_2(\nicefrac{1}{\eps})\rceil$. Now, $\appsq_\eps:=\Phi_m$ satisfies the approximation bound in Equation \eqref{eq:appsq_bound} and $\act{\appsq_\eps}(0)=0$. The estimate \eqref{eq:bound_derivative} holds because of Equation \eqref{eq:bound_derivative_apmult}. The number of weights can be bounded by 
	\[
		M(\appsq_\eps)\leq c_1\cdot m\leq c_1\cdot(\log_2(\nicefrac{1}{\eps})+1)\leq 2\cdot c_1\cdot \log_2(\nicefrac{1}{\eps}).
		\]
	In the same way, the number of neurons and layers can be bounded. This concludes the proof.

\end{proof}

As in \cite{yarotsky2017error}, we will now use Proposition \ref{prop:appro_square} and the polarization identity
\begin{equation}\label{eq:polarization}
	xy=\frac{1}{2}((x+y)^2-x^2-y^2)\qquad\text{for }x,y\in\R,	
\end{equation}
to define an approximate multiplication where the approximation error is again (and in contrast to \cite{yarotsky2017error}) measured in the $W^{1,\infty}$ norm.

\begin{proposition}\label{prop:approximate_multiplication}
	For any $M\geq1$, there exist constants $c,c_1>0$ and $c_2= c_2(M)>0$ such that for any $\eps \in \epsin$, there is a neural network $\apmult$ with two-dimensional input and one-dimensional output that satisfies the following properties:
	\begin{enumerate}[(i)]
		\item \label{item:network_approximation}$\norm{R_\rho(\apmult)(x,y) - xy}_{\WinfMd}\leq \epsilon$;
		\item \label{item:network_zero}if $x=0$ or $y=0$, then $R_\rho(\apmult) \left(x,y\right) = 0$;
		\item \label{item:network_derivative}$\pabs{R_\rho(\apmult_\eps)}_{\WinfM}\leq c M$;
		\item \label{item:network_complexity_apmult}the depth and the number of weights and neurons in $\apmult$ is at most $c_1 \log_2(1/\epsilon) + c_2$.
	\end{enumerate}
\end{proposition}

\begin{proof}
	Set $\delta:=\eps/(6M^2C)$, where $C$ is the constant from Corollary \ref{cor:composition_norm} for $n=2$ and $m=1$, and let $\appsqdelta$ be the approximate squaring network from Proposition \ref{prop:appro_square} such that  
	\begin{equation}\label{eq:delta_bound}
		\norm{\act{\appsqdelta} - x^2}_{\Wkpd[1][\infty][\cube]{x}}<\delta.
	\end{equation}
	As in the proof of \cite[Proposition~3]{yarotsky2017error}, we use the fact that $\pabs{x}=\rho(x)+\rho(-x)$ to see that a network $\apmult$ can be constructed with two-dimensional input and one-dimensional output that satisfies
\[
	\act{\apmult} (x,y)= 2M^2\left(\act{\appsqdelta}\left(\frac{\pabs{x+y}}{2M}\right)-\act{\appsqdelta}\left(\frac{\pabs{x}}{2M}\right)-\act{\appsqdelta}\left(\frac{\pabs{y}}{2M}\right)\right).
\]
	As a consequence of Proposition \ref{prop:appro_square}, there exists a constant $c_0>0$ such that $\apmult$ has at most $c_0 \log_2(\nicefrac{1}{\eps})+c_0 \log_2(6M^2)+3\leq c_0 \log_2(\nicefrac{1}{\eps}) + c_1$ layers, $3 c_0 \log_2(\nicefrac{1}{\eps})+3c_0 \log_2(6M^2)+9\leq c'\log_2(\nicefrac{1}{\eps})+c_2$ neurons and $3 c_0 \log_2(\nicefrac{1}{\eps})+3c_0 \log_2(6M^2)+17\leq c''\log_2(\nicefrac{1}{\eps})+c_3$ nonzero weights. Here $c',c''>0$ and $c_1=c_1(M),c_2=c_2(M),c_3=c_3(M)>0$ are suitable constants and (\ref{item:network_complexity_apmult}) is hence satisfied.

	Since $\act{\appsqdelta}(0) = 0$, (\ref{item:network_zero}) easily follows.

	Using the polarization identity \eqref{eq:polarization}, we can write
\begin{align*}
xy &= 4M^2\frac{x}{2M}\frac{y}{2M}\\
&= 4M^2\frac{1}{2}\left(\left(\frac{x}{2M}+\frac{y}{2M}\right)^2 - \left(\frac{x}{2M}\right)^2-\left(\frac{y}{2M}\right)^2\right)\\
&= 2M^2\left(\left(\frac{\pabs{x+y}}{2M}\right)^2 - \left(\frac{\pabs{x}}{2M}\right)^2-\left(\frac{\pabs{y}}{2M}\right)^2\right).\numberthis \label{eq:polarization_xy}
\end{align*}
	To keep the following calculations simple, we introduce some notation. In particular, we define
	\[
		u_{xy}:\intervalo{-M}{M}^2\to\cube,\quad (x,y)\mapsto \frac{\pabs{x+y}}{2M},
		\]
	\[
		u_x:\intervalo{-M}{M}^2\to\cube,\quad (x,y)\mapsto \frac{\pabs{x}}{2M},
		\]
		and
	\[
		u_y:\intervalo{-M}{M}^2\to\cube,\quad (x,y)\mapsto \frac{\pabs{y}}{2M}.
		\]
		Setting $f:\cube\to\R$, $x\mapsto x^2$ and using \eqref{eq:polarization_xy} we get
		\begin{equation}
			xy =2M^2\Big(f\circ u_{xy}(x,y) - f\circ u_x(x,y) -f\circ u_y(x,y)\Big).
		\end{equation}
		Now, we can estimate
\begin{align*}
	&\norm{\act{\apmult}(x,y) - xy}_{\WinfMd}\\
	&= \bigg\Vert 2M^2\big(\act{\appsqdelta}\circ u_{xy}-\act{\appsqdelta}\circ u_x-\act{\appsqdelta}\circ u_y\big)\\
	&\alplus-2M^2\big(f\circ u_{xy} - f\circ u_x-f\circ u_y\big)\bigg\Vert_{\WinfM}\\
			&\leq2M^2\Big(\norm{\left(\act{\appsqdelta}-f\right)\circ u_{xy}}_{\WinfM}+\norm{(\act{\appsqdelta}-f)\circ u_x}_{\WinfM}\\
		&\alplus+\norm{(\act{\appsqdelta}-f)\circ u_y}_{\WinfM}\Big).\numberthis\label{eq:chainrule}
\end{align*}
	Note that for the inner functions in the compositions in Equation \eqref{eq:chainrule}, it holds that
	\begin{equation}\label{eq:uxy_derivative_bound}
		\pabs{u_{xy}}_{\WinfM}=\pabs{u_x}_{\WinfM}=\pabs{u_y}_{\WinfM}=\nicefrac{1}{(2M)}.
	\end{equation}
		Hence, to finally prove (\ref{item:network_approximation}), we apply Corollary \ref{cor:composition_norm} to \eqref{eq:chainrule} and get 
	\begin{align*}
	&\norm{\act{\apmult}(x,y) - xy}_{\WinfMd}\\
		&\leq2M^2 3C\max\left\{\norm{\act{\appsqdelta}-f}_{\Lp[\infty][\cube]},\pabs{\act{\appsqdelta}-f}_{\Wkp[1][\infty][\cube]}\cdot \nicefrac{1}{(2M)}\right\}\\
	&\leq 6M^2C \delta=\eps,
\end{align*}
where we used \eqref{eq:delta_bound} in the second step.

	To show (\ref{item:network_derivative}) we use the chain rule estimate from Corollary \ref{cor:composition_norm} and get
	\begin{align*}
		&\pabs{\act{\apmult}}_{\WinfM}\\
		&\alplus\leq 2 M^2 C\sum_{z\in\{x,y,xy\}}\pabs{\act{\appsqdelta}}_{\Wkp[1][\infty][\cube]}\pabs{u_{z}}_{\WinfM}\\
		&\alplus\leq 2 M^2 C3 C'\frac{1}{2M}\alplus=3CC'M,
	\end{align*}
	where we used Equation~\eqref{eq:uxy_derivative_bound} in the third step and $C'$ is the absolute constant from Equation~\eqref{eq:bound_derivative} in~Proposition~\ref{prop:appro_square}.
\end{proof}

Next, we construct a partition of unity (in the same way as in \cite[Theorem 1]{yarotsky2017error}) that can be defined as a product of piecewise linear functions, such that each factor of the product can be realized by a neural network. We will use this product structure in Lemma~\ref{lemma:network_polynomial_approximation} together with a generalized version of the approximate multiplication from Proposition~\ref{prop:approximate_multiplication} to approximate localized polynomials with ReLU networks.

\begin{lemma}\label{lemma:partition_of_unity}
	For any $d,N\in \N$ there exists a collection of functions 
	\[
		\Psi=\left\{\phi_m:m\in\MNd\right\}
	\]
	with $\phi_m:\R^d\to\R$ for all $m\in\MNd$ with the following properties:
	\begin{enumerate}[(i)]
		\item \label{item:pou_01} $0\leq \phi_m(x)\leq 1$ for every $\phi_m \in\Psi$ and every $x\in \R^d$; 
		\item  \label{item:pou_sum} $\sum_{\phi_m\in\Psi}\phi_m(x)=1$ for every $x\in[0,1]^d$;
		\item \label{item:pou_supp} $\supp \phi_m\subset B_{\nicefrac{1}{N},\norm{\cdot}_{\ell^\infty}}(\nicefrac{m}{N})$ for every $\phi_m \in \Psi$;
		\item \label{item:pou_derivative}  there exists a constant $c\geq 1$ such that $\norm{\phi_m}_{\Wkp[k][\infty][\R^d]}\leq (c\cdot N)^k$ for $k\in\{0,1\}$;
		\item \label{item:pou_network} there exist absolute constants $C,c\geq 1$ such that for each $\phi_m\in\Psi$ there is a neural network $\Phi_m$ with $d$-dimensional input and $d$-dimensional output, with at most three layers, $C d$ nonzero weights and neurons, that satisfies
			\[
	\prod_{l=1}^d [\act{\Phi_m}]_l=\phi_m \quad
				\]
			and $\norm{[\act{\Phi_m}]_l}_{\Wkp[k][\infty][\cube^d]}\leq (cN)^k$ for all $l=1,\ldots,d$ and $k\in\{0,1\}$.
	\end{enumerate}
\end{lemma}

\begin{proof}
	As in \cite{yarotsky2017error}, we define the functions
	\[
		\psi:\R\to\R,\qquad\psi(x):=\begin{cases}1, & \pabs{x}<1,\\0, &2<\pabs{x},\\2-\pabs{x}, &1\leq\pabs{x}\leq 2,\end{cases}
		\]
		and $\phi_m:\R^d\to\R$ as a product of scaled and shifted versions of $\psi$. Concretely, we set
	\begin{equation}\label{eq:partition_unity}
		\phi_m(x):=\prod_{l=1}^d \psi\left(3N\left(x_l-\frac{m_l}{N}\right)\right),
	\end{equation}
	for $m=(m_1,\ldots,m_d)\in\MNd$. Then, (\ref{item:pou_01}),(\ref{item:pou_sum}) and (\ref{item:pou_supp}) follow easily from the definition.
	
	To show (\ref{item:pou_derivative}), note that $\norm{\phi_m}_{L^\infty}\leq 1$ follows already from (\ref{item:pou_01}). It now suffices to show the claim for the $W^{1,\infty}$ semi-norm. For this, let $l\in\{1,\ldots,d\}$ and $x\in\R^d$, then
	\begin{align*}
		\pabs*{\ddx{x_l}\phi_m(x)}&=\pabs*{\prod_{i=1,i\neq l}^d\psi\left(3N\left(x_i-\frac{m_l}{N}\right)\right)}\pabs*{\psi'\left(3N\left(x_l-\frac{m_l}{N}\right)\right)3N}\\
		&\leq 3N.
	\end{align*}
	It follows that $\pabs{\phi_m}_{W^{1,\infty}}\leq 3N$.

	To show (\ref{item:pou_network}), we start by constructing a network $\Phi_{\psi}$ that realizes the function $\psi$. For this we set
	\[
		A_1:=\bmat{c}{
			 1\\[1em]
			 1\\[1em]
			 1\\[1em]
			 1
		}, \quad
		b_1:=\bmat{c}{
			 2\\[1em]
			 1\\[1em]
			 -1\\[1em]
			 -2
		} \quad\text{and}\quad 
		A_2:=\bmat{c|c c c c}{
			 0&1&-1&-1&1
		},\quad
		b_2:=0,
		\]
		and $\Phi_{\psi}:=((A_1,b_1),(A_2,b_2))$. Then $\Phi_{\psi}$ is a two-layer network with one-dimensional input and one-dimensional output, with $12$ nonzero weights and $6$ neurons such that
		\[
			\act{\Phi_{\psi}}(x)=\psi(x)\quad\text{for all}\quad x\in\R.
			\]
			We denote by $\Phi_{m,l}$ the one-layer network with $d$-dimensional input and one-dimensional output, with $2$ nonzero weights and $d+1$ neurons such that $\act{\Phi_{m,l}}=3N(x_l-m_l/N)$ for all $x\in\R^d$ and $m\in\MNd$, $l=1,\ldots,d$. Finally, we define
			\[
			\Phi_m:=P(\Phi_{\psi}\odot\Phi_{m,l}:l=1,\ldots,d),
				\]
				which is a three-layer network with $d$-dimensional input and $d$-dimensional output, with at most $d\cdot2\cdot(12+2)=d\cdot 28$ nonzero weights and at most $9d$ neurons, and
				\[
				[\act{\Phi_m}]_l(x)=\act{\Phi_{\psi}\odot\Phi_{m,l}}(x)=\psi\left(3N\left(x_l-\frac{m_l}{N}\right)\right)
				\]
				for $l=1,\ldots,d$ and $x\in\R^d$. Clearly, it follows that $\prod_{l=1}^d [\act{\Phi_m}]_l(x)=\phi_m(x)$ for all $x\in\R^d$. The last part of (\ref{item:pou_network}) can be shown similarly as (\ref{item:pou_derivative}).
\end{proof}

The following lemma uses the partition of unity from Lemma~\ref{lemma:partition_of_unity} and the Bramble-Hilbert Lemma~\ref{lemma:bramble_hilbert} in a classical way to derive an approximation with localized polynomials in the $L^p$ norm and the $W^{1,p}$ norm. Using an interpolation argument, we can generalize this result to the case where the approximation is performed with respect to the $W^{s,\infty}$ norm, where $0\leq s\leq 1$.

\begin{lemma}\label{lemma:polynomial_approximation}
	Let $d,N\in \N$, $n\in\N_{\geq 2}$, $1\leq p\leq \infty$ and $\Psi=\Psi(d,N)=\left\{\phi_m:m\in\MNd\right\}$ be the partition of unity from Lemma \ref{lemma:partition_of_unity}. Then there is a constant $C=C(d,n,p)>0$ such that for any $f\in\Wkp[n][p][\cube^d]$, there exist polynomials $p_{f,m}(x)=\sum_{|\alpha|\leq n-1}c_{f,m,\alpha}x^\alpha$ for $m\in\MNd$ with the following properties:
	
Let $0\leq\fraccoef\leq 1$ and set $f_N:=\sum_{m\in\MNd}\phi_m p_{f,m}$, then the operator $T_\fraccoef:\Wkp[n][p][\cube^d]\to\Wkp[s][p][\cube^d]$ with $T_\fraccoef f=f-f_N$ is linear and bounded with
		\[
			\norm{T_\fraccoef f}_{\Wkp[s][p][\cube^d]}\leq C\left(\frac{1}{N}\right)^{n-\fraccoef}\norm{f}_{\Wkp[n][p][\cube^d]}.
			\]
	Furthermore, there is a constant $c=c(d,n)>0$ such that for any $f\in\Wkp[n][p][\cube^d]$ the coefficients of the polynomials $p_{f,m}$ satisfy 
	\[
		\pabs{c_{f,m,\alpha}}\leq c\norm{\tilde f}_{\Wkp[n][p][\Omega_{m,N}]}N^{d/p}
		\]
		for all $\alpha\in\N^d_0$ with $\pabs{\alpha}\leq n-1$ and $m\in\{0,\ldots,N\}^d$, where $\Omega_{m,N}:=B_{\frac{1}{N},\norm{\cdot}_{\ell^\infty}}\left(\frac{m}{N}\right)$ and $\tilde f\in\Wkp[n][p][\R^d]$ is an extension of $f$.
\end{lemma}

\begin{proof}
	The idea of the proof is similar to the first part of the proof of \cite[Theorem 1]{yarotsky2017error}. We use approximation properties of averaged Taylor polynomials (see Bramble-Hilbert Lemma~\ref{lemma:bramble_hilbert}) to derive local estimates and then combine them using a partition of unity to obtain a global estimate. In order to use this strategy also near the boundary, we make use of an extension operator.

For this, let $E:\Wkp[n][p][\cube^d]\to\Wkp[n][p][\R^d]$ be the extension operator from \cite[Theorem VI.3.1.5]{stein2016singular} and set $\tilde f:=Ef$. Note that for arbitrary $\Omega\subset\R^d$ and $1\leq k\leq n$ it holds
	\begin{equation}\label{eq:extension_bound}
		\pabs[\big]{\tilde f}_{\Wkp[k][p][\Omega]} \leq \norm[\big]{\tilde f}_{\Wkp[n][p][\R^d]} \leq C_E \norm{f}_{\Wkp[n][p][\cube^d]},
	\end{equation}
	where $C_E=C_E(n,p,d)$ is the norm of the extension operator. 

	\textbf{Step 1 (Averaged Taylor polynomials)}: For each $m\in \MNd$ we set 
	\[
		\Omega_{m,N}:=B_{\frac{1}{N},\norm{\cdot}_{\ell^\infty}}\Big(\frac{m}{N}\Big) \quad \text{and}\quad B_{m,N} := B_{\frac{3}{4N},\pabs{\cdot}}\Big(\frac{m}{N}\Big),
	\]and denote by $p_m=p_{f,m}$ the Taylor polynomial of order $n$ of $\tilde f$ averaged over $B_{m,N}$ (cf.\ Definition~\ref{def:taylor}). It follows from Proposition~\ref{prop:taylor_is_polynom} (for $\Omega=\Omega_{m,N}$, $B=B_{m,N}$ and $R=2$) that we can write $p_m=\sum_{|\alpha|\leq n-1}c_{m,\alpha}x^\alpha$ and that there is a constant $c'=c'(n,d)>0$ such that
	\[
		\pabs{c_{m,\alpha}}\leq c'\norm[\big]{\tilde f}_{\Wkp[n][p][\Omega_{m,N}]}\left(\frac{3}{4N}\right)^{-d/p} \leq c''\norm{\tilde f}_{\Wkp[n][p][\Omega_{m,N}]}N^{d/p}
	\]
		for $m\in\MNd$, where $c''=c''(n,d,p)>0$ is a suitable constant. It now suffices to show (i) and~(ii).
	
	\textbf{Step 2 (Local estimates in $\norm{\cdot}_{W^{k,p}}, k\in\{0,1\}$)}: To check that the conditions of the Bramble-Hilbert Lemma \ref{lemma:bramble_hilbert} are fulfilled, note that $B_{m,N}\comp\Omega_{m,N}$. Furthermore, $B_{m,N}$ is a ball in $\Omega_{m,N}$ such that $\Omega_{m,N}$ is star-shaped with respect to $B_{m,N}$. We have $\diam_{\pabs{\cdot}}(\Omega_{m,N})=(2\sqrt d)/N$ and $\starmax(\Omega_{m,N})=1/N$ and, thus,
	\[
		r_{\pabs{\cdot}}\left(B_{m,N}\right)=\frac{3}{4N}>\frac{1}{2}\cdot\frac{1}{N}=\frac{1}{2}\cdot\starmax(\Omega_{m,N}).
	\]Finally, we have for the chunkiness parameter of $\Omega_{m,N}$ 
	\begin{equation}\label{eq:chunkiness}
		\gamma(\Omega_{m,N})=\diam(\Omega_{m,N})\cdot\frac{1}{\starmax(\Omega_{m,N})}=\frac{2\sqrt d}{N}\cdot N=2\sqrt d.
	\end{equation}
	Applying the Bramble-Hilbert Lemma \ref{lemma:bramble_hilbert} yields for each $m\in\MNd$ the local estimate
	\begin{align*}
		\norm[\big]{\tilde f-p_m}_{\Lp[p][\Omega_{ m,N}]}&\leq C_1\left(\frac{2\sqrt d}{N}\right)^n \pabs[\big]{\tilde f}_{\Wkp[n][p][\Omega_{m,N}]}\leq C_2 \left(\frac{1}{N}\right)^n\norm{\tilde f}_{\Wkp[n][p][\Omega_{m,N}]}.\numberthis\label{eq:local_Linf_bound}
	\end{align*}
	Here, $C_1=C_1(n,d)>0$ is the constant from Lemma \ref{lemma:bramble_hilbert} which only depends on $n$ and $d$, since the chunkiness parameter of $\Omega_{m,N}$ is a constant depending only on $d$ (see \eqref{eq:chunkiness}) and~$C_2 = C_2(n,d)>0$.
	In the same way, we get
	\begin{equation}\label{eq:local_Winf_bound}
		\pabs[\big]{\tilde f-p_m}_{\Wkp[1][p][\Omega_{ m,N}]}\leq C_3\left(\frac{1}{N}\right)^{n-1}\norm{\tilde f}_{\Wkp[n][p][\Omega_{m,N}]},
	\end{equation}
	where $C_3=C_3(n,d)>0$ is a suitable constant. 

	The first step towards a global estimate is now to combine Equation \eqref{eq:local_Linf_bound} and \eqref{eq:local_Winf_bound} with the cut-off functions from the partition of unity. We have
	\begin{align*}
		\norm[\big]{ \phi_m(\tilde f-p_m)}_{\Lp[p][\Omega_{m,N}]}&\leq \norm{\phi_m}_{\Lp[\infty][\Omega_{m,N}]}\cdot\norm[\big]{ \tilde f-p_m}_{\Lp[p][\Omega_{m,N}]}\\
		\expl{Lemma \ref{lemma:partition_of_unity} (\ref{item:pou_derivative}), Equation \eqref{eq:local_Linf_bound}}&\leq C_2\left(\frac{1}{N}\right)^n\norm{\tilde f}_{\Wkp[n][p][\Omega_{m,N}]}\numberthis\label{eq:part_unity_Linf_bound}.
	\end{align*}
	Furthermore, using the product inequality for weak derivatives from Lemma~\ref{lemma:product_rule_bound_p} we get that there is a constant $C'=C'(d,p)>0$ such that
	\begin{align}
		\pabs[\big]{\phi_m(\tilde f-p_m)}_{\Wkp[1][p][\Omega_{m,N}]}&\leq C'\pabs{\phi_m}_{\Wkp[1][\infty][\Omega_{m,N}]}\cdot\norm[\big]{\tilde f-p_m}_{\Lp[p][\Omega_{m,N}]}\nonumber\\
		&\alplus+C'\norm{\phi_m}_{\Lp[\infty][\Omega_{m,N}]}\cdot\pabs[\big]{\tilde f-p_m}_{\Wkp[1][p][\Omega_{m,N}]}\nonumber\\
		&\leq C'\cdot\tilde cN\cdot C_2 \left(\frac{1}{N}\right)^n\norm{\tilde f}_{\Wkp[n][p][\Omega_{m,N}]} + C'\cdot C_3 \left(\frac{1}{N}\right)^{n-1}\norm{\tilde f}_{\Wkp[n][p][\Omega_{m,N}]}\nonumber\\
		&=C_4 \left(\frac{1}{N}\right)^{n-1}\norm{\tilde f}_{\Wkp[n][p][\Omega_{m,N}]}, \label{eq:part_unity_secondbound}
	\end{align}
	where the first part of the second inequality follows from Lemma \ref{lemma:partition_of_unity} (\ref{item:pou_derivative}) with $k=1$ and an absolute constant $\tilde c\geq 1$, together with Equation \eqref{eq:local_Linf_bound} and the second part from Lemma~\ref{lemma:partition_of_unity}~(\ref{item:pou_derivative}) with $k=0$, together with Equation \eqref{eq:local_Winf_bound}. Here, $C_4 = C_4(n,d,p)>0$.
	Now it easily follows from \eqref{eq:part_unity_Linf_bound} and \eqref{eq:part_unity_secondbound} that
	\begin{equation}\label{eq:part_unity_Winf_bound}
		\norm[\big]{\phi_m(\tilde f -p_m)}_{\Wkp[1][p][\Omega_{m,N}]}\leq C_5\left(\frac{1}{N}\right)^{n-1}\norm{\tilde f}_{\Wkp[n][p][\Omega_{m,N}]},
	\end{equation}
	for some constant $C_5=C_5(n,d,p)>0$.

	\textbf{Step 3 (Global estimate in $\norm{\cdot}_{W^{k,p}}, k\in\{0,1\}$)}: To derive the global estimate, we start by noting that with property (\ref{item:pou_sum}) from Lemma \ref{lemma:partition_of_unity} we have
		\begin{equation}\label{eq:sumup}
			\tilde f(x)=\sum_{ m\in\MNd}\phi_m(x) \tilde f(x), \quad \text{for a.e.\ }x\in\cube^d.
		\end{equation} 
		%\todo{maybe skip the intersection}
		% Formerly we $\Omega_{\wtilde m,N}\cap\cube^d$ from the last Equation of the next block on
		Using that $\tilde f $ is an extension of $f$ on $\cube^d$ we can write for $k\in\{0,1\}$
	\begin{align*}
		\norm*{f -\sum_{m\in\MNd} \phi_m p_m}_{\Wkp[k][p][\cube^d]}^p&=\norm*{\tilde f -\sum_{m\in\MNd} \phi_m p_m}_{\Wkp[k][p][\cube^d]}^p\\
		\expl{Equation \eqref{eq:sumup}}&=\norm*{\sum_{m\in\MNd} \phi_m(\tilde f-p_m)}_{\Wkp[k][p][\cube^d]}^p\\
		&\leq\sum_{\wtilde m\in\MNd} \norm*{\sum_{m\in\MNd} \phi_m(\tilde f-p_m)}_{\Wkp[k][p][\Omega_{\wtilde m,N}]}^p,\numberthis\label{eq:global_max}
	\end{align*}
	where the last step follows from $\cube^d\subset\bigcup_{\wtilde m\in\MNd}\Omega_{\wtilde m,N}$. Now we obtain for each $\wtilde m\in\MNd$ 
	\begin{align*}
		\norm*{\sum_{m\in\MNd} \phi_m(\tilde f-p_m)}_{\Wkp[k][p][\Omega_{\wtilde m,N}]}&\leq \sum_{\substack{m\in\MNd,\vspace{0.2em}\vspace{0.2em}\\ \norm{m-\wtilde m}_{\ell^\infty} \leq 1}} \norm{\phi_m(\tilde f-p_m)}_{\Wkp[k][p][\Omega_{\wtilde m,N}]}\\
		&\leq \sum_{\substack{m\in\MNd,\vspace{0.2em}\vspace{0.2em}\\ \norm{m-\wtilde m}_{\ell^\infty} \leq 1}} \norm{\phi_m(\tilde f-p_m)}_{\Wkp[k][p][\Omega_{m,N}]}\\
		&\leq C_6 \left(\frac{1}{N}\right)^{n-k}\sum_{\substack{m\in\MNd,\vspace{0.2em}\\\norm[\tiny]{m-\wtilde m}_{\ell^\infty} \leq 1}} \norm{\tilde f}_{\Wkp[n][p][\Omega_{m,N}]}\numberthis\label{eq:max_neighbors}
	\end{align*}
	where the triangle inequality together with the support property~(\ref{item:pou_supp}) from~Lemma~\ref{lemma:partition_of_unity} is used in the first step. The second step follows again from~Lemma~\ref{lemma:partition_of_unity}~(\ref{item:pou_supp}) and the third step follows from \eqref{eq:part_unity_Linf_bound} for $k=0$ and from \eqref{eq:part_unity_Winf_bound} for $k=1$. Here $C_6=C_6(n,d,p)>0$ can be chosen independent of $k$ (e.g.\ $C_6:=\max\{C_2,C_5\}$). 

	Finally, the boundedness claim in (i) and (ii) follows from using the definition of $f_N$ and combining Equation \eqref{eq:global_max} with Equation \eqref{eq:max_neighbors}:
	\begin{align*}
		\norm{f-f_N}_{\Wkp[k][p][\cube^d]}^p&\leq\sum_{\wtilde m\in\MNd}C_6^p \left(\frac{1}{N}\right)^{(n-k)p} \left(\sum_{\substack{m\in\MNd,\vspace{0.2em}\\\norm[\tiny]{m-\wtilde m}_{\ell^\infty} \leq 1}} \norm{\tilde f}_{\Wkp[n][p][\Omega_{m,N}]}\right)^p\\
		\expl{\text{H{\"o}lder's inequality};\;$q:=1-1/p$}&\leq C_6^p \left(\frac{1}{N}\right)^{(n-k)p} \sum_{\wtilde m\in\MNd}\sum_{\substack{m\in\MNd,\vspace{0.2em}\\\norm[\tiny]{m-\wtilde m}_{\ell^\infty} \leq 1}} \norm{\tilde f}_{\Wkp[n][p][\Omega_{m,N}]}^p (3^d)^{p/q}\\
		&\leq C_6^p 3^{dp/q} \left(\frac{1}{N}\right)^{(n-k)p} 3^d \sum_{\wtilde m\in\MNd} \norm{\tilde f}_{\Wkp[n][p][\Omega_{\wtilde m,N}]}^p \\
		&\leq C_6^p 3^{dp/q} \left(\frac{1}{N}\right)^{(n-k)p} 3^d 2^d \norm{\tilde f}_{\Wkp[n][p][\bigcup_{\wtilde m\in\MNd}\Omega_{\wtilde m,N}]}^p,
	\end{align*}
	where the last two steps follow from the definition of $\Omega_{\wtilde m,N}$. Thus, we have
	\[
		\norm{f-f_N}_{\Wkp[k][p][\cube^d]}\leq C_7\left(\frac{1}{N}\right)^{n-k}\norm{\tilde f}_{\Wkp[n][p][\R^d]}\leq C_8 \left(\frac{1}{N}\right)^{n-k}\norm{f}_{\Wkp[n][p][\cube^d]}
		\]
	for $k\in\{0,1\}$, where Equation~\eqref{eq:extension_bound} was used in the first and second step. Here $C_7=C_7(n,d,p)>0$ and $C_8=C_8(n,d,p)$ are constants. The linearity of $T_k$, $k\in\{0,1\}$ is a consequence of the linearity of the averaged Taylor polynomial (cf.\ Remark \ref{remark:sob_taylor_lin}).

\textbf{Step 4 (Interpolation)}: For $0<s<1$ we use a Banach space interpolation argument. Set 
	\[
		A_0=A_1=\Wkp[n][p][\cube^d]
		\]
		together with
		\[
			B_0=\Lp[p][\cube^d]\qquad\text{and}\qquad B_2=\Wkp[1][p][\cube^d].
			\]
	Then, we can apply Theorem~\ref{thm:interpolation_operator} to the operator $T$ with $Tf:=f-f_N$ and get for a constant $C=C(n,d,p)>0$ that
	\begin{align*}
		\norm{T}_{\Lcal(W^{n,p}, W^{s,p})}&\leq \norm{T}_{\Lcal(W^{n,p}, L^p)}^{1-s} \norm{T}_{\Lcal(W^{n,p}, W^{1,p})}^s\\
		&\leq C\cdot\left(\frac{1}{N}\right)^{n(1-s)}\left(\frac{1}{N}\right)^{(n-1)s}=C\cdot\left(\frac{1}{N}\right)^{n-s},
	\end{align*}
	where we used Lemma \ref{lemma:interpolation_relations} (\ref{item:interpolation_identity}) to see that $(A_0,A_1)_{s,p}=\Wkp[n][p][\cube^d]$.
\end{proof}

The following technical lemma lays the foundation for approximating localized monomials with neural networks. Using the notation and statement (\ref{item:pou_network}) from Lemma~\ref{lemma:partition_of_unity}, a localized monomial $\phi_m x^\alpha$ can be expressed as the product of the output components of a network~$\Phi_{(m,\alpha)}$ with a suitable output dimension $n$, i.e.\ 
\[
	\phi_m(x) x^\alpha=\prod_{l=1}^n \big[\actbig{\Phi_{(m,\alpha)}}\big]_l(x).
	\]
	Given a network $\Phi$ with $n$-dimensional output, we construct a network $\Psi_{\Phi}$ that approximates the product of the output components of $\Phi$.

%\todo{the problem with the constant which depends on $s$ and the network which doesn't could also be solved by using two constants and $c_1^s \cdot c_2^{1-s}$, or we could write (independent of $s$) behind the netork. See \cite[Proof of Lemma A.5]{petersen2017optimal}}

\begin{lemma}\label{lemma:network_multiplikation}
Let $d,m,K\in\N$ and $N\geq 1$ be arbitrary. Then there is a constant $C=C(m)>0$ such that the following holds:

	For any $\eps\in \epsin$, and any neural network $\Phi$ with $d$-dimensional input and $n$-dimensional output where $n\leq m$, and with number of layers, neurons and weights all bounded by $K$, such that 
	\[
		\norm{[\act{\Phi}]_l}_{\Wkp[k][\infty][\cube^d]}\leq N^k \quad \text{for}\quad k\in\{0,1\}\text{ and }l=1,\ldots,n
		\]
		there exists a neural network $\Psi_{\eps,\Phi}$ with $d$-dimensional input and one-dimensional output, and with number of layers, neurons and weights all bounded by $K C\log_2(\nicefrac{1}{\eps})$, such that
	\begin{equation}\label{eq:approximation_s}
		\norm*{\act{\Psi_{\eps,\Phi}}-\prod_{l=1}^n [\act{\Phi}]_l}_{\Wkp[k][\infty][\cube^d]} \leq c  N^k \eps
	\end{equation}
	for $k\in\{0,1\}$ and some constant $c=c(d,m,k)$. Moreover, 
	\begin{equation}\label{eq:zero_multiplication}
		\act{\Psi_{\eps,\Phi}}(x)=0\quad \text{if}\quad \prod_{l=1}^n [\act{\Phi}]_l(x)=0
	\end{equation}
	for $x\in\cube^d$.
\end{lemma}

\begin{proof}
	% We make the constant explicit here, because it makes it easier in the composition, where we have to have an interval for the range of the first function.
	Let $d,K\in \N$ and $N\geq 1$. We show by induction over $m\in\N$ that the statement holds. To make the induction argument easier we will additionally show that $c=c(d,m,k)=m^{1-k}c_1^k$, where $c_1=c_1(d,m)>0$, and that the network $\Psi_{\eps,\Phi}$ can be chosen such that the first $L(\Phi)-1$ layers of $\Psi_{\eps,\Phi}$ and $\Phi$ coincide and $\pabs{\act{\Psi_{\eps,\Phi}}}_{\Wkp[1][\infty][\cube^d]}\leq C_1N$ for a constant $C_1=C_1(d,m)>0$.

	If $m=1$, then we can choose $\Psi_{\eps,\Phi}=\Phi$ for any $\eps\in\epsin$ and the claim holds. 
	
	Now, assume that the claim holds for some $m\in\N$. We show that it also holds for $m+1$. For this, let $\eps \in\epsin$ and let $\Phi=((A_1,b_1),(A_2,b_2),\dots,(A_L,b_L))$ be a neural network with $d$-dimensional input and $n$-dimensional output, where $n\leq m+1$, and with number of layers, neurons and weights all bounded by $K$, where each $A_l$ is an $N_l\times \sum_{k=0}^{l-1} N_k$ matrix, and $b_l\in\R^{N_l}$ for $l=1,\ldots L$.
	
	\textbf{Case 1}: If $n\leq m$, then we use the induction hypothesis and get that there is a constant $C_0=C_0(m)>0$ and a neural network $\Psi_{\eps,\Phi}$ with $d$-dimensional input and one-dimensional output, and at most $KC_0\log_2(\nicefrac{1}{\eps})$ layers, neurons and weights such that 
	\begin{equation*}
		\norm*{\act{\Psi_{\eps,\Phi}}-\prod_{l=1}^n [\act{\Phi}]_l}_{\Wkp[k][\infty][\cube^d]} \leq m^{1-k}c_1^k N^k \eps\leq (m+1)^{1-k}c_1^k N^k\eps
	\end{equation*}
	for $k\in\{0,1\}$ and $c_1=c_1(d,m)$. Moreover, 
	\begin{equation*}
		\act{\Psi_{\eps,\Phi}}(x)=0\quad \text{if}\quad \prod_{l=1}^n [\act{\Phi}]_l(x)=0,
	\end{equation*}
	for any $x\in\cube^d$.
	Furthermore, we have $\pabs{\act{\Psi_{\eps,\Phi}}}_{\Wkp[1][\infty][\cube^d]}\leq C_1 N$, for $C_1=C_1(d,m)$.
	
	\textbf{Case 2}: Now, we assume that $n=m+1$ and show the claim for constants $\tilde C_0,\tilde c_1$ and~$\tilde C_1$ depending on $m+1$, possibly different from the constants $C_0,c_1$ and $C_1$ from Case~1, respectively. The maximum of each pair of constants fulfills then the claim for networks with output dimension $n\leq m+1$.
	%for possibly different constants $C,c_1$ and $C_1$ depending on $m+1$. Taking the maximum of each two constants shows the claim for $n\leq m+1$.
	
	We denote by $\Phi_m$ the neural network with $d$-dimensional input and $m$-dimensional output which results from $\Phi$ by removing the last output neuron and corresponding weights. In detail, we write 
	\[
		A_L=\bmat{c}{
			A_L^{(1,m)}\\[1em]
			a_L^{(m+1)}
		}\quad\text{and}\quad
		b_L=
		\bmat{c}{
			b_L^{(1,m)}\\[1em]
			b_L^{(m+1)}
		},
		\] where $A_L^{(1,m)}$ is a $m\times\sum_{k=0}^{L-1}N_k$ matrix and $a_L^{(m+1)}$ is a $1\times\sum_{k=0}^{L-1}N_k$ vector, and $b_L^{(1,m)}\in\R^m$ and $b_L^{(m+1)}\in\R^1$. Now we set 
		\[
			\Phi_m:=\Big((A_1,b_1),(A_2,b_2),\dots,(A_{L-1},b_{L-1}),\Big(A_L^{(1,m)},b_L^{(1,m)}\Big)\Big).
			\]
		Using the induction hypothesis and the constants $C_0,c_1$ and $C_1$ from Case~1, we get that there is a neural network $\Psi_{\eps,\Phi_m}=((A'_1,b'_1),(A'_2,b'_2),\dots,(A'_{L'},b'_{L'}))$ with $d$-dimensional input and one-dimensional output, and at most $KC_0\log_2(\nicefrac{1}{\eps})$ layers, neurons and weights such that 
	\begin{equation}\label{eq:induction_bound}
		\norm*{\act{\Psi_{\eps,\Phi_m}}-\prod_{l=1}^m [\act{\Phi_m}]_l}_{\Wkp[k][\infty][\cube^d]} \leq m^{1-k}c_1^k N^k \eps
	\end{equation}
	for $k\in\{0,1\}$. Moreover, 
	\begin{equation}\label{eq:induction_zero}
		\act{\Psi_{\eps,\Phi_m}}(x)=0\quad \text{if}\quad \prod_{l=1}^m [\act{\Phi_m}]_l(x)=0,
	\end{equation}
	for any $x\in\cube^d$.
	Furthermore, we can assume that $\pabs{\act{\Psi_{\eps,\Phi_m}}}_{\Wkp[1][\infty][\cube^d]}\leq C_1 N$, and that the first $L(\Phi)-1$ layers of $\Psi_{\eps,\Phi_m}$ and $\Phi_m$ coincide and, thus, also the first $L(\Phi)-1$ layers of $\Psi_{\eps,\Phi_m}$ and $\Phi$, i.e.\ $A_l=A'_l$ for $l=1,\ldots,L(\Phi)-1$. 
	
	Now, we add the formerly removed neuron with corresponding weights back to the last layer of $\Psi_{\eps,\Phi_m}$. For the resulting network
	\[
		\widetilde \Psi_{\eps,\Phi}:=\left((A'_1,b'_1),(A'_2,b'_2),\dots,(A'_{L'-1},b'_{L'-1}),
		\left(
		\bmat{c c}{
			\multicolumn{2}{c}{A'_{L'}}\\[1em]
			a^{(m+1)}_L & 0_{\R^{1,\sum_{k=L}^{L'}N_L'}}
		},
	\bmat{c}{
		 b'_{L'}\\[1em]
		 b_L^{(m+1)}
	}\right)\right)
		\]
		it holds that the first $L-1$ layers of $\widetilde \Psi_{\eps,\Phi}$ and $\Phi$ coincide, and $\widetilde \Psi_{\eps,\Phi}$ is a neural network with two-dimensional output. Note that
		\begin{align*}
			&\norm*{\big[\actbig{\wtilde\Psi_{\eps,\Phi}}\big]_1}_{\Linfc}\\
			&\alplus =\norm{\act{\Psi_{\eps,\Phi_m}}}_{\Linfc}\\
			&\alplus\leq \norm*{\act{\Psi_{\eps,\Phi_m}}-\prod_{l=1}^m [\act{\Phi_m}]_l}_{\Linfc} +\norm*{\prod_{l=1}^m [\act{\Phi_m}]_l}_{\Linfc}\\
			&\alplus\leq m \eps + 1< m+1,
		\end{align*}
		where we used Equation \eqref{eq:induction_bound} for $k=0$. Additionally, we have 
		\[
 				\norm*{\big[\actbig{\wtilde\Psi_{\eps,\Phi}}\big]_2}_{\Linfc}=\norm{[\act{\Phi}]_{m+1}}_{\Linfc}\leq 1.
			\]
			Now, we denote by $\apmult$ the network from Proposition~\ref{prop:approximate_multiplication} with $M=m+1$ and accuracy $\eps$ and define 
		\[
			\Psi_{\eps,\Phi}:=\apmult\odot\wtilde \Psi_{\eps,\Phi}.
			\] 
			Consequently, $\Psi_{\eps,\Phi}$ has $d$-dimensional input, one-dimensional output and, combining the induction hypothesis with statement (\ref{item:network_complexity_apmult}) of Proposition~\ref{prop:approximate_multiplication} and Remark \ref{remark:neural_sparse_concat}, at most 
			\[
				2KC_0\log_2(\nicefrac{1}{\eps}) + 2(c'\log_2(\nicefrac{1}{\eps}) +c'')\leq KC\log_2(\nicefrac{1}{\eps})
			\]
	layers, number of neurons and weights. Here $c'$ and $c''=c''(m+1)$ are the constants from Proposition~\ref{prop:approximate_multiplication} (\ref{item:network_complexity_apmult}) for the choice $M=m+1$ and $C=C(m+1)>0$ is a suitable constant. Clearly, the first $L-1$ layers of $\Psi_{\eps,\Phi}$ and $\Phi$ coincide and for the approximation properties it holds that 
	\begin{align*}
		&\pabs*{\act{\Psi_{\eps,\Phi}}-\prod_{l=1}^{m+1} [\act{\Phi}]_l}_{\Wkp[k][\infty][\cube^d]}\\
		&\alplus = \pabs*{\act{\apmult}\circ\actbig{\wtilde\Psi_{\eps,\Phi}}-[\act{\Phi}]_{m+1}\cdot\prod_{l=1}^{m} [\act{\Phi}]_l}_{\Wkp[k][\infty][\cube^d]}\\
		&\alplus \leq \pabs[\Big]{\act{\apmult}\circ(\act{\Psi_{\eps,\Phi_m}},[\act{\Phi}]_{m+1}) -\act{\Psi_{\eps,\Phi_m}}\cdot[\act{\Phi}]_{m+1}}_{\Wkp[k][\infty][\cube^d]}\\
		&\alplus \alplus+\pabs*{[\act{\Phi}]_{m+1}\cdot\Big(\act{\Psi_{\eps,\Phi_m}}-\prod_{l=1}^{m} [\act{\Phi}]_l\Big)}_{\Wkp[k][\infty][\cube^d]},\numberthis\label{eq:induction_add_zero}
	\end{align*}
	for $k\in\{0,1\}$.
	We continue by considering the first term of the Inequality \eqref{eq:induction_add_zero} for $k=0$ and obtain
	\begin{align*}
		&\norm[\Big]{\act{\apmult}\circ(\act{\Psi_{\eps,\Phi_m}},[\act{\Phi}]_{m+1}) -\act{\Psi_{\eps,\Phi_m}}\cdot[\act{\Phi}]_{m+1}}_{\Lp[\infty][\cube^d]}\\
		&\alplus\leq\norm{\act{\apmult}(x,y)-x\cdot y}_{L^\infty(\intervalo{-(m+1)}{m+1}^2;dxdy)} \leq \eps,\numberthis\label{eq:1term_0k}
	\end{align*}
	where we used Proposition~\ref{prop:approximate_multiplication}~(i) for the last step. Next, we consider the same term for $k=1$ and apply the chain rule from Corollary \ref{cor:composition_norm}. For this, let $\hat C=\hat C(d)$ be the constant from Corollary \ref{cor:composition_norm} (for $n=d$ and $m=2$). We get
	\begin{align*}
		&\pabs[\Big]{\act{\apmult}\circ(\act{\Psi_{\eps,\Phi_m}},[\act{\Phi}]_{m+1}) -\act{\Psi_{\eps,\Phi_m}}\cdot[\act{\Phi}]_{m+1}}_{\Wkp[1][\infty][\cube^d]}\\
		&\alplus\leq \hat C\cdot\pabs{\act{\apmult}(x,y)-x\cdot y}_{\Wkp[1][\infty][\intervalo{-(m+1)}{m+1}^2;dxdy]} \pabs*{\actbig{\wtilde\Psi_{\eps,\Phi}}}_{\Wkpm[1][\infty][\cube^d][2]}\\
		&\alplus\leq \hat C\cdot\eps\max\{C_1 N,N\}=C_1' \eps N,\numberthis\label{eq:1term_1k}
	\end{align*}
	where we used the induction hypothesis together with $\pabs{[\act{\Phi}]_{m+1}}_{\Wkp[1][\infty][\cube^d]}\leq N$ in the third step, and $C_1'=C_1'(d,m+1)>0$ is a suitable constant.
	
	To estimate the second term of \eqref{eq:induction_add_zero} for $k=0$ we use the induction hypothesis (for $k=0$) and get
	\begin{align*}
		&\norm*{[\act{\Phi}]_{m+1}\cdot\Big(\act{\Psi_{\eps,\Phi_m}}-\prod_{l=1}^{m} [\act{\Phi}]_l\Big)}_{\Lp[\infty][\cube^d]}\\
		&\alplus \leq \norm{[\act{\Phi}]_{m+1}}_{\Lp[\infty][\cube^d]}\cdot\norm*{\act{\Psi_{\eps,\Phi_m}}-\prod_{l=1}^{m} [\act{\Phi}]_l}_{\Linfc}\leq 1\cdot m\cdot\eps.\numberthis\label{eq:2term_0k}
	\end{align*}
	For $k=1$ we apply the product rule from Lemma~\ref{lemma:product_rule_bound_p} together with $\norm{[\act{\Phi}]_{m+1}}_{\Li}\leq 1$ and get
	\begin{align*}
		&\pabs*{[\act{\Phi}]_{m+1}\cdot\Big(\act{\Psi_{\eps,\Phi_m}}-\prod_{l=1}^{m} [\act{\Phi}]_l\Big)}_{\Wkp[1][\infty][\cube^d]}\\
		&\alplus \leq \pabs{[\act{\Phi}]_{m+1}}_{\Wkp[1][\infty][\cube^d]}\cdot\norm*{\act{\Psi_{\eps,\Phi_m}}-\prod_{l=1}^{m} [\act{\Phi}]_l}_{\Linfc}\\
		&\alplus\alplus + \norm{[\act{\Phi}]_{m+1}}_{\Linfc}\cdot\pabs*{\act{\Psi_{\eps,\Phi_m}}-\prod_{l=1}^{m} [\act{\Phi}]_l}_{\Wkp[1][\infty][\cube^d]}\\
		&\alplus\leq N\cdot m\eps+1\cdot c_1 N\eps=c_1'N\eps,\numberthis\label{eq:2term_1k}
	\end{align*}
	where we used the induction hypothesis for $k=1$, and $c_1'=c_1'(d,m+1)>0$.

	Combining \eqref{eq:induction_add_zero} with \eqref{eq:1term_0k} and \eqref{eq:2term_0k} we have
	\begin{equation}\label{eq:final_s_0}
		\norm*{\act{\Psi_{\eps,\Phi}}-\prod_{l=1}^{m+1} [\act{\Phi}]_l}_{\Lp[\infty][\cube^d]}\leq \eps+m\cdot \eps=(m+1)\cdot\eps,
	\end{equation}
	and in the same way a combination of \eqref{eq:induction_add_zero} with \eqref{eq:1term_1k} and \eqref{eq:2term_1k} yields
	\begin{equation*}
		\pabs*{\act{\Psi_{\eps,\Phi}}-\prod_{l=1}^{m+1} [\act{\Phi}]_l}_{\Wkp[1][\infty][\cube^d]}\leq (C_1'+c_1')\cdot N\cdot \eps=c_1'' N\eps,
	\end{equation*}
	where $c_1''=c_1''(d,m+1)>0$. Putting together the two previous estimates yields
	\begin{equation}\label{eq:final_s_1}
		\norm*{\act{\Psi_{\eps,\Phi}}-\prod_{l=1}^{m+1} [\act{\Phi}]_l}_{\Wkp[1][\infty][\cube^d]}\leq c_1''' N\eps,
	\end{equation}
	for a suitable constant $c_1'''=c_1'''(d,m+1)>0$. 	

	We now show Equation \eqref{eq:zero_multiplication} for $m+1$. To this end, assume that $[\act{\Phi}]_l(x)=0$ for some $l\in\{1,\ldots,m+1\}$ and $x\in\cube^d$. If $l\leq m$, then Equation \eqref{eq:induction_zero} implies that 
	\[
		\big[\actbig{\wtilde\Psi_{\eps,\Phi}}\big]_1(x)=\act{\Psi_{\eps,\Phi_m}}(x)=0.
	\]
	If $l=m+1$, then we have 
	\[
		\big[\actbig{\wtilde\Psi_{\eps,\Phi}}\big]_2(x)=[\act{\Phi}(x)]_{m+1}(x)=0.
	\]
	Hence, by application of Proposition \ref{prop:approximate_multiplication}, we have 
	\[
		\act{\Psi_{\eps,\Phi}}(x)=\act{\apmult}\left(\big[\actbig{\wtilde\Psi_{\eps,\Phi}}\big]_1(x),\big[\actbig{\wtilde\Psi_{\eps,\Phi}}\big]_2(x)\right)=0.
		\]
	Finally, we need to show that there is a constant $C_1''=C_1''(d,m+1)>0$ such that 
	\[
		\pabs{\act{\Psi_{\eps,\Phi}}}_{\Wkp[1][\infty][\cube^d]}\leq C_1'' N.
		\]
		Similarly as in \eqref{eq:1term_1k} we have for a constant $C_1''=C_1''(d,m+1)>0$ that
	\begin{align*}
		\pabs{\act{\Psi_{\eps,\Phi}}}_{\Wkp[1][\infty][\cube^d]}&=\pabs*{\act{\apmult}\circ\actbig{\wtilde\Psi_{\eps,\Phi}}}_{\Wkp[1][\infty][\cube^d]}\\
		&\leq \hat C\cdot \pabs{\act{\apmult}}_{\Wkp[1][\infty][\intervalo{-(m+1)}{m+1}^2]}\cdot\pabs*{\actbig{\wtilde\Psi_{\eps,\Phi}}}_{\Wkpm[1][\infty][\cube^d][2]}\\
		&\leq \hat C\cdot \hat c\cdot(m+1)\cdot\max\left\{C_1 N,N\right\}= C_1''\cdot N,
	\end{align*}
	where Corollary \ref{cor:composition_norm} was used for the second step and $\hat c$ is the constant from Proposition~\ref{prop:approximate_multiplication}~(\ref{item:network_derivative}) which together with an argument as in \eqref{eq:1term_1k} implies the third step.

Taking the maximum of the each pair of constants derived in Case~1 and Case~2 concludes the proof.
\end{proof}

In the next lemma, we approximate a sum of localized polynomials with a neural network. One of the difficulties is to control the derivative of the localizing functions from the partition of unity.
%\todo{Maybe use interpolation inequality erst in diesem Result, da nicht klar ist (oder gezeigt werden muss) dass man die interpolation norm aufsplitten kann}
\begin{lemma}\label{lemma:network_polynomial_approximation}
	Let $d,N\in \N$, $n\in\N_{\geq 2}$, $1\leq p\leq\infty$, and $0\leq s \leq 1$, and let  $\Psi = \Psi(d,N)=\left\{\phi_m:m\in\MNd\right\}$ be the partition of unity from Lemma \ref{lemma:partition_of_unity}. Then, there are constants $C_1=C_1(n,d,p,s)>0$ and $C_2=C_2(n,d),C_3=C_3(n,d)>0$ with the following properties:

	For any $\eps\in \epsin$ there is a neural network architecture $\Arch_\eps=\Arch_\eps(d,n,N,\eps)$ with $d$-dimensional input and one-dimensional output, with at most $C_2\cdot\log_2(\nicefrac{1}{\eps})$ layers and $C_3 (N+1)^d \log_2(\nicefrac{1}{\eps})$ weights and neurons, such that the following holds: Let $f\in\Wkp[n][p][\cube^d]$ and $p_m(x):=p_{f,m}(x)=\sum_{\pabs{\alpha}\leq n-1} c_{m,\alpha}x^\alpha$ for $m\in\MNd$ be the polynomials from Lemma~\ref{lemma:polynomial_approximation}, then there is a neural network $\Phi_{P,\eps}$ that has architecture $\Arch_\eps$ such that 
		\[
		\norm*{\sum_{m\in\MNd}\phi_m p_m -\act{\Phi_{P,\eps}}}_{\Wkp[s][p][\cube^d]}\leq C_1 \norm{f}_{\Wkp[n][p][\cube^d]} N^s \eps.
		\]
\end{lemma}

\begin{proof}
	\textbf{Step 1 (Approximating localized monomials $\phi_m x^\alpha$)}: Let $|\alpha|\leq n-1$ and $m \in \{0, \dots, N \}^d$. It is easy to see that there is a neural network $\Phi_\alpha$ with $d$-dimensional input and $\pabs{\alpha}$-dimensional output, with one layer and at most $n-1$ weights and $n-1+d$ neurons such that
	\[
		x^\alpha=\prod_{l=1}^{\pabs{\alpha}}[\act{\Phi_\alpha}]_l(x)\quad\text{for all } x\in\cube^d 
		\]
		and
	\begin{equation}\label{eq:phi_poly_bound}
		\norm{[\act{\Phi_\alpha}]_l}_{\Wkp[k][\infty][\cube^d]}\leq 1 \quad \text{for all }l=1,\ldots,\pabs{\alpha}\text{ and } k\in\{0,1\}. 
	\end{equation}
	Let now $\Phi_m$ be the neural network and $C,c\geq 1$ the constants from Lemma \ref{lemma:partition_of_unity} (\ref{item:pou_network}) and define the network
	\[
		\Phi_{m,\alpha}:=P(\Phi_m,\Phi_\alpha).
		\]
		Then $\Phi_{m,\alpha}$ has at most $3\leq K_0$ layers, $Cd+n-1\leq K_0$ nonzero weights, and $Cd+n-1+d\leq K_0$ neurons for a suitable constant $K_0=K_0(n,d)\in\N$, and $\prod_{l=1}^{\pabs{\alpha}+d}[\act{\Phi_{m,\alpha}}]_l(x)=\phi_m x^\alpha$ for all $x\in\cube^d$. Moreover, as a consequence of Lemma \ref{lemma:partition_of_unity} (\ref{item:pou_network}) together with Equation \eqref{eq:phi_poly_bound} we have 
	\begin{equation*}
		\norm{[\act{\Phi_{m,\alpha}}]_l}_{\Wkp[k][\infty][\cube^d]}\leq (c N)^k \quad\text{for all }l=1,\ldots,\pabs{\alpha}+d\text{ and } \text{for } k\in\{0,1\}. 
	\end{equation*}
	To construct an approximation of the localized monomials $\phi_m x^\alpha$, let $\Psi_{\eps,(m,\alpha)}$ be the neural network provided by Lemma \ref{lemma:network_multiplikation} (with $\Phi_{m,\alpha}$ instead of $\Phi$, $m=n-1+d\in\N$, $K=K_0\in\N$ and $cN$ instead of $N$) for $m\in\{0,\ldots,N\}^d$ and $\alpha\in\N_0^d,\pabs{\alpha}\leq n-1$. There exists a constant $C_1=C_1(n,d)\geq 1$ such that $\Psi_{\eps,(m,\alpha)}$ has at most $C_1\log_2(\nicefrac{1}{\eps})$ layers, number of neurons and weights. Moreover, 
	\begin{equation}\label{eq:approx_mon_approx}
		\norm*{\phi_m(x) x^\alpha -\actbig{\Psi_{\eps,(m,\alpha)}}(x)}_{\Wkpd[k][\infty][\cube^d]{x}}\leq c' N^k \eps
	\end{equation} for a constant $c'=c'(n,d,k)>0$ and $k\in\{0,1\}$, and
	\begin{equation}\label{eq:support_mon_approx}
		\actbig{\Psi_{\eps,(m,\alpha)}}(x)=0\quad\text{if}\quad \phi_m(x)x^\alpha=0\quad\text{for all }x\in\cube^d.
	\end{equation}

\textbf{Step 2 (Constructing an architecture capable of approximating sums of localized polynomials)}:
	We set 
	\[
		M:=\pabs{\{(m,\alpha):m\in\MNd, \alpha\in\N_0^d,\pabs{\alpha}\leq n-1\}}
		\]
		and define the matrix $\As\in\R^{1\times M}$ by $\As:=[c_{m,\alpha} :m\in\MNd, \alpha\in\N_0^d,\pabs{\alpha}\leq n-1]$ and the neural network $\Phis:=((\As,0))$. Finally, we set 
	\[
		\PhiPs:=\Phis\odot P\big(\Psi_{\eps,(m,\alpha)}:m\in\MNd, \alpha\in\N_0^d,\pabs{\alpha}\leq n-1\big).
		\]
		Then, there are constants $C_2=C_2(n,d),C_3=C_3(n,d)>0$ such that $\PhiPs$ is a neural network with $d$-dimensional input and one-dimensional output, with at most $1+C_1\log_2(\nicefrac{1}{\eps})\leq C_2 \log_2(\nicefrac{1}{\eps})$ layers, 
		\[
			2(M+M C_1\log_2(\nicefrac{1}{\eps}))\leq 4MC_1\log_2(\nicefrac{1}{\eps})\leq C_3 (N+1)^d \log_2(\nicefrac{1}{\eps})
		\]
	nonzero weights and neurons, and
		\[
			\act{\PhiPs}=\sum_{m\in\MNd}\sum_{|\alpha|\leq n-1}c_{m,\alpha}\Psi_{\eps,(m,\alpha)}.
		\]

	Note that the network $\Phi_{P,\eps}$ only depends on $p_{f,m}$ (and thus on $f$) via the coefficients $c_{m,\alpha}$. Now, it is easy to see that there exists a neural network architecture $\Arch_\eps=\Arch_\eps(d,n,N,\eps)$ with $L(\Arch_\eps)\leq C_2\log_2(\nicefrac{1}{\eps})$ layers and number of neurons and weights bounded by $C_3(N+1)^d\log_2(\nicefrac{1}{\eps})$ such that $\Phi_{P,\eps}$ has architecture $\Arch_\eps$ for every of choice of coefficients $c_{m,\alpha}$, and hence for every choice of $f$.
	
	\textbf{Step 3 (Estimating the approximation error in $\norm{\cdot}_{W^{k,p}}, k\in\{0,1\}$)}:
	For each $m\in \MNd$ we set 
	\[
		\Omega_{m,N}:=B_{\frac{1}{N},\norm{\cdot}_\infty}\Big(\frac{m}{N}\Big)
		\]
		and we get for $k\in\{0,1\}$
		\begin{align*}
		&\norm*{\sum_{m\in\MNd}\phi_m(x) p_m(x) -\act{\Phi_{P,\eps}}(x)}_{\Wkpd[k][p][\cube^d]{x}}^p\\
		&\alplus=\norm*{\sum_{m\in\MNd}\sum_{|\alpha|\leq n-1}c_{m,\alpha}\Big(\phi_m(x) x^\alpha -\actbig{\Psi_{\eps,(m,\alpha)}}(x)\Big)}_{\Wkpd[k][p][\cube^d]{x}}^p\\
			&\alplus\leq\sum_{\wtilde m\in\MNd}\norm*{\sum_{m\in\MNd}\sum_{|\alpha|\leq n-1}c_{m,\alpha}\Big(\phi_m(x) x^\alpha -\actbig{\Psi_{\eps,(m,\alpha)}}(x)\Big)}_{\Wkpd[k][p][\Omega_{\wtilde m,N}\cap\cube^d]{x}}^p\numberthis\label{eq:estimate_outer_sum},
	\end{align*}
	where the last step is a consequence of $\cube^d\subset\bigcup_{\wtilde m\in\MNd}\Omega_{\wtilde m,N}$. For $\wtilde m\in\MNd$ we have
	\begin{align*}
		&\norm*{\sum_{m\in\MNd}\sum_{|\alpha|\leq n-1}c_{m,\alpha}\Big(\phi_m(x) x^\alpha -\actbig{\Psi_{\eps,(m,\alpha)}}(x)\Big)}_{\Wkpd[k][p][\Omega_{\wtilde m,N}\cap\cube^d]{x}}\\
		&\alplus\leq\sum_{m\in\MNd}\sum_{|\alpha|\leq n-1}\pabs{c_{m,\alpha}}\norm*{\phi_m(x) x^\alpha -\actbig{\Psi_{\eps,(m,\alpha)}}(x)}_{\Wkpd[k][p][\Omega_{\wtilde m,N}\cap\cube^d]{x}}\\
		&\alplus\leq c_1 \sum_{m\in\MNd} \sum_{\pabs{\alpha}\leq n-1}\norm{\tilde f}_{\Wkp[n-1][p][\Omega_{m,N}]}N^{d/p} \norm*{\phi_m(x) x^\alpha -\actbig{\Psi_{\eps,(m,\alpha)}}(x)}_{\Wkpd[k][p][\Omega_{\wtilde m,N}\cap\cube^d]{x}}\numberthis\label{eq:approximate_loc_taylor_with_network},
	\end{align*}
	where we used the triangle inequality in the first step and Lemma~\ref{lemma:polynomial_approximation} in the second step. Here $c_1=c_1(n,d)>0$ is a constant. Next, note that 
	\begin{align*}
		&\norm*{\phi_m(x) x^\alpha -\actbig{\Psi_{\eps,(m,\alpha)}}(x)}_{\Wkpd[k][p][\Omega_{\wtilde m,N}\cap\cube^d]{x}}\\
		&\alplus\leq \lambda\left(\Omega_{\wtilde m,N}\cap\cube^d\right)^{1/p}(d+1)^{1/p}\norm*{\phi_m(x) x^\alpha -\actbig{\Psi_{\eps,(m,\alpha)}}(x)}_{\Wkpd[k][\infty][\Omega_{\wtilde m,N}\cap\cube^d]{x}}\\
		&\alplus\leq c_3 \left(\frac{1}{N}\right)^{d/p}\norm*{\phi_m(x) x^\alpha -\actbig{\Psi_{\eps,(m,\alpha)}}(x)}_{\Wkpd[k][\infty][\Omega_{\wtilde m,N}\cap\cube^d]{x}},
	\end{align*}
	where $\lambda$ denotes the Lebesgue measure and $c_3=c_3(d,p)>0$ is a constant. Combining Equation~\eqref{eq:approximate_loc_taylor_with_network} with the last estimate yields
	\begin{align*}
		&\norm*{\sum_{m\in\MNd}\sum_{|\alpha|\leq n-1}c_{m,\alpha}\Big(\phi_m(x) x^\alpha -\actbig{\Psi_{\eps,(m,\alpha)}}(x)\Big)}_{\Wkpd[k][p][\Omega_{\wtilde m,N}\cap\cube^d]{x}}\\
		&\alplus\leq c_4 \sum_{m\in\MNd} \sum_{\pabs{\alpha}\leq n-1}\norm{\tilde f}_{\Wkp[n-1][p][\Omega_{m,N}]}\norm*{\phi_m(x) x^\alpha -\actbig{\Psi_{\eps,(m,\alpha)}}(x)}_{\Wkpd[k][\infty][\Omega_{\wtilde m,N}\cap\cube^d]{x}}\\
		&\alplus\leq c_4 \sum_{\substack{m\in\MNd,\vspace{0.2em}\\\norm[\tiny]{m-\wtilde m}_{\ell^\infty} \leq 1}} \sum_{\pabs{\alpha}\leq n-1}\norm{\tilde f}_{\Wkp[n-1][p][\Omega_{m,N}]}\norm*{\phi_m(x) x^\alpha -\actbig{\Psi_{\eps,(m,\alpha)}}(x)}_{\Wkpd[k][\infty][\Omega_{\wtilde m,N}\cap\cube^d]{x}}\\
		&\alplus\leq c_4 c' N^k\eps \sum_{\substack{m\in\MNd,\vspace{0.2em}\\\norm[\tiny]{m-\wtilde m}_{\ell^\infty} \leq 1}} \sum_{\pabs{\alpha}\leq n-1}\norm{\tilde f}_{\Wkp[n-1][p][\Omega_{m,N}]}\numberthis\label{eq:estimate_inner_block}
	\end{align*}
	where we used Lemma~\ref{lemma:partition_of_unity} (\ref{item:pou_supp}) together with Equation~\eqref{eq:support_mon_approx} in the second step and Equation~\eqref{eq:approx_mon_approx} in the last step. Here, $c_4=c_4(n,d,p)>0$ is a constant. Now, using that $\pabs{\{\alpha:\alpha\in\N_0^d,\pabs{\alpha}\leq n-1\}}\leq (n-1)^d$ and H{\"o}lder's inequality shows 
	\begin{align*}
		\sum_{\substack{m\in\MNd,\vspace{0.2em}\\\norm[\tiny]{m-\wtilde m}_{\ell^\infty} \leq 1}} \sum_{\pabs{\alpha}\leq n-1}\norm{\tilde f}_{\Wkp[n-1][p][\Omega_{m,N}]}&\leq  (n-1)^d  \sum_{\substack{m\in\MNd,\vspace{0.2em}\\\norm[\tiny]{m-\wtilde m}_{\ell^\infty} \leq 1}} \norm{\tilde f}_{\Wkp[n-1][p][\Omega_{m,N}]}\\
		\expl{H{\"o}lder's inequality; $q:=1-1/p$}&\leq (n-1)^d 3^{d/q} \left(\sum_{\substack{m\in\MNd,\vspace{0.2em}\\\norm[\tiny]{m-\wtilde m}_{\ell^\infty}\leq 1}} \norm{\tilde f}_{\Wkp[n-1][p][\Omega_{m,N}]}^p\right)^{1/p}\numberthis\label{eq:estimate_sum_with_hoelder}.
	\end{align*}
	Combining Equation~\eqref{eq:estimate_inner_block} with Equation~\eqref{eq:estimate_sum_with_hoelder} and plugging the result in Equation~\eqref{eq:estimate_outer_sum} finally yields
	\begin{align*}
		&\norm*{\sum_{m\in\MNd}\phi_m(x) p_m(x) -\act{\Phi_{P,\eps}}(x)}_{\Wkpd[k][p][\cube^d]{x}}^p\\
		\alplus&\leq \left(c_4 c' 3^{d/q} (n-1)^d N^k\eps \right)^p\sum_{\wtilde m\in\MNd} \sum_{\substack{m\in\MNd,\vspace{0.2em}\\\norm[\tiny]{m-\wtilde m}_{\ell^\infty}\leq 1}} \norm{\tilde f}_{\Wkp[n-1][p][\Omega_{m,N}]}^p\\
		&\leq c_5 N^{kp}\eps^p\norm{f}_{\Wkp[n][p][\cube^d]}^p,\numberthis\label{eq:k_p}
	\end{align*}
	where the last step is the same as Step 3 of the proof of Lemma~\ref{lemma:polynomial_approximation} and $c_5=c_5(n,d,p,k)>0$ is a constant. Hence the case $s=0$ and $s=1$ is proven.
	
	\textbf{Step 4 (Interpolation)}: To show the general statement for $0\leq s\leq 1$ we use the interpolation inequality from Corollary~\ref{cor:interpol_norm} together with Equation~\eqref{eq:k_p} and directly get
	\[
		\norm*{\sum_{m\in\MNd}\phi_m(x) p_m(x) -\act{\Phi_{P,\eps}}(x)}_{\Wkpd[s][p][\cube^d]{x}}\leq c_6 N^s\eps\norm{f}_{\Wkp[n][p][\cube^d]}
		\]
		for a constant $c_6=c_6(n,p,d,s)>0$. This concludes the proof of the lemma.
\end{proof}

Finally, we are ready to proof the upper complexity bounds.
\begin{proof}[Proof of Theorem {\rm\ref{thm:main}}]
	The proof can be divided into two steps: First, we approximate the function $f$ by a sum of localized polynomials and then approximate this sum by a network.

	For the first step, we set
	\begin{equation}\label{eq:big_N}
		N:=\ceil*{\left(\frac{\eps}{2CB}\right)^{-1/(n-s)}},
	\end{equation}
	where $C=C(n,d,p)>0$ is the constant from Corollary~\ref{lemma:polynomial_approximation}. Without loss of generality we may assume that $CB\geq 1$. The same corollary yields that if $\Psi=\Psi(d,N)=\left\{\phi_m:m\in\MNd\right\}$ is the partition of unity from Lemma \ref{lemma:partition_of_unity}, then there exist polynomials $p_m(x)=\sum_{\pabs{\alpha}\leq n-1}c_{m,\alpha}x^\alpha$ for $m\in\{0,\ldots,N\}^d$ such that
	\begin{align*}
		\norm*{f-\sum_{m\in\{0,\ldots,N\}^d}\phi_m p_m}_{\Wkp[s][p][\cube^d]}&\leq CB\left(\frac{1}{N}\right)^{n-s}\\
		&\leq CB \frac{\eps}{2CB}=\frac{\eps}{2}\numberthis\label{eq:final_triangle_1}.
	\end{align*}

	For the second step, let $C_1=C_1(n,d,p,s)>0$, $C_2=C_2(n,d)>0$ and $C_3=C_3(n,d)>0$ be the constants from Lemma~\ref{lemma:network_polynomial_approximation} and $\Phi_{P,\eps}$ be the neural network provided by Lemma \ref{lemma:network_polynomial_approximation} with $\eps^{n/(n-s)}/(8CB^2 C_1)$ instead of $\eps$. Then $\Phi_{P,\eps}$ has at most 
	\[
		C_2\log_2\left(8CB^2 C_1\eps^{-n/(n-s)}\right)\leq C_2\left(\log_2(8CB^2 C_1)+\log_2\big(\eps^{-n/(n-s)}\big)\right)\leq C' \log_2\left(\eps^{-n/(n-s)}\right)
		\]
			layers for a constant $C'=C'(n,d,p,B,s)>0$ and at most
	\begin{align*}
		&C_3\left(\left(\frac{\eps}{2CB}\right)^{-1/(n-s)}+2\right)^d\log_2\left(8CB^2 C_1\eps^{-n/(n-s)}\right)\\
		&\alplus\leq C_3 3^d \left(\frac{\eps}{2CB}\right)^{-d/(n-s)}\log_2\left(8CB^2 C_1\eps^{-n/(n-s)}\right)\\
		&\alplus\leq C''\eps^{-d/(n-s)}\log_2\left(\eps^{-n/(n-s)}\right)
	\end{align*}
		nonzero weights and neurons. Here $C''=C''(n,d,p,B,s)$ is a suitable constant and we used $(2CB)/\eps\geq 1$ in the first step. It holds that the architecture of $\Phi_{P,\eps}$ is independent of the function $f$. Furthermore, we have
	\begin{align*}
		\norm*{\sum_{m\in\MNd}\phi_m p_m -\act{\Phi_{P,\eps}}}_{\Wkp[s][p][\cube^d]}&\leq C_1 B N^s \frac{\eps^{n/(n-s)}}{8CB^2 C_1}\\
		&\leq \left(\left(\frac{\eps}{2CB}\right)^{-1/(n-s)}+1\right)^s\frac{\eps^{n/(n-s)}}{8CB}\\
		&\leq \left(\left(\frac{\eps}{2CB}\right)^{-s/(n-s)}+1\right)\frac{\eps^{n/(n-s)}}{8CB},
	\end{align*}
	where we used the inequality $(x+y)^\sigma\leq x^\sigma+y^\sigma$ for $x,y\geq 0$ and $0\leq \sigma\leq 1$ in the last step. We now continue the above computation using that $s/(n-s)\leq 1$ and $2CB\geq 1$ and therefore $(2CB)^{s/(n-s)}\leq 2CB$ and get
	\begin{align*}
		\left(\left(\frac{\eps}{2CB}\right)^{-s/(n-s)}+1\right)\frac{\eps^{n/(n-s)}}{8}&\leq \frac{1}{8CB}\left(2CB \eps+\eps^{n/(n-s)}\right)\\
		&\leq \frac{1}{8CB}\left(2CB \eps+2CB \eps\right)=\frac{\eps}{2}.
	\end{align*}
	Combining the previous computations we get
	\begin{equation}\label{eq:final_triangle_2}
		\norm*{\sum_{m\in\MNd}\phi_m p_m -\act{\Phi_{P,\eps}}}_{\Wkp[s][p][\cube^d]}\leq \frac{\eps}{2}.
	\end{equation}
	Using the triangle inequality and Equations \eqref{eq:final_triangle_1} and \eqref{eq:final_triangle_2} we finally obtain
		\begin{align*}
			&\norm*{f-\act{\Phi_{P,\eps}}}_{\Wkp[s][p][\cube^d]}\\
			&\alplus\leq \norm*{f-\sum_{m\in\MNd}\phi_m p_m}_{\Wkp[s][p][\cube^d]}+\norm*{\sum_{m\in\MNd}\phi_m p_m-\act{\Phi_{P,\eps}}}_{\Wkp[s][p][\cube^d]}\\
			&\alplus \leq\frac{\eps}{2}+\frac{\eps}{2}=\eps,
		\end{align*}
		which concludes the proof.
\end{proof}

\section{Lower bounds for approximations}\label{app:lower_bounds}
Lower complexity bounds for approximations in $L^\infty$ norm, which correspond to the case $k=0$ in Theorem~\ref{thm:lower_bound}, have already been shown in Theorem~4~a) in \cite{yarotsky2017error}. For lower complexity bounds of $W^{1,\infty}$ approximations we modify the proof strategy outlined in~\cite[Theorem~4~a)]{yarotsky2017error}. 

We start by showing an auxiliary result, that is used in the proof of Proposition~\ref{prop:lower_bounds}.
\begin{lemma}\label{lemma:lower_bounds_affine_set}
	Let $d\in\N$ and $\Phi$ be a neural network with $d$-dimensional input and one-dimensional output. Moreover, let $x\in \cube^d$ and $\nu\in\R^d$. Then, there exists an open set $T=T(x,\nu)\subset\cube^d$ and $\delta=\delta(x,\nu,T)>0$ with $x+\lambda\delta\nu\in \overline{T}$ for $0\leq \lambda\leq 1$ and $\act{\Phi}$ is affine-linear on $T$.
\end{lemma}

\begin{proof}
	We start by defining the set $U:=\{x\in\R^d:\act{\Phi}\text{ is affine-linear on a neighborhood of }$x$\}$. Standard results on the number of pieces of ReLU neural networks \cite{montufar2014number} yield that $U$ has only finitely many polyhedral, connected components, $(V_i)_{i=1}^k$ for some $k\in\N$, with $U=\bigcup_{i=1}^k V_i$ and $\R^d=\bigcup_{i=1}^k\overline{V_i}$.	Note that if follows from the definition of $U$ that $V_i$ is open for $i=1,\ldots,k$.
	
	Now, set $x_n:=x+(1/n)\nu$. By the pigeonhole principle, there exists $q\in\{1,\ldots,k\}$, such that $\overline{V_q}$ contains infinitely many $x_n$. It is not hard to see that if a closed polyhedron contains a converging sequence on a line, then it also contains a small section of the line including the limit point of the sequence. Thus, there exists $\delta>0$ such that $\{x+\lambda\delta\nu:0\leq \lambda\leq 1\}\subset \overline{V_q}\cap\cube^d\subset\overline{V_q\cap\cube^d}$. Then, setting $T:=V_q\cap\cube^d$ shows the claim.
\end{proof}

As in~\cite{yarotsky2017error}, we make use of a combinatorial quantity measuring the expressiveness of a set of binary valued functions $H$ defined on some set $X$, called \emph{VC-dimension} (see e.g.\ \cite[Chapter 3.3]{anthony2009neural}). We define
\[
	\vcdim(H):=\sup\left\{m\in\N\,: \begin{array}{l}
		\text{there exist }x_1,\ldots,x_m\in X\text{ such that for any }y\in\{0,1\}^m\\[0.3em]
		\text{there is a function }h\in H \text{ with }h(x_i)=y_i\text{ for }i=1,\ldots,m
	\end{array}
		\right\}.
	\]

The idea of the next proposition is to relate the approximation error $\eps$ with the number of weights $M(\Arch_\eps)$ of an architecture $\Arch_\eps$ capable of realizing such an approximation. To this end, we construct a set of functions $H$ parameterized by elements of $\R^{M(\Arch_\eps)+1}$. 

If $w\in\R^{M(\Arch_\eps)}$ and $\delta>0$ is chosen appropriately, then a directional derivative of the function realized by $\Arch_\eps(w)$ is computed for the evaluation of $h((w,\delta),\cdot)\in H$. By exploiting the approximation capacity of derivatives of functions realized by $\Arch_\eps(w)$ we can find a lower bound for $\vcdim(H)$ depending on $\eps$ (Claim~1).

%Using finite differences functions of $H$ makes use of the architecture $\Arch_\eps$ (Step~1). Using finite differences we can access the derivative of a neural network with architecture $\Arch_\eps$ and then make use of the approximation capacity of this derivative to lower bound $\vcdim(H)$ (Claim~1).

On the other hand, \cite[Theorem~8.4]{anthony2009neural} yields an upper bound of the VC-dimension of $H$ in terms of the number of computations and the dimension of the parameterization of $H$ which can be expressed as a function of $M(\Arch_\eps)$ (Claim~2). Together this gives the desired relation.
\begin{proposition}\label{prop:lower_bounds}
	Let $d\in\N$, $n\in\N_{\geq 2}$ and $B>0$. Then, there are constants $c=c(n,B)>0$ and $C=C(d)$ with the following property:

	Let $N\in\N,0 <\eps\leq cN^{-(n-1)}$ and $\Arch_\eps=\Arch_\eps(d,n,\eps)$ be a neural network architecture with $d$-dimensional input and one-dimensional output such that for any $f\in\Fnd$ there is a neural network $\Phi_\eps^f$ that has architecture $\Arch_\eps$ and 
	\begin{equation}\label{eq:approximation_lower_bound}
		\norm*{\act{\Phi_\eps^f}-f}_{\Wkp[1][\infty][\cube^d]}\leq \eps,
	\end{equation}
	then 
	\[
		N^d\leq C\cdot M(\Arch_\eps)^2.
		\]
\end{proposition}
%\begin{lemma}
%	Let $d,N\in\N$ and $n\in\N_{\geq 2}$. Then, there are  constants $c=c(d,n)>0$ and $C=C(d,n)>0$ with the following property:
%
%	L}et $0 <\eps\leq cN^{-(n-1)}$ and $\Acal=\Acal(\eps,d,n)$ be a neural network architecture such that for any $f\in\Fndone$ there is a neural network $\Phi_{\eps,f}$ with $\Acal(\Phi_{\eps,f})=\Acal$ and 
%	\[
%		\norm*{\act{\Phi_{\eps,f}}-f}_{\Wkp[1][\infty][\cube^d]}\leq \eps.
%		\]
%		Then there exists a neural network architecture $\Acal'$ with $M(\Acal')=M(\Acal)+1$ such that $\vcdim(H)\geq N^d$, where
%		\[
%			H:=\left\{\chi_{[0,\infty)}\circ \Phi:\Phi\text{ is a neural network with } \Acal(\Phi)=\Acal'\right\}.
%			\]
%\end{lemma}

\begin{proof}
We prove the proposition by showing that there exists a function $h:\R^{M(\Arch_\eps)+1}\times [0,1]^d\to\{0,1\}$ with
	\[
		N^d\leq \vcdim\left(\left\{x\mapsto h(w,x): w\in\R^{M(\Arch_\eps)+1}\right\}\right)\leq C\cdot M(\Arch_\eps)^2.
		\]
		To simplify the notation, we set 
		\[
			H:=\left\{x\mapsto h(w,x): w\in\R^{M(\Arch_\eps)+1}\right\}.
			\]

\textbf{Step 1 (Construction of $h$):} Let now $0<\eps<c_1 N^{-(n-1)}/(3\sqrt d)$ for some constant $c_1=c_1(n)>0$ to be chosen later, and $\Arch_\eps=\Arch_\eps(d,n,\eps)$ be a neural network architecture as in the claim of the proposition. 
	
	For $x\in [0,1]^d$ we define a direction $\nu(x)\in\R^d$ that points from $x$ into $\cube^d$ if $x\in [0,1]^d\setminus (0,1)^d$ and equals $e_1$ if $x\in (0,1)^d$. We set
	\[
	\tilde \nu(x):=\begin{cases}
			e_1, & \text{if }0<x_k<1\text{ for }k=1,\ldots,d\\
			\bmat{c}{\chi_{\{0\}}(x_k)-\chi_{\{1\}}(x_k): k=1,\ldots,d},& \text{else,}
	\end{cases}
		\]
		and define $\nu(x):=\tilde \nu(x)/\pabs{\tilde \nu(x)}$. Moreover, we set $\tilde x:=x+\nu(x)/(4N)$ for $x\in [0,1]^d$.

		To construct $h$ we start be defining a function $g:\R^{M(\Arch_\eps)+1}\times [0,1]^d\to\R$ and then, to get a binary valued function, define $h$ by thresholding $g$. In detail, we set
		\[
			g((w,\delta),x):=\begin{cases}
				\frac{1}{\delta}\cdot\Big(\act{\Arch_\eps(w)}(\tilde x-\delta \nu(x))-\act{\Arch_\eps(w)}(\tilde x)\Big), & \text{if }\delta\neq 0\\
				0,& \text{if }\delta=0
			\end{cases}
			\]
			for $w\in\R^{M(\Arch_\eps)}$, $\delta\in\R$ and $x\in [0,1]^d$. Now, we define $h:\R^{M(\Arch_\eps)+1}\times  [0,1]^d\to\{0,1\}$ by
			\[
				h((w,\delta),x):=\begin{cases}
					1, &\text{if }g((w,\delta),x)>cN^{-(n-1)}/2\\
					0, &\text{else}
				\end{cases}
\]
for $w\in\R^{M(\Arch_\eps)}$, $\delta\in\R$ and $x\in [0,1]^d$.

	\textbf{Claim 1 ($N^d\leq \vcdim(H)$):} Let $x_1,\ldots,x_{N^d}\in [0,1]^d$ such that $\pabs{x_m-x_n}\geq 1/N$ for all $m,n=1,\ldots,N^d$ with $m\neq n$ and such that $\tilde x_m\in\cube^d$ for $m=1,\ldots,N^d$. Moreover, let $y_1,\ldots,y_{N^d}\in\{0,1\}$ be arbitrary. We aim to construct construct $w_y\in\R^{M(\Arch_\eps)}$ and $\delta_y\in\R$ with 
	\[
		h((w_y,\delta_y),x_m)=y_m\quad\text{for}\quad m=1,\ldots,N^d.
		\] To this end, we first define a function $f_y\in\Fnd$ with $f_y(x_m)=y_m\cdot a$ for some constant $a>0$, and then make use of a neural network $\Phi^{f_y}$ that approximates $f_y$.

	\textbf{Step 2 (Construction of $f_y$):} We start by defining a bump function $\psi\in C^\infty(\R^d)$ by 
		\begin{equation*}
			\psi(x):=\begin{cases}
				e^{-\left(1-4\pabs{x}^2\right)^{-1}+1}, & \text{if }\pabs{x}<\nicefrac{1}{2},\\
					0,&\text{else,}
			\end{cases}
		\end{equation*}
		such that $\psi(0)=1$ and $\supp \psi\subset B_{1/2,\pabs{\cdot}}(0)$. For the derivative $D\psi$ of $\psi$ it holds that there exists a function $\phi:(-\nicefrac{1}{2},\nicefrac{1}{2})\to\R_{> 0}$ such that\footnote{Precisely, $\phi(r)=e^{-\left(1-4r^2\right)^{-1}+1}\cdot\frac{8}{(1-4r^2)^2}$.}
		\[
			(D\psi)(x)=\phi(\pabs{x})\cdot (-x)\quad\text{for all }x\text{ with }\pabs{x}<\nicefrac{1}{2}.
			\]
	Thus, if $x\in\R^d$ with $\pabs{x}<1/2$, then we have for the derivative of $\psi$ in direction $-x/\pabs{x}$ at $x$ that $(D_{-\nicefrac{x}{\pabs{x}}}\psi)(x)=\phi(\pabs{x})\pabs{x}>0$ only depends on the norm of $x$.

			Next, we define $f_y\in C^\infty(\R^d)$ by 
			\[
			f_y(x):=\sum_{m=1}^{N^d} y_m \frac{BN^{-n}}{\norm{\psi}_{\Wkp[n][\infty][\cube^d]}}\psi\left(N(x-x_m)\right)	
			\]
			for $x\in\R^d$. We have $\pabs{f}_{\Wkp[k][\infty][\cube^d]}\leq B N^{-n}N^{k}\leq B$ for $1 \leq k\leq n$ and, consequently, $f\in\Fnd$. Furthermore, for $x\in\R^d$ with $\pabs{x}<1/(2N)$ it holds that
			\[
				(D_{-\nicefrac{x}{\pabs{x}}} f)(x_m+x)=y_m \frac{BN^{-n}}{\norm{\psi}_{\Wkp[n][\infty][\cube^d]}}\phi(\pabs{Nx})N^2\pabs{x}
				\]
				and, in particular, if $\pabs{x}=1/(4N)$, then 
				\begin{equation}\label{eq:lower_directional_derivative}
					(D_{-\nicefrac{x}{\pabs{x}}} f)(x_m+x)=y_m\phi(1/4)\cdot\frac{BN^{-(n-1)}}{4\norm{\psi}_{\Wkp[n][\infty][\cube^d]}}=y_m c_1 N^{-(n-1)}.
				\end{equation}
					Here, we defined the constant $c_1=c_1(n,B)>0$ which was left unspecified in the beginning of the proof by $c_1:=\frac{B\phi(1/4)}{4\norm{\psi}_{\Wkp[n][\infty][\cube^d]}}$.
					
			 \textbf{Step 3 (Existence of $w_y$ and $\delta_y$):} We can find a vector $w_y\in\R^{M(\Arch_\eps)}$ such that for the neural network $\Phi_\eps^{f_y}:=\Arch_\eps(w_y)$ Equation~\eqref{eq:approximation_lower_bound} holds (with ${f_y}$ instead of $f$). In particular, we have 
					\begin{equation}\label{eq:lower_approximation_fy_semi_norm}
						\pabs{\act{\Phi_\eps^{f_y}}-f_y}_{\Wkp[1][\infty][\cube^d]}\leq \eps.
					\end{equation}

Next, for $x\in \cube^d$ and $\nu\in\R^d$ we get from~Lemma~\ref{lemma:lower_bounds_affine_set} that there exists an open set $T_{x,\nu}\subset\R^d$ and $\delta=\delta(x,\nu,T_{x,\nu})>0$ with $x+\lambda\delta\nu\in \overline{T_{x,\nu}}$ for $0\leq \lambda\leq 1$ and $\act{\Phi_\eps^{f_y}}$ is affine-linear on $T_{x,\nu}$. We define
		\[
		\delta_y:=\min_{m=1,\ldots N^d} \delta\left(\tilde x_m,-\nu(x_m), T_{\tilde x_m,-\nu(x_m)}\right)>0.
			\]
			%\Red{$\delta$ mus kleiner $1/N$ oder noch kleiner!}--> Nein, es ist nicht schlimm, wenn man dann aus dem Gebiet rausgeht fuer den zweiten Punkt.

			Let $m\in\{1,\ldots,N^d\}$ and let $F_m:\R^d\to \R$ be an affine-linear function such that $\act{\Phi_\eps^{f_y}}(x)= F_m(x)$ for all $x \in T_{\tilde x_m,-\nu(x_m)}$. It then follows from the continuity of $\act{\Phi_\eps^{f_y}}$ that $\act{\Phi_\eps^{f_y}}(x) = F_m(x)$ for all $x\in \overline{T}_{\tilde x_m,-\nu(x_m)}$. This, together with the choice of $\delta_y$ implies that 
			\[
				g((w_y,\delta_y),x_m)=D_{-\nu(x_m)}F_m(\tilde x_m)=D_{-\nu(x_m)}F_m.
				\] 
			Recall that $T_{\tilde x_m,-\nu(x_m)}$ is open and $f_y$ and $\act{\Phi_\eps^{f_y}}$ are continuously differentiable on $T_{\tilde x_m,-\nu(x_m)}$. Thus, the weak and strong derivative agree and Equation~\eqref{eq:lower_approximation_fy_semi_norm} implies 
			\[
				\pabs{D^i f_y(x)- D^i F_m}=\pabs{D^i f_y(x)- D^i\act{\Phi_\eps^{f_y}}(x)}\leq \eps
				\]
				for all $x\in T_{\tilde x_m,-\nu(x_m)}$ and $i=1,\ldots,d$. Using the continuity of $D^i f_y$ we get that $\pabs{D^i f_y(\tilde x_m)- D^i F_m}\leq\eps$ for $i=1,\ldots,d$ and hence 
				\begin{equation}\label{eq:approximation_in_tildexm}
					\pabs{D_\nu f_y(\tilde x_m)-D_\nu F_m}\leq\sqrt d \eps\pabs{\nu}\quad\text{for}\quad \nu\in\R^d.
				\end{equation} An addition of zero yields
			\begin{equation}\label{eq:lower_addition_zero}
				g((w_y,\delta_y),x_m)=D_{-\nu(x_m)} F_m=D_{-\nu(x_m)} f_y(\tilde x_m)+D_{-\nu(x_m)} F_m - D_{-\nu(x_m)} f_y(\tilde x_m).
			\end{equation}

			We get for the case $y_m=1$ that
				\begin{align}\label{eq:lower_y1}
					g((w_y,\delta_y),x_m)\geq D_{-\nu(x_m)} f_y(\tilde x_m) -\sqrt d \eps \geq c_1 N^{-(n-1)}-c_1 N^{-(n-1)}/3,
				\end{align}
				where we used Equation~\eqref{eq:lower_addition_zero} together with Equation~\eqref{eq:approximation_in_tildexm} and $\pabs{\nu(x_m)}=1$ for the first step and Equation~\eqref{eq:lower_directional_derivative} together with the upper bound for $\eps$ for the second step.

				 In a similar way, we get for the case $y_m=0$ that
				\begin{equation}\label{eq:lower_y0}
					g((w_y,\delta_y),x_m)\leq \sqrt d \eps\leq  c_1 N^{-(n-1)}/3,
				\end{equation}
where we used that $D_{-\nu(x_m)} f_y(\tilde x_m)=0$.

Finally, combining Equation~\eqref{eq:lower_y1} and~Equation~\eqref{eq:lower_y0} reads as
	\[
		g((w_y,\delta_y),x_m)\begin{cases}
		> c_1 N^{-(n-1)}/2, &\text{if }y_m=1,\\
		< c_1 N^{-(n-1)}/2, &\text{if }y_m=0,
	\end{cases}
		\]
		which proves Claim~1.

			\textbf{Claim 2 ($\vcdim(H)\leq C\cdot M(\Arch_\eps)^2$):} We start by showing that there exists a constant $C'=C'(d)$ such that $h((w,\delta),x)$ can be computed using $C'\cdot M(\Arch_\eps)$ operations of the following types 
		\begin{itemize}
			\item the arithmetic operations $+,-,\times$, and $/$ on real numbers,
			\item jumps conditioned on $>,\geq,<,\leq,=$, and $\neq$ comparisons of real numbers
		\end{itemize}
		for all $w\in\R^{M(\Arch_\eps)+1}$ and $x\in [0,1]^d$.

		There exists an absolute constant $C_1>0$ such that at most $C_1\cdot d$ operations of the specified type are needed to compute $\nu(x)$. Hence, the same holds true for $\tilde x$ and $\tilde x - \delta \nu(x)$.
			
		Note that the number of neurons that are needed for the computation of $\act{\Arch_\eps(w)}$ can be bounded by $M(\Arch_\eps)$.	Thus, $\act{\Arch_\eps(w)}$ can be computed using at most $C_2\cdot M(\Arch_\eps)$ operations where $C_2>0$ is an absolute constant. Hence, there exists a constant $C_3=C_3(d)$ such that for the number of operations of the specified type $t$ needed for the computation of $h(x)$ where $x\in [0,1]^d$ it holds that $t\leq C_3 M(\Arch_\eps)$.

 Finally, \cite[Theorem~8.4]{anthony2009neural} implies 
		\[
			\vcdim\left(\left\{x\mapsto h(w,x): w\in\R^{M(\Arch_\eps)+1}\right\}\right)\leq 4 (M(\Arch_\eps)+1)(C_3\cdot M(\Arch_\eps)+2)\leq C_4 M(\Arch_\eps)^2,
			\]
			where $C_4=C_4(d)>0$ is a suitable constant.
\end{proof}

The proof of the lower complexity bounds is now a simple consequence of Proposition~\ref{prop:lower_bounds}.
\begin{proof}[Proof of Theorem {\rm\ref{thm:lower_bound}}]
	The case $k=0$ corresponds to~\cite[Theorem~4~a)]{yarotsky2017error}.

	For the case $k=1$, let $c=c(n,B)$ and $C=C(d)$ be the constants from Proposition~\ref{prop:lower_bounds} and set 
	\[
		N:=\floor*{\left(\frac{c}{\eps}\right)^{1/(n-1)}}.
		\]
	Then, $N\leq (c/\eps)^{1/(n-1)}$ and, thus, $0<\eps \leq c N^{-(n-1)}$. Now, Proposition~\ref{prop:lower_bounds} implies that 			\begin{equation}\label{eq:lower_Nd}
	N^d\leq  C M(\Arch_\eps)^2.
	\end{equation}
	We also have $(c/\eps)^{1/(n-1)}\leq 2N$ and hence $c_1 \eps^{-d/(n-1)}\leq N^d$ for a suitable constant $c_1=c_1(n,B)>0$. Combining this estimate with Equation~\ref{eq:lower_Nd} yields
	\[
		c_1 \eps^{-d/(n-1)}\leq N^d\leq C M(\Arch_\eps)^2,
		\]
		which finally results in
		\[
			C'\eps^{\nicefrac{-d}{2(n-1)}}\leq M(\Arch_\eps) 
			\]
		for a constant $C'=C'(d,n,B)>0$.	
\end{proof}

\bibliography{literature}
\bibliographystyle{abbrv}

\end{document}